\documentclass{article}

\usepackage{latexsym}
\usepackage{a4wide}
\usepackage{amscd}
\usepackage{graphics}
\usepackage{amsmath}
\usepackage{amssymb}
\usepackage{mathrsfs}
\input xy
\xyoption{all}

\newcommand{\N}{\mathbb{N}}                     
\newcommand{\Z}{\mathbb{Z}}                     
\newcommand{\R}{\mathbb{R}}                     
\newcommand{\C}{\mathbb{C}}                     
\newcommand{\T}{\mathbb{T}}                     
\newcommand{\set}[2]{\left\{{#1}\mid{#2}\right\}}       
\newcommand{\qed}{\hfill $\Box$ \bigskip}       
\newcommand{\im}{\mathrm{Im\,}}                 
\newcommand{\dist}{\mathrm{dist\,}}             
\newcommand{\Ker}{\mathrm{Ker\,}}               
\newcommand{\sgn}{\mathrm{sgn\,}}               
\newcommand{\diam}{\mathrm{diam\,}}             
\newcommand{\ind}{\mathrm{ind\,}}               
\newcommand{\codim}{\mathrm{codim}}           
\newcommand{\graf}{\mathrm{graph\,}}            
\newcommand{\crit}{\mathrm{crit}}               
\newcommand{\sing}{\mathrm{sing}\,}             
\newcommand{\tr}{\mathrm{tr}\,}                         

\newtheorem{mainthm}{\bf Theorem}           
\newtheorem{thm}{\bf Theorem}[section]      
\newtheorem{lem}[thm]{\bf Lemma}            
\newtheorem{prop}[thm]{\bf Proposition}     
\newtheorem{rem}[thm]{\bf Remark}           

\title{Estimates and computations in Rabinowitz-Floer homology}
\author{Alberto Abbondandolo\thanks{The first author was supported by a Humboldt Research Fellowship for Experienced Researchers.} \hspace{2pt}
and Matthias Schwarz\thanks{The second author was partially supported by the DFG grant SCHW 892/2-3. \newline Mathematics Subject Classification 53D40, 57R58, 37J45, 37J50.}}

\date{December 6, 2009}

\begin{document}

\maketitle

\begin{abstract}
The Rabinowitz-Floer homology of a Liouville domain $W$ is the Floer homology of the Rabinowitz free period Hamiltonian action functional associated to a Hamiltonian whose zero energy level is the boundary of $W$. This invariant has been introduced by K.\ Cieliebak and U.\ Frauenfelder and has already found several applications in symplectic topology and in Hamiltonian dynamics. Together with A.\ Oancea, the same authors have recently computed the Rabinowitz-Floer homology of the cotangent disk bundle $D^*M$ of a closed Riemannian manifold $M$, by means of an exact sequence relating the Rabinowitz-Floer homology of $D^*M$ with its symplectic homology and cohomology. The first aim of this paper is to present a chain level construction of this exact sequence. In fact, we show that this sequence is the long homology sequence induced by a short exact sequence of chain complexes, which involves the Morse chain complex and the Morse differential complex of the energy functional for closed geodesics on $M$. These chain maps are defined by considering spaces of solutions of the Rabinowitz-Floer equation on half-cylinders, with suitable boundary conditions which couple them with the negative gradient flow of the geodesic energy functional.
The second aim is to generalize this construction to the case of a fiberwise uniformly convex compact subset $W$ of $T^*M$ whose interior part contains a Lagrangian graph. Equivalently, $W$ is the energy sublevel associated to an arbitrary Tonelli Lagrangian $L$ on $TM$ and to any energy level which is larger than the strict Ma\~n\'e critical value of $L$. In this case, the energy functional for closed geodesics is replaced by the free period Lagrangian action functional associated to a suitable calibration of $L$. An important issue in our analysis is to extend the uniform estimates for the solutions of the Rabinowitz-Floer equation - both on cylinders and on half-cylinders - to Hamiltonians which have quadratic growth in the momenta. These uniform estimates are obtained by the Aleksandrov integral version of the maximum principle. In the case of half-cylinders, they are obtained by an Aleksandrov-type maximum principle with Neumann conditions on part of the boundary. 
\end{abstract}

\renewcommand{\theenumi}{\alph{enumi}}
\renewcommand{\labelenumi}{(\theenumi)}

\tableofcontents

\section*{Introduction}
\addcontentsline{toc}{section}{\numberline{}Introduction}

Let $(W,\lambda)$ be a Liouville domain, that is a compact connected $2n$-dimensional manifold with boundary, equipped with a one-form $\lambda$ - the Liouville form - such that $\omega:= d\lambda$ is a symplectic form on $W$, and that the Liouville vector field $Y$ defined by $\iota_Y \omega = \lambda$ is transverse to $\partial W$ and points in the outward direction (we borrow the terminology {\em Liouville domain} from Seidel, see \cite{sei06b}). Recently, K.\ Cieliebak and U.\ Frauenfelder \cite{cf09} have associated an algebraic invariant to such a structure, called the {\em Rabinowitz-Floer homology} of $(W,\lambda)$.

This is the homology of a chain complex, whose generators are either constant loops in $\partial W$ or reparametrized closed orbits of the Reeb vector field  $R$ on $\partial W$ induced by the contact form $\alpha:=\lambda|_{\partial W}$. In the construction, one considers the completion $\hat{W}$ of $W$, that is the symplectic manifold obtained by attaching to $W$ the collar $\partial W \times [1,+\infty[$ and by extending the Liouville form to $\hat{W}$ by setting
$\lambda := \rho \alpha$ on the collar, where $\rho\in [1,+\infty[$ denotes the second component in $\partial W \times [1,+\infty[$. One chooses a Hamiltonian $H\in C^{\infty} (\hat{W})$ whose zero level set is $\partial W$, which is negative on the interior part of $W$, and whose Hamiltonian vector field restricts to the Reeb vector field $R$ on $\partial W$. Then one looks at the Hamiltonian action functional on the space of loops $y:\R/\eta \Z \rightarrow \hat{W}$ of arbitrary period $\eta>0$:
\[
y \mapsto \int_{\R/\eta\Z} y^* \lambda - \int_{\R/\eta \Z} H(y(\tau))\, d\tau.
\]
If we reparametrize each of such loops $y$ on the circle of unit length $\T:=\R/\Z$ by setting $x(t) = y(\eta t)$, we obtain the {\em Rabinowitz action functional}
\[
\mathbb{A}(x,\eta) := \int_{\T} x^* \lambda - \eta \int_{\T} H(x(t)) \, dt,
\]
which was used by P.\ H.\ Rabinowitz \cite{rab78} in his proof of the Weinstein conjecture for star-shaped domains of $\R^{2n}$. Although we started with a positive $\eta$, the functional $\mathbb{A}$ makes sense for all real values of $\eta$, and we can see it as a functional on $\Lambda \hat{W} \times \R$, where $\Lambda \hat{W}$ is the space of smooth closed loops $x:\T \rightarrow \hat{W}$. Its critical points
 are either of the form $(x,0)$, with $x$ a constant loop in $\partial W$, or $(x,\eta)$, with $\eta\neq 0$ and $y(\tau) := x(\tau/\eta)$ a closed Reeb orbit on $\partial W$ with (not necessarily minimal) period $|\eta|$.  
If one endows $\Lambda \hat{W}$ with the $L^2$-Riemannian structure induced by 
an $\omega$-compatible $t$-dependent almost complex structure $J_t$, $t\in \T$, on $\hat{W}$, the {\em negative gradient equation} for $\mathbb{A}$ has the form
\begin{equation}
\label{Fleq}
\begin{split}
& \frac{\partial u}{\partial s} (s,t) + J_t(u(s,t)) \left( \frac{\partial u}{\partial t}(s,t) - \eta(s) X_H(u(s,t)) \right) = 0, \\
& \frac{d\eta}{ds} (s) = \int_{\T} H(u(s,t))\, dt, 
\end{split} \end{equation}      
where $u$ is a map from the cylinder $\R\times \T$, endowed with coordinates $(s,t)$, to $\hat{W}$, $\eta$ is a real function on $\R$, and $X_H$ denotes the Hamiltonian vector field on $(\hat{W},\omega)$ induced by $H$. We refer to the system of PDE's  (\ref{Fleq}) as to the {\em Rabinowitz-Floer equation}. The Rabinowitz-Floer complex of $(W,\lambda)$ is then defined by counting solutions of the equation (\ref{Fleq}) which connect different critical points of $\mathbb{A}$. Actually, the critical points of $\mathbb{A}$ are never isolated, so one needs to work in a Morse-Bott setup: Up to a small perturbation of $\partial W$ within $\hat{W}$, the critical set $\crit\, \mathbb{A}$ of the Rabinowitz action functional consists of smooth finite dimensional manifolds, so one fixes an auxiliary Morse function $a$ on $\crit \, \mathbb{A}$, and considers the set of {\em negative gradient flow lines with cascades} for $(\mathbb{A},a)$ (see the Appendix of \cite{fra04} and Section \ref{rfcsec} below). The $\Z_2$-vector space $RF$ is defined as the space of formal sums of critical points of $a$, which are possibly infinite but on which $\mathbb{A}$ is bounded from above. A counting process on the space of gradient flow lines with cascades for $(\mathbb{A},a)$ produces a homomorphism $\partial : RF \rightarrow RF$
such that $\partial^2=0$. When the transverse Conley-Zehnder index of the 
Reeb orbits of $R$ is well-defined, $RF$ carries a $\Z$-grading, so $(RF_*,\partial)$ is a chain complex of $\Z_2$-vector spaces, which is called the {\em Rabinowitz-Floer complex of $(W,\lambda,a)$} and is denoted by $RF(W,\lambda,a)$. Its homology does not depend on the choice of the auxiliary data and is called the {\em Rabinowitz-Floer homology of $(W,\lambda)$}. Rabinowitz-Floer homology has the following important vanishing property: $HRF(W,\lambda) =0$ whenever there is an embedding $\varphi: W \hookrightarrow W'$ into the interior part of another Liouville domain $(W',\lambda')$, such that $\varphi^* \lambda' - \lambda$ is exact and $\varphi(W)$ is displaceable within $W'$ by a Hamiltonian isotopy (see Theorem 1.2 in \cite{cf09}; in order to prove this theorem and other invariance results, it is useful to define the Rabinowitz-Floer homology of $(W,\lambda)$ by using more general ambient manifolds than the completion $\hat{W}$; the fact that the resulting homology does not depend on the choice of the ambient manifold is proved in \cite{cfo09}, Proposition 3.1). The fact that the Rabinowitz-Floer homology of $(W,\lambda)$ vanishes implies the existence of closed Reeb orbits on $\partial W$, because otherwise $HRF(W,\lambda)$ would be isomorphic to the singular homology of $\partial W$ (see Corollary 3.3 in \cite{sch06} for a proof of this fact under weaker assumptions).

Rabinowitz-Floer homology has already found quite a number of other applications. It provides obstructions for the existence of contact embeddings of unit cotangent sphere bundles (see \cite{cf09} and \cite{cfo09}), it allows to prove existence and multiplicity results for leaf-wise intersections points (see \cite{af08a} and \cite{af08b}), and it plays a relevant role in the study of the
 dynamics and the symplectic topology of energy hypersurfaces of magnetic flows on cotangent bundles (see \cite{cfp09}, where Rabinowitz-Floer homology is extended to domains whose boundary need not be of contact type, but satisfies substantially weaker assumptions). An important ingredient in all these applications is to determine the Rabinowitz-Floer homology of $(D^*M,\lambda)$, the unit cotangent disk bundle of a closed Riemannian manifold $(M,g)$, equipped with the restriction of the canonical Liouville one-form of the cotangent bundle $T^*M$. In this case, the completion of $D^*M$ is $T^*M$, its boundary is $S^*M$, the unit cotangent sphere bundle, and the Reeb flow on $S^*M$ coincides with the geodesic flow.

When $M$ is the sphere, $HRF(D^*M,\lambda)$ was computed in \cite{cf09}, while in the case of an arbitrary closed manifold it has been determined by K.\ Cieliebak, U.\ Frauenfelder, and A.\ Oancea in \cite{cfo09}. In the latter paper, the authors show that the Rabinowitz-Floer homology groups fit into an exact sequence of the form
\begin{equation}
\label{exaseq1}
\dots \rightarrow SH_k(W,\lambda) \rightarrow HRF_k(W,\lambda) \rightarrow SH^{1-k}(W,\lambda) \rightarrow SH_{k-1}(W,\lambda) \rightarrow \dots
\end{equation}
where $SH_*(W,\lambda)$ and $SH^*(W,\lambda)$ are the symplectic homology and cohomology groups of $(W,\lambda)$ (see \cite{fh94}, \cite{cfh95}, \cite{vit99}, \cite{oan04}, and \cite{sei06b}). This exact sequence is built by introducing a version of symplectic homology for $V$-shaped Hamiltonians, which naturally fits into an exact sequence of the above form, and by proving that such a homology is isomorphic to Rabinowitz-Floer homology. The latter isomorphism is the deep part of the proof in \cite{cfo09}. In the case $W=D^*M$, using the fact that $SH_*(D^*M,\lambda)$ and $SH^*(D^*M,\lambda)$ are isomorphic to the singular homology and cohomology of $\Lambda M$, the free loop space of $M$ (as proved by Viterbo in \cite{vit03}), the exact sequence (\ref{exaseq1}) takes the form
\begin{equation}
\label{exaseq2}
\dots
\rightarrow H_k(\Lambda M) \rightarrow HRF_k(D^*M ,\lambda) \rightarrow H^{1-k}(\Lambda M) \rightarrow H_{k-1}(\Lambda M) \rightarrow \dots
\end{equation}
Moreover, the homomorphism
\[
H^{1-k}(\Lambda M) \rightarrow H_{k-1}(\Lambda M)
\]
is clearly zero, except possibly for $k=1$, where it is shown to be the multiplication by $\chi(T^*M)$, the Euler number of the vector bundle $T^*M \rightarrow M$, on the space of contractible loops, and to vanish on all the other components of $\Lambda M$. These facts allow to completely determine $HRF(D^*M,\lambda)$ (see Theorem 1.10 in \cite{cfo09} and Section \ref{exasec} below).

\medskip

\noindent{\bf The first aim of this paper} is to present an alternative, direct construction of the exact sequence (\ref{exaseq2}) (hence, an alternative way of determining $HRF(D^*M,\lambda)$). The starting idea is that (\ref{exaseq2}) should be thought as the long exact homology sequence induced by a short exact sequence of chain complexes of the form
\[
0 \rightarrow C_k(\Lambda M) \rightarrow RF_k(D^*M,\lambda,a) \rightarrow C^{1-k}(\Lambda M) \rightarrow 0.
\]
Here $C_*(\Lambda M)$ and $C^*(\Lambda M)$ are a chain complex and a differential complex inducing the homology and the cohomology of $\Lambda M$, for instance the complexes associated to a cellular filtration of $\Lambda M$. 
These complexes can also be obtained from the classical infinite dimensional Morse theory (see \cite{pal63}, \cite{kli82}, and \cite{ama06m}) of the 
geodesic energy functional   
\[
\mathbb{E}(\gamma) = \int_0^1 g(\gamma'(t),\gamma'(t))\, dt
\]
on the space $W^{1,2}(\T,M)$ of absolutely continuous closed loops in $M$ with square integrable velocity. For a generic Riemannian metric $g$, $\mathbb{E}$ is a Morse-Bott functional on $W^{1,2}(\T,M)$, and one can associate to it a {\em Morse chain complex} $(M_*(\mathbb{E},e),\partial)$ and a {\em Morse differential complex} $(M^*(\mathbb{E},f),\delta)$ by looking at the flow lines with cascades produced by the $W^{1,2}$-negative gradient flow of $\mathbb{E}$, together with the negative gradient flow of auxiliary Morse functions $e$ and $f$ on $\crit\, \mathbb{E}$. Here, it is actually useful to choose $e$ to be suitably compatible with the auxiliary Morse function $a$ on $\crit \, \mathbb{A}$,  and $f$ to coincide with $-e$ (see conditions (A0-A4) in Section \ref{criti} below). The homology of $M_*(\mathbb{E},e)$ is isomorphic to the singular homology of $\Lambda M$, while the cohomology of $M^*(\mathbb{E},-e)$ is isomorphic to the singular cohomology of $\Lambda M$. Therefore, our aim is to construct a short exact sequence of chain complexes of the form 
\begin{equation}
\label{exa0}
0 \rightarrow M_k(\mathbb{E},e) \rightarrow RF_k(D^*M,\lambda,a) \rightarrow 
M^{1-k} (\mathbb{E},-e) \rightarrow 0.
\end{equation}
In the construction of the Rabinowitz-Floer complex, it is natural to consider the Hamiltonian 
\begin{equation}
\label{laham}
H(q,p) = \frac{1}{2} ( g^*_q(p,p) - 1), \quad \forall q\in M, \; p\in T_q^* M,
\end{equation}
where $g^*$ is the inner product on $T^*M$ dual to $g$, so that the flow of $X_H$ is precisely the geodesic flow on $T^*M$. The necessity to work with the quadratic Hamiltonian $H$, instead of with Hamiltonians which are constant outside a compact set, requires us to generalize the $L^{\infty}$ estimates for the Rabinowitz-Floer equation (\ref{Fleq}) proved in  \cite{cf09}. We discuss this point at the end of this Introduction. 
 
The construction of the chain maps appearing in (\ref{exa0}) is based on counting solutions of a mixed problem and is similar to the construction that we introduced in \cite{as06}. More precisely, the definition of the chain map 
\[
\Phi : M_* (\mathbb{E},e) \rightarrow RF_*(D^*M,\lambda,a)
\]
is based on counting solutions of the following mixed problem: Given $\gamma$ a critical point of $\mathbb{E}$ and $z$ a critical point of $\mathbb{A}$, we consider the space of pairs
\[
u : [0,+\infty[ \times \T \rightarrow T^*M, \quad \eta :[0,+\infty[ \rightarrow \R,
\]
which solve the Rabinowitz-Floer equation (\ref{Fleq}), converge to $z$ for $s\rightarrow +\infty$, and satisfy the boundary conditions
\begin{equation}
\label{cpl}
\pi\circ u(0,\cdot) \in W^u(\gamma;-\nabla \mathbb{E}), \quad \eta(0) = \sqrt{ \mathbb{E}(\pi\circ u(0,\cdot))},
\end{equation}
where $\pi : T^*M \rightarrow M$ is the projection and $W^u(\gamma;-\nabla \mathbb{E})\subset W^{1,2}(\T,M)$ denotes the unstable manifold of $\gamma$ with respect to the $W^{1,2}$-negative gradient flow of $\mathbb{E}$. Actually, the Morse-Bott situation requires us to consider solutions with cascades of the above problem, see Section \ref{phisec} below for a precise definition. In order to keep the presentation simpler, we systematically ignore this point within this Introduction.
The second coupling condition in (\ref{cpl}) is suggested by the inequality
\begin{equation}
\label{actin1}
\mathbb{A}(x,\sqrt{\mathbb{E}(\pi\circ x)}) \leq \sqrt{\mathbb{E} (\pi\circ x)}, \quad \forall x\in \Lambda T^*M,
\end{equation}
and by the fact that equality holds when $(x,\sqrt{\mathbb{E}(\pi\circ x)})$ is a critical point of the Rabinowitz action functional (see Lemma \ref{levrel} below). Indeed, the latter inequality allows us to prove the necessary compactness for the above mixed problem, and to construct a left inverse $\hat{\Phi}$ for $\Phi$.

Similarly, the chain map 
\[
\Psi : RF_*(D^*M,\lambda,a) \rightarrow M^{1-*}(\mathbb{E},-e)
\]
is built by considering solutions
\[
u:\;  ]-\infty,0] \times \T \rightarrow T^*M, \quad \eta :\; ]-\infty,0] \rightarrow \R, 
\]
of the Rabinowitz-Floer equation (\ref{Fleq}), which converge to a critical point $z$ of $\mathbb{A}$ for $s\rightarrow -\infty$, and satisfy the boundary conditions
\[
\pi\circ u(0,-\cdot) \in W^u(\gamma;-\nabla \mathbb{E}), \quad \eta(0) = - \sqrt{ \mathbb{E}(\pi\circ u(0,\cdot))}.
\] 
Note that the reversed parametrization of the loop $u(0,\cdot)$ is considered here. Then we use the inequality
\begin{equation}
\label{actin2}
\mathbb{A}(x,-\sqrt{\mathbb{E}(\pi\circ x)}) \geq - \sqrt{\mathbb{E} (\pi\circ x)}, \quad \forall x\in \Lambda T^*M,
\end{equation}
and we construct a right inverse $\hat{\Psi}$ for $\Psi$, which is then surjective.

The composition $\Psi\circ \Phi$ might not be zero, but we can show that it is chain homotopic to zero, that is
\[
\Psi\circ \Phi = P \partial + \delta P,
\]
where the chain homotopy
\[
P : M_*(\mathbb{E},e) \rightarrow M^{-*} (\mathbb{E},-e)
\]
is defined by counting finite length cylinders over minimal closed geodesics, that is solutions
\[
u : [-S,S] \times \T \rightarrow T^*M, \quad \eta : [-S,S] \rightarrow \R,
\]
of the Rabinowitz-Floer equation (\ref{Fleq}) for some $S>0$, which satisfy the boundary conditions
\[
\pi\circ u(-S,\cdot) = \gamma^-, \quad \eta(-S) = \sqrt{ \mathbb{E}(\gamma^-)}, \quad \pi\circ u(S,\cdot) = \gamma^+(-\cdot), \quad \eta(S) = -\sqrt{ \mathbb{E}(\gamma^+)},
\] 
where $\gamma^-$ and $\gamma^+$ are closed geodesics of Morse index zero. The main point in the construction of $P$ and in the proof of the fact that it is a chain homotopy between $\Psi\circ \Phi$ and zero is to show that the length $2S$ of  the cylinders in the above problem -- and in the analogous problem in which the sum of the indices of the closed geodesics $\gamma^-$ and $\gamma^+$ is 1 -- is bounded away from zero (see Lemma \ref{X} below). 
Then $\Phi$ can be modified within its chain homotopy class, obtaining the chain map
\[
\Theta := \Phi - \hat{\Psi} P \partial - \partial \hat{\Psi} P,
\]
and our first main result is the following:

\begin{mainthm}
\label{main1}
The short sequence of chain complexes
\[
0 \rightarrow M_*(\mathbb{E},e) \stackrel{\Theta}{\rightarrow} RF_*(D^*M,\lambda,a) \stackrel{\Psi}{\rightarrow} M^{1-*} (\mathbb{E},-e) \rightarrow 0
\]
is exact and has a canonical splitting. After identifying Morse (co)homology with singular (co)homology by the isomorphisms
\[
HM_k(\mathbb{E},e) \cong H_k (\Lambda M), \quad HM^k(\mathbb{E},-e) \cong H^k(\Lambda M),
\]
the associated long exact sequence takes the form
\[
\dots \rightarrow H_k(\Lambda M) \stackrel{\Theta_* = \Phi_*}{\longrightarrow} HRF_k(D^*M ,\lambda) \stackrel{\Psi_*}{\longrightarrow} H^{1-k}(\Lambda M) \stackrel{\Delta_*}{\longrightarrow} H_{k-1}(\Lambda M) \rightarrow \dots
\]
where the connecting homomorphism $\Delta_*$ can be non-zero only in degree zero, and where
\[
\Delta_* : H^0 (\Lambda M) \rightarrow H_0 (\Lambda M)
\]
vanishes on the components of $\Lambda M$ containing non-contractible loops, and is the multiplication by the Euler number $\chi(T^*M)$ on the component of $\Lambda M$ consisting of contractible loops.
\end{mainthm}

See Theorem \ref{main} below for a more precise statement. The long exact sequence of the theorem above is precisely the exact sequence (\ref{exaseq2}) from \cite{cfo09} (see the end of Section \ref{exasec} for comments about this).
When compared to \cite{cfo09}, our approach has advantages and disadvantages. An obvious limitation is that our construction of the maps $\Phi$ and $\Psi$ heavily uses the cotangent bundle structure, so we do not find the exact sequence (\ref{exaseq2}) as a particular case of (\ref{exaseq1}). On the other hand, we obtain quite natural maps already at the chain level. Moreover, although here - for sake of simplicity - we work with $\Z_2$ coefficients, our construction could be extended to arbitrary coefficient groups, in particular to the integers, and the canonical splitting would still exist. The approach in \cite{cfo09}, instead, makes extensive use of direct and inverse limits of chain complexes, and for this reason it needs a Mittag-Leffler condition to be fulfilled, which forces to choose coefficient in a field (see \cite{cf09b} for a discussion of this point).

\medskip

\noindent{\bf In our second result}, we treat the following class of domains: $W$ is a compact subset of $T^*M$, it has smooth fiberwise uniformly convex boundary, and its interior contains a Lagrangian graph, that is
\[
\set{(q,\theta(q))}{q\in M} \subset \mathrm{Int}\, (W),
\]
for some closed one-form $\theta$ on $M$. Such a $W$ is a Liouville domain with respect to the one-form
\[
\lambda_{\theta} := \lambda - \pi^* \theta,
\]
and the standard symplectic form $\omega=d\lambda=d\lambda_{\theta}$. Its completion is naturally identified with $(T^*M,\lambda_{\theta})$. Equivalently, such a $W$ can be obtained by starting from an arbitrary {\em Tonelli Lagrangian} $L\in C^{\infty}(TM)$ (that is, $L$ is superlinear and has positive fiberwise second differential) and by considering the set of phase space points whose energy does not exceed $\kappa$, where $\kappa$ is larger than the {\em strict Ma\~n\'e critical value} $c_0(L)$ (see \cite{pp97}, \cite{cipp98}, and Remark \ref{smcv} below).  

A standard homotopy argument could be used to prove that the Rabinowitz-Floer homology of $(W,\lambda_{\theta})$ is isomorphic to the Rabinowitz-Floer homology of $(D^*M,\lambda)$, but again we wish to work at the chain level, and to find connections with the Morse theory of functionals from Lagrangian classical mechanics. The fact that the interior part of $W$ contains a Lagrangian graph implies that the flow of the Reeb vector field $R$ on $(\partial W,\lambda_{\theta}|_{\partial W})$ is equivalent, up to time reparametrizations, to a Finsler geodesic flow (see \cite{cipp98}). Therefore, one could adapt the proof of Theorem \ref{main1} to the Finsler energy functional. Instead, we prefer to compare the Rabinowitz-Floer complex of $(W,\lambda_{\theta})$ to the Morse chain complex and differential complex of the {\em free period Lagrangian action functional}. We start by seeing $\partial W$ as the zero level set  of a suitably quadratic Hamiltonian $H\in C^{\infty}(T^*M)$, whose fiberwise second differential is positive, and we consider its Fenchel dual Lagrangian $L\in C^{\infty}(TM)$, defined by
\[
L(q,v) := \max_{p\in T_q^*M} \Bigl( \langle p,v \rangle - H(q,p) \Bigr).
\]
By using the Hamiltonian characterization of the {\em Ma\~n\'e critical value} $c(L)$, which was discovered by G.\ Contreras, R.\ Iturriaga, G.\ P.\ Paternain, and M.\ Paternain in \cite{cipp98},  we deduce that the  Ma\~n\'e critical value
\begin{eqnarray*}
c(L-\theta) := \inf \Bigl\{ \kappa \in \R\, \Big| \int_0^T \bigl( L(\gamma,\gamma') - \theta(\gamma)[\gamma'] + \kappa \bigr)\, d\tau \geq 0, \\ \forall \gamma : \R/T\Z \rightarrow M \mbox{ absolutely continuous, } \forall T>0 \Bigl\}
\end{eqnarray*}
is strictly negative. This fact implies that the  free period Lagrangian action functional associated to the Lagrangian $L-\theta$, which, after the usual reparametrization on $\T=\R/\Z$, takes the form
\[
\mathbb{S}: W^{1,2} (\T,M) \times ]0,+\infty[ \rightarrow \R, \quad 
\mathbb{S}(q,T) := T \int_{\T} (L-\theta)(q(t),q'(t)/T)\, dt,
\]
is positive and satisfies the Palais-Smale condition at every positive level, as shown by G.\ Contreras in \cite{con06} (see also \cite{cipp00}). We recall that $(q,T)$ is a critical point of $\mathbb{S}$ if and only if $\gamma(t):=q(t/T)$ is a $T$-periodic solution of the Euler-Lagrange equation associated to $L-\theta$, which coincides with the Euler-Lagrange equation associated to $L$, of zero energy. Therefore, the Legendre transform identifies $(\gamma,\gamma')$ with a $T$-periodic orbit of the Hamiltonian vector field $X_H$ on $\partial W$, which by our choice of $H$ coincides with the Reeb vector field $R$.
 
The functional $\mathbb{S}$ is in general just $C^{1,1}$ on the Hilbert manifold $\mathcal{M} := W^{1,2}(\T,M)\times ]0,+\infty[$, but under the usual Morse-Bott assumption we can build a smooth pseudo-gradient vector field for it, arguing as in \cite{as08b}. The Morse chain complex and differential complex of such a pseudo-gradient vector field are denoted by
\[
M_*^{]0,+\infty[} (\mathbb{S},s), \quad M^*_{]0,+\infty[} (\mathbb{S},-s),
\]
where $s$ is the auxiliary Morse function on $\crit\, \mathbb{S}$, and their homology and cohomology are isomorphic to
\[
HM^{]0,+\infty[}_k (\mathbb{S},s) \cong H_k(\Lambda M,M), \quad HM^k_{]0,+\infty[} (\mathbb{S},-s) \cong H^k(\Lambda M,M), \quad \forall k\in \Z,
\]
where $M\subset \Lambda M$ denotes the space of constant loops. The restriction of the functional $\mathbb{S}$ to the space $\mathcal{M}^0$ of all $(q,T)\in \mathcal{M}$ with $q$ contractible has {\em critical points at infinity}, in the sense of A.\ Bahri \cite{bah89}: The infimum of $\mathbb{S}|_{\mathcal{M}^0}$ is zero, while the infimum of $\mathbb{S}|_{\mathcal{M}^0}$ over its critical set is positive, and the $\omega$-limit in $\mathcal{M}^0$ of each point $(q,T)\in \mathcal{M}^0$ with 
\[
\mathbb{S}(q,T)< \min_{\crit\, \mathbb{S}|_{\mathcal{M}^0}} \mathbb{S},
\]
with respect to the negative (pseudo-)gradient flow of $\mathbb{S}$ is empty.
Actually, if $\sigma \mapsto (q(\sigma),T(\sigma))$, $\sigma \in ]\sigma^-,\sigma^+[$, is the (maximal) orbit of such a point, then $T(\sigma)$ converges to zero and $q(\sigma)$ converges to some constant loop in $W^{1,2}$, for $\sigma \uparrow \sigma^+$. These facts allow to consider the manifold $M$ as the space of critical points at infinity of $\mathbb{S}$. If we take also these virtual critical points into account, we get an enlarged Morse chain complex and differential complex, that we denote by
\[
M_*(\mathbb{S},s), \quad M^* (\mathbb{S},-s),
\]
such that
\begin{equation}
\label{ide2}
HM_k (\mathbb{S},s) \cong H_k(\Lambda M), \quad HM^k (\mathbb{S},-s) \cong H^k(\Lambda M), \quad \forall k\in \Z.
\end{equation}
These constructions allow us to recover a situation which is similar to the one leading to Theorem \ref{main1}: We can define chain maps
\[
\Xi : M_* (\mathbb{S},s) \rightarrow RF_* (W,\lambda_{\theta},a), \quad \Upsilon : RF_*  (W,\lambda_{\theta},a) \rightarrow M^{1-*} (\mathbb{S},-s),
\]
which are, respectively, injective and surjective, with left and right inverses $\hat{\Xi}$ and $\hat{\Upsilon}$. The definition of $\Xi$ is based on counting solutions of the following mixed problem: Given a critical point $w$ of $\mathbb{S}$ and a critical point $z$ of $\mathbb{A}$, we consider the space of pairs
\[
u : [0,+\infty[ \times \T \rightarrow T^*M, \quad \eta: [0,+\infty[ \rightarrow \R,
\]
which solve the Rabinowitz-Floer equation (\ref{Fleq}), converge to $z$ for $s\rightarrow +\infty$, and are such that $(\pi\circ u(0,\cdot),\eta(0))$ is in the unstable manifold of $w$ with respect to the negative pseudo-gradient flow of $\mathbb{S}$. The definition of $\Upsilon$ is similar. The compactness of these spaces of maps and the existence of a left, respectively right, inverse for $\Xi$, respectively $\Upsilon$, are based on the following inequality,
\begin{equation}
\label{fenchel}
\mathbb{A}(x,\eta) \leq \mathbb{S}(\pi\circ x,\eta), \quad \forall x\in \Lambda T^* M, \; \forall \eta>0,
\end{equation}
which is a straighforward consequence of Fenchel duality (see Lemma \ref{levrel2} below).

Again, the composition $\Upsilon \circ \Xi$ is chain homotopic to zero by a chain homotopy $Q: M_*(\mathbb{S},s) \rightarrow M^{-*}(\mathbb{S},-s)$, which can be used to modify $\Xi$ within its chain homotopy class and to obtain the chain map
\[
\Omega := \Xi - \hat{\Upsilon} Q \partial - \partial \hat{\Upsilon} Q.
\]
Let $RF^I_*$ denote the chain complex obtained by considering only the critical points $z$ of the Rabinowitz action functional $\mathbb{A}$ with action $\mathbb{A}(z)$ in the interval $I\subset \R$. 
Then we have the following result, which can be seen as a generalized version of Theorem \ref{main1}:

\begin{mainthm}
\label{main2}
The chain maps $\Xi$ and $\Psi$ induce isomorphisms
\[
M_*^{]0,+\infty[}(\mathbb{S},s) \cong RF^{]0,+\infty[}_*(W,\lambda_{\theta},a), \quad RF^{]-\infty,0[}_*(W,\lambda_{\theta},a) \cong M^{1-*}_{]0,+\infty[} (\mathbb{S},-s),
\]
and, in particular,
\[ 
HRF^{]0,+\infty[}_k (W,\lambda_{\theta},a) \cong H_k (\Lambda M,M), \quad
HRF_k^{]-\infty,0[}(W,\lambda_{\theta},a) \cong H^{1-k} (\Lambda M,M), \quad \forall k\in \Z.
\]
The short sequence of chain complexes
\[
0 \rightarrow M_*(\mathbb{S},s) \stackrel{\Omega}{\rightarrow} RF_*(W,\lambda_{\theta},a) \stackrel{\Upsilon}{\rightarrow} M^{1-*} (\mathbb{S},-s) \rightarrow 0
\]
is exact and has a canonical splitting. After identifying Morse (co)homology with singular (co)homology by (\ref{ide2}), the associated long exact sequence is
\[
\dots \rightarrow H_k (\Lambda M) \stackrel{\Omega_* = \Xi_*}{\longrightarrow} H RF_k (W,\lambda_{\theta},a) \stackrel{\Upsilon_*}{\longrightarrow} H^{1-k} (\Lambda M) \stackrel{\Delta_*}{\longrightarrow} H_{k-1} (\Lambda M) \rightarrow \dots
\]
where the connecting homomorphism $\Delta_*$ can be non-zero only in degree zero, where
\[
\Delta_* : H^0(\Lambda M) \rightarrow H_0 (\Lambda M)
\]
vanishes on the components of $\Lambda M$ containing non-contractible loops, and is the multiplication by the Euler number $\chi(T^* M)$ on the component of $\Lambda M$ consisting of contractible loops.
\end{mainthm}   

See Section \ref{finale} for more precise statements.

\medskip

\noindent {\bf Estimates.}
We conclude this Introduction by discussing the question of the $L^{\infty}$ estimates for the Rabinowitz-Floer equation (\ref{Fleq}). The inequalities (\ref{actin1}) and (\ref{actin2}), or more generally (\ref{fenchel}), play a fundamental role in our proofs. These inequalities require the Hamiltonian $H$ to be the classical geodesic Hamiltonian (\ref{laham}) (up to an additive constant), or more generally to be the Fenchel dual of a Tonelli Lagrangian $L\in C^{\infty}(TM)$. The critical points of the Rabinowitz action functional $\mathbb{A}$ are confined to the energy level $H^{-1}(0)$, so Rabinowitz-Floer homology should not be very sensitive to the behavior of the Hamiltonian at infinity. Indeed, the uniform estimates for the second component of solutions $(u,\eta)$ of (\ref{Fleq}) (see Proposition 3.2 in \cite{cfo09}, or Proposition \ref{etalim} below) use only the properties of $H$ near its zero level. However, uniform estimates on $u$ are based on the maximum principle. When the almost complex structures $J_t$ are of contact type and the Hamiltonian $H$ depends only on the second variable $\rho$ outside a compact subset of the collar $\partial W \times [1,+\infty[$, say $H=h(\rho)$, the real valued function $\rho := \rho\circ u$ 
satisfies the elliptic differential inequality
\begin{equation}
\label{eldineq}
\Delta \rho - \eta \rho h^{\prime\prime} (\rho) \frac{\partial \rho}{\partial s} \geq \rho h'(\rho) \int_{\T} H(u)\, dt
\end{equation}
outside a compact set. When the Hamiltonian $H$ is eventually constant, that is constant outside a compact set, the right-hand side vanishes for $\rho$ large and the maximum principle immediately produces uniform estimates on $\rho$, hence on $u$. However, the fact that $H$ takes also negative values implies that in general the right-hand side of (\ref{eldineq}) might not have a lower bound.
For this reason, K.\ Cieliebak and U.\ Frauenfelder's original definition of the Rabinowitz-Floer complex uses eventually constant Hamiltonians. By studying the differential inequality for $\log \rho$, one can actually apply the maximum principle and obtain uniform estimates also when $H$ is eventually linear in $\rho$ (this fact is used in \cite{cfo09}), but here we need to consider Hamiltonians which are quadratic in $\rho$. As we show in Section \ref{unesti}, uniform estimates for $u$ with such Hamiltonians - and actually for much more general ones - can be derived from the integral version of the maximum principle of A.\ D.\ Aleksandrov. We also need similar estimates in the case of domains with boundary, as the half-cylinders used in the constructions of the chain maps $\Phi$, $\Psi$, $\Xi$, and $\Upsilon$, or the finite cylinders used in the construction of the chain homotopies $P$ and $Q$. The estimates in these cases come from a version of the Aleksandrov maximum principle for solutions of elliptic inequalities satisfying Neumann boundary conditions on part of the boundary, that we prove in the Appendix A. Since these $L^{\infty}$ estimates may find other applications, we prove them at the beginning of the paper, in a more general setting than strictly needed here, for the completion of arbitrary Liouville domains. We also discuss in a similar generality the energy estimates for solutions of the equation (\ref{Fleq}) when the Hamiltonian $H$ depends on $s$, an important issue when dealing with invariance properties of the Rabinowitz-Floer complex.

\paragraph{Acknowledgment.} We are grateful to Kai Cieliebak, Urs Frauenfelder, and Alexandru Oancea for sharing with us an early version of their paper \cite{cfo09}, and to Alexandru Oancea also for the observations which conclude Section 9. The first author wishes to thank the Max-Planck-Institut f\"ur Mathematik in den Naturwissenschaften of Leipzig for its kind hospitality, and the Alexander von Humboldt Stiftung for financial support.  

\numberwithin{equation}{section}

\section{The setting}
\label{setting}

\paragraph{Liouville domains.} Let $(W,\lambda)$ be a $2n$-dimensional {\em  Liouville domain}. This means that $W$ is a compact connected $2n$-dimensional manifold with boundary $\partial W$, that $\lambda$ is a one-form on $W$ (the Liouville form) such that $\omega:= d\lambda$ is non-degenerate (hence it is a symplectic form), and that the {\em Liouville vector field} $Y$ defined by $\iota_Y \omega = \lambda$ is transverse to $\partial W$ and points in the outward direction. The last condition is equivalent to the fact that $\alpha:= \lambda|_{\partial W}$ is a contact form on $\partial W$, whose associated $(2n-1)$-dimensional volume form $\alpha \wedge (d\alpha)^{n-1}$ induces the standard orientation of $\partial W$ as the boundary of $W$ ($W$ is endowed with the orientation induced by the volume form $\omega^n$). 

The {\em Reeb vector field} $R$ is the nowhere vanishing vector field on $\partial W$ defined by $\iota_R d\alpha = 0$ and $\alpha(R)=1$.  

Notice that the flow $\phi_t^Y$ of $Y$ is defined for all $t\leq 0$ and induces an embedding
\[
j: \partial W \times ]0,1] \rightarrow W, \quad (x,\rho) \mapsto \phi_{\log \rho}^Y(x),
\] 
such that $j^* \lambda = \rho \alpha$ and $j^* Y = \rho \partial/\partial \rho$, where $\rho$ denotes the coordinate on the second factor of $\partial W \times ]0,1]$. The {\em completion} of $W$ is the manifold
\[
\hat{W} := W \cup_{\partial W} ( \partial W \times [1,+\infty[).
\]
The Liouville form $\lambda$, the Liouville vector field $Y$, and the symplectic form $\omega$ extend naturally on the completion by setting 
\[
\lambda|_{\partial W \times [1,+\infty[} = \rho \alpha, \quad Y|_{\partial W \times [1,+\infty[} = \rho \frac{\partial}{\partial \rho}, \quad \omega|_{\partial W \times [1,+\infty[} = d(\rho \alpha),
\]
so that the identities $\omega=d\lambda$ and $\iota_Y \omega = \lambda$ hold on the whole $\hat{W}$. The open set $\partial W \times ]0,+\infty[$, endowed with the restriction of $\omega$, is the {\em symplectization} 
of $\partial W$. 

See \cite{sei06b} for many examples of Liouville domains. In this paper we are mainly interested in the following particular subdomains of a cotangent bundle.

\paragraph{Example 1: The cotangent disk bundle.} Let $(M,g)$ be a closed Riemannian manifold, and let $g^*$ be the dual Riemannian form on the vector bundle $T^*M$. The cotangent bundle $T^*M$ carries the standard Liouville form $\lambda= p\, dq$, which can be characterized as the unique 1-form on $T^*M$ such that for every 1-form $\theta$ on $M$, seen as a section of $T^* M$, there holds $\theta^* \lambda = \theta$. The restriction of $\lambda$ to the cotangent disk bundle
\[
D^* M := \set{(q,p)\in T^*M}{g^*(p,p)\leq 1}
\]
defines a Liouville domain, whose boundary is the cotangent sphere bundle
\[
S^* M := \set{(q,p)\in T^*M}{g^*(p,p)=1}.
\]
The map
\[
S^*M \times ]0,+\infty[ \rightarrow T^*M, \quad \bigl((q,p),\rho \bigr) \mapsto (q, \rho p)
\]
extends to a diffeomorphism from the completion of $D^* M$ onto $T^*M$, which identifies the Liouville form of the completion of $D^* M$ with the standard Liouville form of $T^*M$. 

\paragraph{Example 2: Uniformly convex domains containing a Lagrangian graph.}
The above example has the following generalization. Let $M$ be a closed manifold and let $W$ be a compact subset of $T^*M$, which has a smooth fiberwise uniformly convex boundary and contains a Lagrangian graph. The uniform convexity assumption means that $W$ can be represented as a closed sublevel of a smooth function on $T^*M$ whose fiberwise second differential is everywhere strictly positive. The assumption on the Lagrangian graph means that there exists a {\em closed} one-form $\theta$ on $M$ such that
\[
\set{(q,\theta(q))}{q\in M} \subset \mathrm{Int}\, (W).
\]
If $\lambda$ is the standard Liouville form on $T^*M$, the set $W$ is a Liouville domain with respect to the one-form
\[
\lambda_{\theta} := \lambda - \pi^* \theta,
\]
where $\pi:T^*M \rightarrow M$ denotes the projection. Indeed, the fact that $\theta$ is closed implies that $d\lambda_{\theta} = d\lambda = \omega$. Moreover, in local cotangent coordinates $q_1,\dots,q_n,p_1,\dots,p_n$ the Liouville vector field $Y_{\theta}$ associated to $\lambda_{\theta}$ takes the form
\[
Y_{\theta} (q,p) = \sum_{j=1}^n (p_j - \theta_j(q) ) \frac{\partial}{\partial p_j}, \quad \mbox{where } \theta(q) = \sum_{j=1}^n \theta_j(q) dq_j,
\]
so the fact that for every $q\in M$ the set $W\cap T_q^* M$ is convex and its interior contains the covector $\theta(q)$ implies that $Y_{\theta}$ is transverse to $\partial W$. It would have been sufficient to assume that for each $q\in M$ the set $W \cap T_q^*M$ is starshaped with respect to $\theta(q)$. However, uniform convexity allows us to use the dual Lagrangian formulation, see Section \ref{scdsec}.  

The map
\[
\partial W \times ]0,+\infty[ \rightarrow T^*M, \quad \bigl( (q,p),\rho \bigr) \rightarrow \bigl( q, \theta(q) + \rho (p-\theta(q)) \bigr)
\]
extends to a smooth diffeomorphism from the completion $\hat{W}$ onto $T^*M$, which identifies the Liouville form of the completion to the one-form $\lambda_{\theta}$ on $T^*M$.

\paragraph{The Rabinowitz action functional.}
A smooth real function $H\in C^{\infty}(\hat{W})$ defines a Hamiltonian vector field $X=X_H$ on $\hat{W}$ by the identity $\iota_{X} \omega = -dH$. The level sets of $H$ are invariant with respect to the flow of $X$. We assume that $0$ is a regular value of $H$ and we denote by $\Sigma$ the hypersurface
\[
\Sigma = \set{w\in \hat{W}}{H(w)=0},
\]
that we assume to be non-empty. In the next sections,  $\Sigma$ will be the boundary of $W$, but for now we do not make such an assumption.
 
Let $\Lambda \hat{W}$ denote the space of smooth closed curves $x:\T \rightarrow \hat{W}$, where $\T := \R/\Z$. The {\em Rabinowitz action functional} (see \cite{rab78}) is the following real valued function on $\Lambda \hat{W} \times \R$:
\[
\mathbb{A} (x,\eta) = \mathbb{A}_H(x,\eta) = \int_{\T} x^*(\lambda) - \eta \int_{\T} H(x(t))\, dt.
\]
The point $(x,\eta)$ is a critical point of $\mathbb{A}$ if and only if 
\begin{eqnarray}
\label{runo}
x'(t) & = & \eta X(x(t)), \\
\label{rue}
H(x(t)) & = & 0,
\end{eqnarray}
for every $t\in \T$. When $\eta=0$, the above equations are equivalent to the fact that $x$ is a constant loop on $\Sigma$. When $\eta\neq 0$, the above equations are equivalent to the fact that the rescaled curve
\[
y: \R \rightarrow W, \quad y(s) := x(s/\eta),
\]
is a (necessarily non-constant) closed orbit of $X$ on $\Sigma$ of (not necessarily minimal) period $|\eta|$. Therefore, the critical set of $\mathbb{A}$ consists of a copy of $\Sigma$, on which $\mathbb{A}$ is identically zero, and for each non-constant closed Hamiltonian orbit $y$ on $\Sigma$ of minimal period $T>0$ and for each non-zero integer $k$, of the following copy of $\T$,
\[
\set{(y(\tau + kT\cdot),kT)}{\tau \in \T}.
\]

Assume that the hypersurface $\Sigma$ is compact and that the Liouville vector field $Y$ is transverse to $\Sigma$, or equivalently that the one-form $\alpha_{\Sigma}$ obtained by restricting $\lambda$ to $\Sigma$ is a contact form. If $(x,\eta)$ is a critical point of $\mathbb{A}$, then 
\[
\mathbb{A}(x,\eta) = \int_{\T} x^*(\lambda) =  \int_{\T} \lambda(x)[\eta X(x)] \, dt = \eta \int_{\T} \omega ( Y(x), X(x) )\, dt = \eta \int_{\T} dH(x)[Y(x)]\, dt.
\]
Therefore, using the fact that $Y$ is transverse to the compact hypersurface $\Sigma=H^{-1}(0)$, on which $dH$ never vanishes, we deduce that there exists a positive number $\delta$ such that
\[
| \mathbb{A}(x,\eta) | \geq \delta |\eta|, 
\]
and we conclude that $|\mathbb{A}|$ is coercive on the critical set of $\mathbb{A}$.

If we also assume that the Hamiltonian vector field $X$ restricts to the Reeb vector field $R_{\Sigma}$ of $(\Sigma,\alpha_{\Sigma})$, we get the identity
\[
\mathbb{A}(x,\eta) =   \int_{\T} \lambda(x)[\eta X(x)] \, dt = \eta \int_{\T} \alpha_{\Sigma} (x)[R_{\Sigma}(x)]\, dt = \eta,
\]
so the value of the Rabinowitz action functional on a critical point $(x,\eta)$ with $\eta\neq 0$ is plus or minus the period of the corresponding Reeb orbit $y=x(\cdot/\eta)$, depending on whether the orientation of $x$ agrees with the orientation of $y$ or not.

\paragraph{The negative gradient equation.}
Let $J_t$, $t\in \T$, be a loop of almost complex structures on $\hat{W}$ which are $\omega$-compatible, meaning that $\langle \cdot, \cdot \rangle_t := \omega(J_t \cdot, \cdot)$ is a loop of Riemannian metrics on $\hat{W}$. The corresponding norm is denoted as
\[
|\zeta|_t := \sqrt{\langle \zeta,\zeta \rangle_t} = \sqrt{ \omega (J_t(w) \zeta, \zeta) }, \quad \forall \zeta\in T_w \hat{W}, \; \forall w\in \hat{W}, \; \forall t\in \T.
\]
Here we are using a sign convention which is different from the one used in \cite{cf09} and \cite{cfo09} (and in most of the literature). The reason is that we prefer to see the {\em negative} gradient flow equation for $\mathbb{A}$ as a Cauchy-Riemann type equation, not the positive one. 
The loop of metrics $\langle \cdot, \cdot \rangle_t$ induces the following $L^2$ metric on $\Lambda \hat{W}$:
\[
\langle \xi,\zeta \rangle_{L^2} := \int_{\T} \langle \xi(t),\zeta(t) \rangle_t \, dt, 
\]
for every pair of section $\xi,\zeta$ of $x^*(T\hat{W})$, for every $x\in \Lambda \hat{W}$. Then $\Lambda \hat{W} \times \R$ is endowed with the product metric by the standard metric of $\R$, the corresponding norm is denoted by $\|\cdot\|$, and the  gradient of $\mathbb{A}$ with respect to such a metric is
\[
\nabla \mathbb{A} (x,\eta) = \left( J_t(x) (x' - X_H(x)), - \int_{\T} H(x(t))\, dt \right).
\]
Therefore, the negative gradient equation for $\mathbb{A}$, i.e.\
\begin{equation}
\label{rfleq}
v'(s) = - \nabla \mathbb{A}(v(s)), 
\end{equation}
with $v=(u,\eta):\R \rightarrow \Lambda \hat{W} \times \R$, is the following system of PDE's:
\begin{eqnarray}
\label{rfleq1}
& \frac{\partial u}{\partial s} (s,t) + J_t (u(s,t)) \left( \frac{\partial u}{\partial t} (s,t)- \eta X(u(s,t)) \right) = 0, \\
\label{rfleq2}
& \frac{d\eta}{ds} = \int_{\T} H(u(s,t))\, dt,
\end{eqnarray}
where $(s,t) \in \R \times \T$. We shall refer to (\ref{rfleq}), or equivalently (\ref{rfleq1}-\ref{rfleq2}), as to the {\em Rabinowitz-Floer equation}.

\section{$\mathbf{L^{\infty}}$ estimates}
\label{unesti}
 
K.\ Cieliebak and U.\ Frauenfelder's construction of Rabinowitz-Floer homology in \cite{cf09} uses compactly supported Hamiltonian vector fields, or equivalently Hamiltonians $H$ which are constant outside a compact set. In this case, the equation (\ref{rfleq1}) reduces to the equation for $J_t$-holomorphic curves outside a compact set. If the almost complex structures $J_t$ are of contact type outside a compact subset of $\hat{W}$ (see the definition below), the maximum principle implies that no such curve can have inner tangengies to the foliation $\{\partial W \times \{\rho\}\}_{\rho>0}$, for $\rho$ large enough. Since one is interested only in solutions  $v=(u,\eta)$ of (\ref{rfleq1}-\ref{rfleq2}) such that $u(s,\cdot)$ converges to some orbits in $\partial W \times \{1\}$ for $s\rightarrow \pm \infty$, the component $u$ is bounded in $L^{\infty}$, and the compactness analysis in \cite{cf09} reduces to finding bounds for the second component of the solutions $v=(u,\eta)$ of (\ref{rfleq1}-\ref{rfleq2}).    

The class of admissible Hamiltonians was enlarged in \cite{cfo09}, admitting also Hamiltonians $H$ which are (non-decreasing) linear functions of $\rho$ outside from a compact set. Since we wish to consider also geodesic Hamiltonians on cotangent bundles, we need to deal with Hamiltonians which are (non-decreasing) quadratic functions of $\rho$. Actually, we shall prove an $L^{\infty}$ bound on $u$ for a larger class of Hamiltonians, including, for instance, all Hamiltonians which, outside of a compact set, are non-decreasing convex functions of $\rho$ growing at most polynomially.  

\paragraph{Estimates on $\eta$.} We start by recalling some results from \cite{cf09} which do not require any growth assumption at infinity on $H$, and just depend on the fact that the Liouville vector field $Y$ is transverse to the compact hypersurface
$\Sigma=H^{-1}(0)$. 
The following result is proved in \cite{cf09}, Proposition 3.2.

\begin{lem}
\label{luno}
Let $H$ be a smooth Hamiltonian on $\hat{W}$, and assume that $\epsilon_0>0$ is such that the set
\[
V_{\epsilon_0} := \set{w\in \hat{W}}{|H(w)| \leq \epsilon_0}
\]
is compact, and
\[
\lambda (X) - H \geq \theta \quad \mbox{on } V_{\epsilon_0},
\]
for some number $\theta>0$.
Then for every $0<\epsilon\leq \epsilon_0$ there exists $\delta= \delta(\epsilon)>0$ such that if $
\|\nabla \mathbb{A}(x,\eta)\| \leq \delta$
then $x(\T) \subset V_{\epsilon}$, and
\[
|\eta| \leq \frac{1}{\theta} \bigl( |\mathbb{A}(x,\eta)| +   \kappa_0 \| \nabla \mathbb{A} (x,\eta) \| \bigr), \quad \mbox{where } \kappa_0 = \max_{\substack{t\in \T \\ w\in V_{\epsilon}}} |Y(w)|_t.
\]
More precisely, one can take $\delta$ to be
\[
\delta = \frac{\epsilon}{2 \max \{1,\kappa_1\}}, \quad \mbox{where } \kappa_1 := \max_{\substack{t\in \T \\ w\in V_{\epsilon}}} |\nabla H(w)|_t.
\]
\end{lem}

If we assume that $0$ is a regular value of $H$, that $\Sigma = H^{-1}(0)$ is compact, that the Liouville vector field $Y$ is transverse to the hypersurface $\Sigma$ with $dH[Y]>0$, and that the infimum of $H$ on the complement of some compact set is positive, then the assumptions of the above lemma hold, because of the identity
\[
\lambda(X) - H = \omega (Y,X) - 0 = dH[Y] \quad \mbox{on } \Sigma. 
\]

In order to establish invariance properties of the Rabinowitz-Floer complex, one needs to consider smooth families $\{H_s\}_{s\in \R}$ of Hamiltonians on $\hat{W}$, and correspondingly the system of PDE's:
\begin{eqnarray}
\label{zeroa}
& \frac{\partial u}{\partial s} + J_t(u) \Bigl( \frac{\partial u}{\partial t} - \eta X_{H_s} (u) \Bigr) = 0 , \\ 
\label{zerob}
& \frac{d\eta}{ds} = \int_{\T} H_s(u)\, dt.
\end{eqnarray}
For this reason, the next results are stated for $s$-dependent Hamiltonians. We shall assume that $H_s$ does not depend on $s$ on $]-\infty,0]$ and on $[1,+\infty[$, that is
\begin{equation}
\label{dipes}
H_s = H_0 , \;\; \forall s\leq 0\; \quad H_s = H_1, \;\; \forall s\geq 1.
\end{equation}
A first consequence of Lemma \ref{luno} is the following uniform estimate on $\eta$ (a simple variant of Corollary 3.5 in \cite{cf09}, the difference being that here $H$ need not be bounded from above):

\begin{prop}
\label{etalim}
Assume that $\{H_s\}_{s\in \R}$ is a smooth family of Hamiltonians on $\hat{W}$ which satisfies (\ref{dipes}) and such that:
\begin{enumerate}
\item There is a number $\mu>0$ such that $H_s(w)\geq -\mu$, for every $s\in \R$ and $w\in \hat{W}$.
\item There is a number $\epsilon_0>0$ such that the set
\[
V_{\epsilon_0} := \set{(s,w)\in [0,1]\times \hat{W}}{|H_s(w)| \leq \epsilon_0}
\]
is compact, and there is a positive number $\theta>0$ such that
\[
\lambda\bigl(X_{H_s}(w)\bigr) - H_s(w) \geq \theta, \quad \forall (s,w) \in V_{\epsilon_0}.
\]
\end{enumerate}
Then for every pair of numbers $A,E$ there exists a constant $c=c(A,E)$ such that if $v=(u,\eta)$ is a solution of (\ref{zeroa}-\ref{zerob}) with 
\begin{eqnarray}
\label{AAA}
\sup_{s\in \R} |\mathbb{A}_{H_s}(v(s))| \leq A, \\ 
\label{EEE}
\int_{-\infty}^{+\infty} \|\nabla \mathbb{A}_{H_s}(v(s))\|^2\, ds \leq E,
\end{eqnarray}
then $|\eta(s)| \leq c$ for every $s\in \R$.
\end{prop}

\begin{proof}
By the assumption (b), the hypotheses of Lemma \ref{luno} are uniformly fulfilled, so if
$\delta_0=\delta(\epsilon_0)$ is the number provided by this lemma, there exists a number $c'(A)$ such that $|\eta(s)| \leq c'(A)$ 
when $\|\nabla \mathbb{A}_{H_s} (v(s))\|\leq \delta_0$. Let $I$ be a maximal interval on which $\| \nabla \mathbb{A}_{H_s}(v(s))\|>\delta_0$. By (\ref{EEE}), $I=]a,b[$ with
$b-a \leq E/\delta_0^2$.
By maximality and continuity, $\|\nabla \mathbb{A}_{H_s}(v(s))\|=\delta_0$ for $s=a$ and $s=b$, so
\[
|\eta(a)| \leq c'(A), \quad |\eta(b)|\leq c'(A).
\]
By (\ref{zerob}), and assumption (a),
\[
\eta'(s) = \int_{\T} H_s(u(s,t))\, dt \geq - \mu.
\]
If $s\in I$, then
\begin{eqnarray}
\label{low}
\eta(s) =& \eta(a) + \int_a^s \eta'(\sigma)\, d\sigma \geq & - c'(A) - (b-a) \mu, \\
\label{up}
\eta(s) =& \eta(b) - \int_s^b \eta'(\sigma)\, d\sigma \leq & c'(A) + (b-a) \mu.
\end{eqnarray}
Therefore, the thesis holds with $c(A,E) = c'(A) + \mu E/\delta_0^2 $.
\end{proof} \qed

\begin{rem}
\label{etabdry}
Let us consider solutions $v=(u,\eta)$ of (\ref{zeroa}-\ref{zerob}) on the half-line $\R^+:= [0,+\infty[$. Then (\ref{low}) shows that the conclusion of the above proposition remains true if we have an a priori lower bound on $\eta(0)$. Analogously, by (\ref{up}), the statement for solutions on the half-line $\R^-:= ]-\infty,0]$ holds if we have an a priori upper bound on $\eta(0)$.
\end{rem} 

\paragraph{Estimates on $u$.} In order to get uniform estimates on the first component of a solution $(u,\eta)$ of (\ref{zeroa}-\ref{zerob}), it is useful to assume that the loop of $\omega$-compatible almost complex structures $J_t$ is of {\em contact type} outside a compact set, meaning that it satisfies the equation
\begin{equation}
\label{contactJ}
d\rho \circ J_t = \lambda \quad \mbox{on } \partial W \times [\rho_0,+\infty[,
\end{equation}
for every $t\in \T$, for some positive number $\rho_0$. 
Together with the $\omega$-compatibility, this is equivalent to the fact that the symplectic splitting of $T_{(x,\rho)} (\partial W \times ]0,+\infty[)$ into $\ker \alpha$ and $\R R \oplus \R Y$ is $J_t$-invariant, that the restriction of $J_t$ to the first component is compatible to $\omega|_{\ker \alpha}$, and that
\[
J_t R = Y, \quad J_t Y = - R.
\]
For every $t\in \T$, the vectors $R$ and $Y$ are $\langle\cdot,\cdot \rangle_t$-orthogonal at $(x,\rho)\in \partial W \times ]0,+\infty[$, and have norm $\sqrt{\rho}$:
\begin{equation}
\label{ort}
g_t ( R, Y ) = 0, \quad |R(x,\rho)|_t^2 = | Y(x,\rho)|_t^2 = \rho.
\end{equation}
It is also useful to assume that the Hamiltonians $H_s$ are radial outside a compact set, meaning that
\begin{equation}
\label{hamatinf}
H_s(x,\rho) = h(s,\rho) \quad  \mbox{on } \partial W \times [\rho_0,+\infty[,
\end{equation}
where $h:\R \times [\rho_0,+\infty[ \rightarrow \R$ is a smooth function. With this choice, the Hamiltonian vector field $X_{H_s}$ has the form
\begin{equation}
\label{radham}
X_{H_s}(x,\rho) = \frac{\partial h}{\partial \rho}(s,\rho) R(x), \quad \forall (x,\rho) \in \partial W \times [\rho_0,+\infty[.
\end{equation}
The following result is based on a standard computation (see \cite{sei06b}, or Lemma 4.1 in \cite{cfo09}):

\begin{lem}
\label{conti}
Assume that $J_t$ is of contact type on $\partial W \times [\rho_0 , +\infty[$ and that the smooth family of Hamiltonians $\{H_s\}_{s\in \R}$ satisfies (\ref{hamatinf}). Let $(u,\eta)$ be a solution of (\ref{zeroa}-\ref{zerob}), and consider the real function $\rho(s,t) := \rho\circ u(s,t)$, for every $(s,t)$ in the open subset $u^{-1} (\partial W \times ]\rho_0,+\infty[ )$ of $\R \times \T$. Then
\begin{equation}
\label{laplarho}
\Delta \rho = \left| \frac{\partial u}{\partial s} \right|_t^2 + \eta \rho  \frac{\partial^2 h}{\partial \rho^2} \frac{\partial \rho}{\partial s} + \rho \frac{\partial h}{\partial \rho}  \int_{\T} H_s(u)\, dt + \eta \rho \frac{\partial^2 h}{\partial s \partial \rho},
\end{equation}
and
\begin{equation}
\label{diffineq}
| \nabla \rho|^2 \leq \rho \left| \frac{\partial u}{\partial s} \right|_t^2,
\end{equation}
for every $(s,t) \in u^{-1} (\partial W \times ]\rho_0,+\infty[ )$.
\end{lem}

\begin{proof}
By (\ref{zeroa}) and (\ref{radham}),
\[
\frac{\partial \rho}{\partial s} = d \rho (u) \left[ \frac{\partial u}{\partial s} \right] = d \rho (u) \left[ - J_t(u) \Bigl( \frac{\partial u}{\partial t} - \eta X_{H_s} \Bigr) \right] = d \rho (u) \left[ - J_t(u) \Bigl( \frac{\partial u}{\partial t} - \eta \frac{\partial h}{\partial \rho} R\Bigr) \right].
\]
Using also (\ref{contactJ}), we get
\begin{equation}
\label{quindici}
\frac{\partial \rho}{\partial s} = - \lambda \left( \frac{\partial u}{\partial t} - \eta \frac{\partial h}{\partial \rho} R \right) = - \lambda \left(\frac{\partial u}{\partial t} \right) + \eta \rho  \frac{\partial h}{\partial \rho} .
\end{equation}
Similarly,
\[
\frac{\partial \rho}{\partial t} = d\rho (u) \left[ J_t(u) \Bigl( \frac{\partial u}{\partial s} - \eta  \frac{\partial h}{\partial \rho} J_t(u) R  \Bigr) \right] = \lambda \left( \frac{\partial u}{\partial s} \right).
\]
Therefore,
\[
d^c \rho := \frac{\partial \rho}{\partial t} \, ds - \frac{\partial \rho}{\partial s} \, dt = u^*(\lambda) - \eta \rho \frac{\partial h}{\partial \rho} \, dt.
\]
By differentiating once more, using (\ref{zerob}), we find
\begin{equation}
\label{dd}
dd^c \rho = u^* (\omega) - \left( \rho \frac{\partial h}{\partial \rho} \int_{\T} H_s(u)\, dt + \eta \frac{\partial h}{\partial \rho} \frac{\partial \rho}{\partial s} + \eta \rho \frac{\partial^2 h}{\partial \rho^2} \frac{\partial \rho}{\partial s} + \eta \rho \frac{\partial^2 h}{\partial s \partial \rho} \right) \, ds\wedge dt.
\end{equation}
Using (\ref{zeroa}), (\ref{radham}), and the $\omega$-compatibility of $J_t$, we can compute $u^*(\omega)$ as
\begin{eqnarray*}
u^*(\omega) = \omega \left( \frac{\partial u}{\partial s}, \frac{\partial u}{\partial t} \right) \, ds \wedge dt = \omega \left( \frac{\partial u}{\partial s}, J_t(u) \frac{\partial u}{\partial s} + \eta X_{H_s} \right) \, ds \wedge dt \\ 
= \left( - \Bigl| \frac{\partial u}{\partial s} \Bigr|_t^2 + \eta dH_s(u) \Bigl[ \frac{\partial u}{\partial s} \Bigr] \right) \, ds\wedge dt = \left( - \Bigl| \frac{\partial u}{\partial s} \Bigr|_t^2 + \eta \frac{\partial h}{\partial \rho} \frac{\partial \rho}{\partial s} \right) \, ds\wedge dt.
\end{eqnarray*}
Together with (\ref{dd}) and the fact that $dd^c \rho = - \Delta \rho \, ds\wedge dt$, we deduce the identity (\ref{laplarho}).

By equations (\ref{zeroa}), (\ref{ort}), and (\ref{radham}), the $\langle \cdot, \cdot \rangle_t$-orthogonal projection of $\partial u/\partial s$ onto the line $\R\, R$ is
\begin{eqnarray*}
\frac{1}{\rho} \left\langle \frac{\partial u}{\partial s}, R\right\rangle_t R = \frac{1}{\rho} \left\langle -J_t \frac{\partial u}{\partial t} + \eta \frac{\partial h}{\partial \rho} J_t R , R\right\rangle_t R \\
= \frac{1}{\rho} \left\langle \frac{\partial u}{\partial t}, Y \right\rangle_t R = \frac{1}{\rho} \left\langle \frac{\partial \rho}{\partial t} \frac{\partial}{\partial \rho},Y \right\rangle_t R = 
\frac{1}{\rho^2} \frac{\partial \rho}{\partial t} |Y|_t^2 R =
\frac{1}{\rho} \frac{\partial \rho}{\partial t} R.
\end{eqnarray*}
Since the $\langle \cdot,\cdot\rangle_t$-orthogonal projection of $\partial u/\partial s$ onto the line $\R Y = \R \partial/\partial \rho$ is obviously 
\[
\frac{\partial \rho}{\partial s} \frac{\partial}{\partial \rho} = \frac{1}{\rho} \frac{\partial \rho}{\partial s} Y,
\] 
we can estimate the norm $\partial u/\partial s$ from below by the norm of its projection onto the plane spanned by the vectors $R$ and $Y$, obtaining
\[
\left| \frac{\partial u}{\partial s} \right|_t^2 \geq \left| \frac{1}{\rho} \frac{\partial \rho}{\partial t} R \right|_t^2 + \left| \frac{1}{\rho} \frac{\partial \rho}{\partial s} Y \right|_t^2 
= \frac{1}{\rho^2} \left| \frac{\partial \rho}{\partial t} \right|^2 |R|_t^2 + \frac{1}{\rho^2} \left| \frac{\partial \rho}{\partial s} \right|^2 | Y|_t^2 = \frac{1}{\rho} |\nabla \rho|^2,
\]
which proves (\ref{diffineq}).
\end{proof} \qed

Our first new observation is the following a priori estimate for the $u$-part of a solution $(u,\eta)$ of (\ref{zeroa}-\ref{zerob}):

\begin{prop}
\label{ulim}
Assume that $J_t$ is of contact type on $\partial W \times [\rho_0 , +\infty[$, that
the smooth family of Hamiltonians $\{H_s\}_{s\in \R}$ satisfies (\ref{dipes}), 
\begin{equation}
\label{trasve}
\lambda\bigl(X_{H_s}(w)\bigr) - H_s(w) \geq \theta, \quad \forall w\in \hat{W} \mbox{ such that } H_s(w)=0, \; \forall s\in \R,
\end{equation}
for some $\theta>0$, and  
\[
H_s(x,\rho) = h(s,\rho), \quad \forall (x,\rho) \in \partial W \times [\rho_0,+\infty[, \quad \forall s\in \R,
\]
where the smooth function $h:\R \times [\rho_0,+\infty[ \rightarrow \R$ satisfies the conditions:
\begin{enumerate}
\item $\frac{\partial h}{\partial \rho}(s,\rho) \geq 0$,
\item $h(s,\rho)\geq \rho/\kappa$,
\item $\left|\frac{\partial^2 h}{\partial \rho^2}(s,\rho)\right| \leq \kappa \rho^{\gamma}$, $\left|\frac{\partial^2 h}{\partial s \partial \rho}(s,\rho)\right| \leq \kappa \rho^{\gamma}$,
\end{enumerate}
for every $(s,\rho) \in [\rho_0,+\infty[ \times \R$, for some constants $\kappa >0$, $\gamma>0$. 
Then for every pair of numbers $A,E$ there exists a compact set $K = K(A,E) \subset \hat{W}$ such that for every solution $v=(u,\eta)$ of (\ref{zeroa}-\ref{zerob}) such that 
\begin{eqnarray*}
\sup_{s\in \R} |\mathbb{A}_{H_s}(v(s))| \leq A,\\
\int_{-\infty}^{+\infty} \|\nabla \mathbb{A}_{H_s}(v(s)) \|^2 \, ds \leq E,
\end{eqnarray*}
 $u(s,t)$ belongs to $K$ for every $(s,t)\in \R \times \T$.
\end{prop}

\begin{proof}
Since $W$ is compact, it is enough to find a uniform upper bound for the function $\rho:= \rho\circ u$ on the open subset $u^{-1} (\partial W \times ]\rho_0,+\infty[)$ of $\R \times \T$. 

By (\ref{dipes}), (\ref{trasve}), and the coercivity guaranteed by the assumption (b), the assumptions of Proposition \ref{etalim} are fulfilled (possibly with a smaller $\theta>0$).   
Up to the choice of a larger $\rho_0$, we may assume that the compact set $V_{\epsilon_0}$ appearing in that proposition is disjoint from $[0,1] \times (\partial W \times [\rho_0,+\infty[)$.

By Lemma \ref{luno}, there is a number $\delta_0 = \delta(\epsilon_0) >0$ such that $(s,u(s,t))$ is in the compact set $V_{\epsilon_0}$, for every $s$ such that $\|\nabla \mathbb{A}_{H_s}(v(s))\|\leq \delta_0$ and every $t\in \T$. By our assumption 
on $\rho_0$, $\rho(s,t)\leq \rho_0$ for every such $s$ and every $t\in \T$.

The length of any interval $I\subset \R$ on which $\|\nabla \mathbb{A}_{H_s}(v(s))\|\geq \delta_0$ is not larger than $E/\delta^2_0$. Therefore, every connected component $\Omega$ of $u^{-1} (\partial W \times ]\rho_0,+\infty[)$ is contained in $I\times \T$, for some interval $I$ of length not exceeding $E/\delta^2_0$. It is enough to find a uniform upper bound for $\rho$ on $\Omega$.

By the identity (\ref{laplarho}) of Lemma \ref{conti}, the function $\rho$ 
satisfies the elliptic differential inequality 
\[
\Delta \rho  + b \frac{\partial \rho}{\partial s}  \geq f \quad \mbox{on } 
\Omega, 
\]
where
\begin{eqnarray*}
b(s,t) := - \eta(s) \rho(s,t) \frac{\partial^2 h}{\partial \rho^2} (s,\rho(s,t)), \quad f := f_1 + f_2, \\ f_1(s,t) := \rho(s,t) \frac{\partial h}{\partial \rho}(s,\rho(s,t)) \int_{\T} H_s(u(s,t))\, d\tau, \quad f_2(s,t) := \eta(s) \rho(s,t) \frac{\partial^2 h}{\partial s \partial \rho} (s,\rho(s,t)).
\end{eqnarray*}
Therefore, the Aleksandrov integral version of the weak maximum principle (see e.g.\ Theorem 9.1 in \cite{gt83}) implies that
\[
\sup_{\Omega} \rho \leq \sup_{\partial \Omega} \rho + C \|f^-\|_{L^2(\Omega)} = \rho_0 + C \|f^-\|_{L^2(\Omega)},
\]
where $f^-$ denotes the negative part of the function $f$ and $C$ depends on the diameter of $\Omega$ and on the $L^2$ norm of the coefficient $b$ on $\Omega$. Actually, Theorem 9.1 in \cite{gt83} is stated for domains in $\R^n$, but the case of domains in the cylinder $\R \times \T$ is easily deduced, by considering the conformal change of variables $(s,t) \mapsto e^{2\pi(s+ i t)} \in \C \cong \R^2$. See the discussion at the beginning of the Appendix A.

Since $\Omega$ is contained in $I\times \T$ and the interval $I$ has length at most $E/\delta^2_0$, the diameter of $\Omega$ is uniformly bounded, and it is enough to find a uniform bound for the $L^2$ norm of $f^-$ and $b$ on $\Omega$. By Proposition \ref{etalim}, equation (\ref{zerob}), and the fact that $H_s \geq -\mu$, we find
\begin{eqnarray*}
2c (A,E) \geq \eta(\sup I) - \eta(\inf I) = \int_I \eta'(s)\, ds = \int_I \int_{\T} H_s(u(s,t))\, dt \, ds \\ =  \int_{(I \times \T) \setminus \Omega} H_s(u(s,t))\, ds dt + \int_{\Omega} h\bigl(s,\rho(s,t)\bigr)\, ds dt  \geq  - \mu |I \times \T|  + \int_{\Omega} h\bigl(s,\rho(s,t)\bigr)\, ds dt  \\ \geq  - \mu \frac{E}{\delta_0^2}  +   \int_{\Omega} h(s,\rho(s,t))\, ds dt .
\end{eqnarray*}
Therefore, the integral of $h(s,\rho(s,t))$ over $\Omega$ has a uniform upper bound, and by the assumption (b) we get a uniform bound for the $L^1$ norm of $\rho$ on $\Omega$, 
\begin{equation}
\label{two}
\| \rho \|_{L^1(\Omega)} \leq d,
\end{equation}
for some number $d=d(A,E)$. By the inequality (\ref{diffineq}) of Lemma \ref{conti}, 
\begin{equation}
\label{three}
|\nabla \rho|^2 \leq \rho \left| \frac{\partial u}{\partial s} \right|^2_t \quad \mbox{on } \Omega.
\end{equation}
We claim that (\ref{two}), (\ref{three}), and the energy estimate
\begin{equation}
\label{four}
\int_{\R \times \T}  \left| \frac{\partial u}{\partial s} \right|^2_t \, ds dt \leq 
\int_{-\infty}^{+\infty} \|\nabla \mathbb{A}_{H_s}(v(s))\|^2\, ds \leq E,
\end{equation}
imply that for every $p<+\infty$ the $L^p$ norm of $\rho$ on $\Omega$ is uniformly bounded. 

In order to prove this fact, we choose a cut-off function $\chi\in C^{\infty}(\R)$ such that
\begin{eqnarray*}
\chi(\sigma) = \sigma & \quad & \forall \sigma \geq \rho_0 + 3, \\
\chi(\sigma) = \rho_0 + 2 & \quad & \forall \sigma \leq \rho_0 +1, \\
\chi(\sigma) \geq \sigma  \mbox{ and } 0\leq \chi'(\sigma) \leq 1 & \quad & \forall \sigma\in \R.
\end{eqnarray*}
Then the function 
\[
\tilde{\rho}(s,t) := \left\{ \begin{array}{ll} \chi(\rho(s,t)) & \mbox{if } (s,t)\in \Omega, \\ \rho_0 + 2 & \mbox{if } (s,t)\in (I\times \T)\setminus \Omega, \end{array} \right.
\]
is smooth on the whole $I \times \T$, and its $L^p$ norm on $I \times \T$ is uniformly bounded if and only if the $L^p$ norm of $\rho$ on $\Omega$ is uniformly bounded. Moreover, by (\ref{three}), the inequality
\begin{equation}
\label{five}
|\nabla \tilde{\rho}|^2 = \chi'(\rho)^2 |\nabla \rho|^2 \leq \rho \left| \frac{\partial u}{\partial s} \right|_t^2 \leq \tilde{\rho}   \left| \frac{\partial u}{\partial s} \right|_t^2
\end{equation}
holds on $\Omega$. Since it trivially holds also on $(I\times \T)\setminus \Omega$, because $\nabla \tilde{\rho}=0$ there, we conclude that (\ref{five}) holds on the whole $I\times \T$. Then, (\ref{five}), (\ref{four}), and the H\"older inequality imply that for every $q\in [1,+\infty[$,
\[
\| \nabla \tilde{\rho} \|_{L^{2q/(q+1)}(I \times \T)}^2 \leq \|\tilde{\rho}
\|_{L^q(I\times \T)} \left\| \frac{\partial u}{\partial s} \right\|_{L^2(I\times \T)}^2 \leq E  \|\tilde{\rho}\|_{L^q(I\times \T)}.
\]
Together with the continuity of the Sobolev embedding $W^{1,2q/(q+1)}(I \times \T) \hookrightarrow L^{2q} (I \times \T)$, we deduce that $\|\tilde{\rho}\|_{L^{2q} (I \times \T)}$ has an upper bound in terms of  
$\|\tilde{\rho}\|_{L^q (I \times \T)}$ (here we are using the fact that $I\times \T$ is a regular domain and that the length of $I$ is uniformly bounded). 
By (\ref{two}), $\|\tilde{\rho}\|_{L^1 (I \times \T)}$ is uniformly bounded, so by bootstrap we conclude that $\|\tilde{\rho}\|_{L^p (I \times \T)}$  uniformly bounded for every $p<+\infty$. Hence, the same is true for $\|\rho\|_{L^p (\Omega)}$, as claimed.

Since $H_s\geq - \mu$ and $\partial h/\partial \rho\geq 0$, it holds
\begin{equation}
\label{negpart}
f_1^- \leq \mu \rho \frac{\partial h}{\partial \rho} (s,\rho).
\end{equation}
Moreover, using again the bound on $\eta$ proved in Proposition \ref{etalim}, 
\begin{equation}
\label{absv}
|f_2| \leq c(A,E)  \rho \left| \frac{\partial^2 h}{\partial s \partial \rho} (s,\rho) \right|
\end{equation}
By (\ref{negpart}) and (\ref{absv}), we have the following upper bound on the negative part of $f$:
\[
(f^-)^2 \leq (f_1^- + f_2^-)^2 \leq 2 (f_1^-)^2 + 2 |f_2|^2 \leq 2 \mu^2 \rho^2 \left| \frac{\partial h}{\partial \rho} \right|^2 + 2 c(A,E)^2 \rho^2 \left| \frac{\partial^2 h}{\partial s \partial \rho} \right|^2.
\]  
By the assumption (c), the latter quantity is bounded from above by $\kappa' \rho^p$, for $\kappa'$ and $p$ large enough. Therefore, the fact that the $L^p$ norm of $\rho$ on $\Omega$ has a uniform upper bound implies that the same is true for the $L^2$ norm of $f^-$ on $\Omega$.
Similarly, 
\[
|b|^2 \leq c(A,E)^2 \rho^2 \left| \frac{\partial^2 h}{\partial \rho^2} \right|^2,
\]
so the assumption (c) guarantees that also the $L^2$ norm of $b$ on $\Omega$ has a uniform upper bound. This concludes the proof.
\end{proof}       \qed

The assumption (b) requires $h(s,\rho)$ to grow at least linearly in $\rho$. However,  the conclusion of the above proposition holds also if we replace the assumptions (b) and (c) by the assumptions:
\renewcommand{\labelenumi}{(\theenumi')}
\begin{enumerate}
\setcounter{enumi}{1}
\item $\inf \set{h(s,\rho)}{\rho\geq \rho_0, \; s\in \R} >0$,
\item $| \frac{\partial h}{\partial \rho}|^2 \leq \kappa ( h + 1)$, $\rho^2 |\frac{\partial^2 h}{\partial \rho^2}|^2 \leq \kappa (h+1)$, and $|\frac{\partial^2 h}{\partial s \partial \rho} |^2 \leq \kappa (h+1)$.
\end{enumerate}
\renewcommand{\labelenumi}{(\theenumi)}
In the $s$-independent case, the conditions (a), (b'), and (c') are fulfilled, for instance, by any $h=h(\rho)$ which is a polynomial of degree at most two with positive leading coefficient.

The proof is actually simpler, because one does not need to find bounds for the $L^p$ norm of $\rho$. One uses the function $r:= \log \rho$ on $\Omega$, which by (\ref{laplarho}) and (\ref{diffineq}) satisfies the elliptic differential inequality
\begin{equation}
\label{laplar}
\begin{split}
\Delta r = \frac{1}{\rho} \Delta \rho - \frac{1}{\rho^2} |\nabla \rho|^2 = \frac{1}{\rho} \left| \frac{\partial u}{\partial s} \right|_t^2 + \eta \frac{\partial^2 h}{\partial \rho^2} \frac{\partial \rho}{\partial s} + \frac{\partial h}{\partial \rho} \int_{\T} H_s(u)\, dt + \eta  \frac{\partial^2 h}{\partial s \partial \rho}  - \frac{1}{\rho^2} |\nabla \rho|^2 \\
\geq \eta \rho \frac{\partial^2 h}{\partial \rho^2} \frac{\partial r}{\partial s} + \frac{\partial h}{\partial \rho} \int_{\T} H_s(u)\, dt + \eta  \frac{\partial^2 h}{\partial s \partial \rho}.
\end{split} \end{equation}
The upper bound on the integral of $h(s,\rho)$ and the assumptions (a), (b'), (c') immediately imply the $L^2$ bounds needed in order to apply the Aleksandrov maximum principle to the function $r$. This proves the following:

\begin{prop}
\label{ulim2}
The same conclusion of Proposition \ref{ulim} holds, if one replaces the assumptions (b) and (c) by the assumptions (b') and (c') above.  
\end{prop}

Conditions (a), (b'), and (c') are satisfied by the natural homotopies which join a constant $h(0,\cdot)$ to a linear, or to a quadratic $h(1,\cdot)$. Notice, however, that in the latter case the action and energy estimates - which here are taken as hypotheses - may be delicate to achieve, see Lemma \ref{nonaut} below.

\paragraph{Estimates for solutions on half-lines.} Our definition of the homomorphisms between the Morse chain complex and differential complex of the geodesic energy functional and the Rabinowitz-Floer chain complex is based on the study of spaces of solutions $v=(u,\eta)$ of the Rabinowitz-Floer equation (\ref{rfleq}) on the half-lines $\R^+= [0,+\infty[$ and $\R^-= ]-\infty,0]$, with suitable boundary conditions at $s=0$. The required estimate on $\eta$ was discussed in Remark \ref{etabdry}. The aim of this section is to prove an $L^{\infty}$ estimate on $u$ in this situation. We need such an estimate only in the case of Hamiltonians which do not depend on $s$, but it does no harm to state it under the more general assumptions of Propositions \ref{ulim} and \ref{ulim2}:

\begin{prop}
\label{ulim+}
Assume that $J_t$ is of contact type on $\partial W \times [\rho_0,+\infty[$, and that the smooth family of Hamiltonians $\{H_s\}_{s\geq 0}$ satisfies the assumptions of Proposition \ref{ulim} or of Proposition \ref{ulim2}. Then for every triplet of positive numbers $A,E,\nu$, there exists a compact set $K=K(A,E,\nu)\subset \hat{W}$ such that for every solution $v=(u,\eta)$ of (\ref{zeroa}-\ref{zerob}) on the half-line $\R^+ $ such that
\begin{eqnarray}
\label{c1+}
\sup_{s\geq 0} |\mathbb{A}_{H_s}(v(s))| & \leq & A,\\ \label{c2+}
\int_0^{+\infty} \|\nabla \mathbb{A}_{H_s}(v(s)) \|^2 \, ds & \leq & E, \\
\label{c3+}
\alpha \Bigl( \frac{\partial u}{\partial t}(0,t) - \eta(0) X_{H_0} (u(0,t)) \Bigr) & \leq & \nu, \quad \forall t\in \T \mbox{ such that } u(0,t)\in \partial W \times [\rho_0,+\infty[, \\ \label{c4+}
\eta(0) & \geq & - \nu,
\end{eqnarray}
 $u(s,t)$ belongs to $K$ for every $(s,t)\in \R \times \T$.
\end{prop}

The proof of the above proposition is based on the following version of the Aleksandrov weak maximum principle for boundary conditions of mixed Dirichlet-Neumann type:

\begin{thm}
\label{alek}
Let $\Omega$ be a bounded open subset of the half-cylinder $]0,+\infty[ \times \T$ and consider the following partition $\{\Sigma,\Sigma'\}$ of $\partial \Omega$:
\[
\Sigma := \overline{\partial \Omega \setminus ( \{0\} \times \T)}, \quad \Sigma' := \partial \Omega \setminus \Sigma.
\]
Then for every $b \in L^2(\Omega,\R^2)$ there exists a number $C$ depending only on the diameter of $\Omega$ and on $\|b\|_{L^2(\Omega)}$ such that for every $f\in L^1_{\mathrm{loc}}(\Omega)$ and every $u\in C^2(\Omega) \cap C^1(\overline{\Omega})$ which satisfies
\begin{eqnarray*}
\Delta u + b \cdot \nabla u & \geq & f \quad \mbox{in }\Omega, \\
\frac{\partial u}{\partial s} & \geq & 0 \quad \mbox{on } \Sigma',
\end{eqnarray*}
there holds
\[
\sup_{\Omega} u \leq \sup_{\Sigma} u + C \|f^-\|_{L^2(\Omega)},
\]
where $f^-$ denotes the negative part of $f$.
\end{thm}

Here, $C^1(\overline{\Omega})$ denotes the space of functions in $C^1(\Omega)$ whose differential extends continuously to $\overline{\Omega}$.
Notice that if $(s,t)\in \Sigma'$, then $s=0$ and there exists $\epsilon>0$ such that the segment $]0,\epsilon[ \times \{t\}$ is contained in $\Omega$. In particular, any $u\in C^1(\overline{\Omega})$ has a well defined  partial derivative $\partial u/\partial s$ at such a point $(0,t)$. The above theorem is probably known, but since we could not find an appropriate reference, we prove it in the Appendix A.

\medskip

\begin{proof}[of Proposition \ref{ulim+}]
By the lower bound (\ref{c4+}) on $\eta(0)$, we deduce that $|\eta|$ is uniformly bounded, see Remark \ref{etabdry}. 
Arguing as in the proof of Proposition \ref{ulim}, we find that, up to the choice of a larger $\rho_0$, each connected component $\Omega$ of the set $u^{-1}(\partial W \times ]\rho_0,+\infty[)$ is contained in $I \times \T$, for some interval $I\subset \R^+$ of length not exceeding the number $E/\delta_0^2$. When the infimum of $I$ is positive, one can proceed as in the proof of Proposition \ref{ulim} and find a uniform bound for $\rho=\rho\circ u$ on $\Omega$, by the classical Aleksandrov maximum principle, or as in the proof of Proposition \ref{ulim2}, using the function $r:= \log \rho$. 

Therefore, we just have to consider the case $I=[0,S]$, for some $0<S\leq  E/\delta_0^2$. It is convenient to work with the function $r:= \log \rho$ on $\Omega$, which by (\ref{laplar}) satisfies the elliptic differential inequality
\begin{equation}
\label{ddii}
\Delta r + b \frac{\partial r}{\partial s} 
\geq f \quad \mbox{on } \Omega,
\end{equation}
where
\[
b := - \eta \rho \frac{\partial^2 h}{\partial \rho^2}, \quad 
f := \frac{\partial h}{\partial \rho} \int_{\T} H_s(u)\, dt + \eta \frac{\partial^2 h}{\partial s \partial \rho}.
\] 
The proof of Proposition \ref{ulim} or \ref{ulim2} shows that $b$ and $f^-$ have uniformly bounded $L^2$ norm on $\Omega$. If we set
\[
\Sigma := \overline{\partial \Omega \setminus ( \{0\} \times \T)}, \quad \Sigma' := \partial \Omega \setminus \Sigma,
\]
by definition of $\Omega$ we get that
\[
r = \log \rho_0 \quad \mbox{on } \Sigma.
\]
By the equations (\ref{quindici}), (\ref{radham}), and by the assumption (\ref{c3+}), for every $(0,t)\in \Sigma'$ we have the estimate
\[
\frac{\partial r}{\partial s} (0,t) = \frac{1}{\rho(0,t)} \frac{\partial \rho}{\partial s} (0,t) = - \alpha \left( \frac{\partial u}{\partial t} (0,t) - \eta(0) X_{H_0} (u(0,t)) \right) \geq - \nu.
\]
Therefore, the function
\[
w(s,t) := r(s,t) + \nu s,
\]
satisfies 
\[
\frac{\partial w}{\partial s} (0,t) \geq 0, \quad \forall (0,t)\in \Sigma',
\]
and by (\ref{ddii}),
\[
\Delta w + b \frac{\partial w}{\partial s} \geq f + \nu b.
\]
By applying Theorem \ref{alek} to the function $w$, we obtain
\[
\sup_{\Omega} w \leq \sup_{\Sigma} w + C \|(f+\nu b)^-\|_{L^2(\Omega)} \leq \log \rho_0 + \nu S +  C \|(f+\nu b)^-\|_{L^2(\Omega)},
\]
where the number $C$ depends on $S$ and on $\|b\|_{L^2(\Omega)}$, and both these quantities are uniformly bounded. Since
\[
(f+\nu b)^- \leq f^- + \nu |b|,
\]
we conclude that $w$, and hence $r$, has a uniform upper bound on $\Omega$.
\end{proof} \qed

\begin{rem}
\label{ubdry}
The analogous statement for solutions on the half-line $\R^-$ requires the bounds (\ref{c3+}) and (\ref{c4+}) in the hypotheses to be replaced by
\begin{eqnarray*}
\alpha \Bigl( \frac{\partial u}{\partial t}(0,t) - \eta(0) X_{H_0} (u(0,t)) \Bigr) & \geq & -\nu, \quad \forall t\in \T \mbox{ such that } u(0,t)\in \partial W \times [\rho_0,+\infty[ \\
\eta(0) & \leq & \nu.
\end{eqnarray*}
\end{rem}

\begin{rem}
When $\hat{W}$ is the cotangent bundle of a Riemannian manifold $(M,g)$ and the almost complex structures $J_t$ are $C^0$-close to the Levi-Civita almost complex structure induced by $g$ (which is not of contact type), there is a way of proving $L^{\infty}$ estimates for solutions of the Floer equation involving an asymptotically quadratic Hamiltonian (not necessarily depending only on $\rho$ outside of a compact set) which does not use the maximum principle, see \cite{as06}. It is based on embedding $(M,g)$ isometrically into an Euclidean space and combining the Calderon-Zygmund estimates for the Cauchy-Riemann operator with suitable interpolation inequalities. This method could be probably used also for the Rabinowitz-Floer equation. We also recall that the $L^{\infty}$ estimates for the Floer equation, in the case of a Hamiltonian depending only on $\rho$ outside of a compact set, follow directly from the standard maximum principle, because the term with the integral does not appear in (\ref{laplarho}).  
\end{rem}

\paragraph{Energy estimates.} The $L^{\infty}$ estimates of Propositions \ref{etalim}, \ref{ulim}, and \ref{ulim2} require a priori bounds on the energy and on the action. The energy identity for a solution on $v=(u,\eta)$ of the $s$-dependent Rabinowitz-Floer equation (\ref{zeroa}-\ref{zerob}) is
\begin{equation}
\label{uno}
\begin{split}
\int_{[a,b] \times \T} \Bigl| \frac{\partial u}{\partial s} \Bigr|_t^2 \, ds\, dt + \int_a^b |\eta'(s)|^2 \, ds = \int_a^b \|v'(s)\|^2\, ds = \int_a^b \| \nabla \mathbb{A}_{H_s} (v(s)) \|^2\, ds \\ = \mathbb{A}_{H_a} (v(a)) - \mathbb{A}_{H_b} (v(b)) - \int_a^b \eta(s) \int_{\T} \frac{\partial H_s}{\partial s} (u)\, dt \, ds,
\end{split}
\end{equation}
for every $-\infty < a \leq b < +\infty$. When the Hamiltonian $H$ does not depend on $s$, the function $\mathbb{A}_H(v(s))$ is decreasing and
\begin{equation}
\label{energy}
\int_{-\infty}^{+\infty} \|\nabla \mathbb{A}_H (v(s)) \|^2 \, ds = \lim_{s\rightarrow -\infty} \mathbb{A}_H(v(s)) - \lim_{s\rightarrow +\infty} \mathbb{A}_H(v(s)),
\end{equation}
so an upper bound of the action at $+\infty$ and a lower bound at $-\infty$ imply bounds both for the action at every $s$ and for the energy. 

In the $s$-dependent case this need not be true any more. The usual arguments from symplectic homology based on choosing Hamiltonians with $\partial H/\partial s \geq 0$ (see \cite{fh94} and \cite{cfh95}) cannot be applied, because in (\ref{uno}) there is a term $\eta(s)$ in front of the $s$-derivative of $H$, and the sign of $\eta$ might vary.  In \cite{cf09}, K.\ Cieliebak and U.\ Frauenfelder have proved energy estimates for homotopies when the Hamiltonians are constant outside a compact subset of $\hat{W}$ (see the proof of Corollary 3.7 in \cite{cf09}).  

Here we need to consider the case of coercive Hamiltonians. More precisely, we shall prove a technical lemma which provides energy estimates for certain homotopies of Hamiltonians in the class of all smooth functions $H$ satisfying
\begin{equation}
\label{suplin}
\lambda(X_H) - H + \theta_0 \geq \theta_1 |Y|_t^2,
\end{equation}
for some numbers $\theta_0\geq 0$ and $\theta_1>0$. Notice that the above condition, together with the request that $H$ is bounded from below, implies that $H$ has superlinear growth in the variable $\rho$:
\[
\lim_{\rho \rightarrow +\infty} \frac{H(x,\rho)}{\rho} = +\infty, \quad \mbox{uniformly in } x\in \partial W.
\]
This fact can be seen by rewriting (\ref{suplin}) on $\partial W \times ]0,+\infty[$ in an equivalent way, as the differential inequality
\[
\frac{\partial H}{\partial \rho} \geq \frac{1}{\rho} (H - \theta_0) + \theta_1.
\]

\begin{lem}
\label{nonaut}
Assume that $\{H_s\}_{s\in \R}$ is a smooth family of Hamiltonians on $\hat{W}$ which satisfies (\ref{dipes}) and the conditions:
\begin{enumerate}
\item There is a positive number $\mu$ such that $H_s(w) \geq - \mu$ for every $s\in \R$ and $w\in \hat{W}$.
\item There is a positive number $\epsilon_0$ such that the set
\[
V_{\epsilon_0} := \set{(s,w)\in [0,1]\times \hat{W}}{|H_s(w)| \leq \epsilon_0}
\]
is compact, and 
\[
\lambda\bigl(X_{H_s}(w)\bigr) - H_s(w) \geq \frac{1}{2}, \quad \forall (s,w) \in V_{\epsilon_0}.
\]
\item There are numbers $\theta_0 \geq 0$ and $\theta_1>0$ such that 
\[
\lambda (X_{H_s}(w)) - H_s(w) + \theta_0 \geq \theta_1 |Y(w)|_t^2,
\] 
for every $s\in \R$, $t\in \T$, and $w\in \hat{W}$.
\item There are numbers $\theta_2\geq 0$ and $\theta_3\geq 0$ such that 
\[
\left|\frac{\partial H_s}{\partial s} (w)\right| \leq \theta_2 H_s(w) + \theta_3,
\]
for every $s\in \R$ and $w\in \hat{W}$.
\item For every $s\in \R$ and $w\in \hat{W}$ there holds
\[
\left| \frac{\partial H_s}{\partial s} (w)\right| \leq \epsilon \bigl(  \lambda (X_{H_s}(w)) - H_s(w) + \theta_0 \bigr),
\]  
where $\theta_0$ is the constant appearing in the assumption (c), and
\begin{equation}
\label{epsi}
\epsilon < \left( 1 +  2\theta_0 + \frac{\theta_0 \mu}{\delta_0^2}  \right)^{-1},
\end{equation}
where $\delta_0 = \delta(\epsilon_0)$ is the number produced by Lemma \ref{luno}, and $\mu$ is the number appearing in the assumption (a).
\end{enumerate}
Then, for every number $M$ there exists a constant $a(M)$ such that for every solution $v=(u,\eta)$ of (\ref{zeroa}-\ref{zerob}) with
\begin{equation}
\label{tre}
\lim_{s\rightarrow -\infty}
\mathbb{A}_{H_0}(v(s)) \leq M, \quad \lim_{s\rightarrow +\infty}\mathbb{A}_{H_1} (v(s)) \geq -M,
\end{equation}
there holds
\begin{eqnarray*}
|\mathbb{A}_{H_s} (v(s)) | \leq a(M), \quad \forall s\in \R, \\
\int_{-\infty}^{+\infty} \|\nabla \mathbb{A}_{H_s} (v(s)) \|^2 \, ds  =  \int_{\R \times \T} \Bigl| \frac{\partial u}{\partial s} \Bigr|_t^2 \, ds\, dt + \int_{-\infty}^{+\infty} |\eta'(s)|^2 \, ds \leq a(M)+M.
\end{eqnarray*}
\end{lem}

\begin{proof}
From the energy identity (\ref{uno}) and from the assumptions (\ref{dipes}) and (\ref{tre}), we deduce the estimate
\begin{equation}
\label{sei}
\begin{split}
 \int_{\R \times \T} \Bigl| \frac{\partial u}{\partial s} \Bigr|_t^2 \, ds\, dt + \int_{-\infty}^{+\infty} |\eta'(s)|^2 \, ds = \int_{-\infty}^{+\infty} \| \nabla \mathbb{A}_{H_s} (v(s)) \|^2 \, ds  \\ \leq 2M + \int_{[0,1] \times \T} |\eta(s)| \left| \frac{\partial H_s}{\partial s} (u) \right| \, ds \, dt.
 \end{split}
\end{equation}
Moreover, the same assumptions imply the inequalities
\[
\mathbb{A}_{H_0} (v(s)) \leq M, \quad \forall s\leq 0, \quad \mathbb{A}_{H_1} (v(s)) \geq -M, \quad \forall s\geq 1.
\]
These inequalities and the identities
\begin{eqnarray*}
\mathbb{A}_{H_s} (v(s))  & = & \mathbb{A}_{H_0}(v(0)) - \int_0^s \| \nabla \mathbb{A}_{H_{\sigma}} (v(\sigma)) \|^2 \, d\sigma - \int_0^s \eta(\sigma) \int_{\T} \frac{\partial H_{\sigma}}{\partial \sigma} (u)\, dt \, d\sigma \\
& = & \mathbb{A}_{H_1}(v(1)) + \int_s^1 \| \nabla \mathbb{A}_{H_{\sigma}} (v(\sigma)) \|^2 \, d\sigma + \int_s^1 \eta(\sigma) \int_{\T} \frac{\partial H_{\sigma}}{\partial \sigma} (u)\, dt \, d\sigma,
\end{eqnarray*}
imply the action bound
\begin{equation}
\label{sette}
|\mathbb{A}_{H_s} (v(s))| \leq M + \int_{[0,1] \times \T} |\eta(\sigma)| \left| \frac{\partial H_{\sigma}}{\partial \sigma} (u) \right| \, d\sigma \,dt, \quad \forall s\in \R.
\end{equation}
The estimates (\ref{sei}) and (\ref{sette}) show that we must provide a uniform upper bound for the quantity
\[
\int_{[0,1] \times \T} |\eta(s)| \left| \frac{\partial H_s}{\partial s} (u) \right| \, ds\, dt.
\]
We start by proving an upper bound for $|\eta|$ in terms of the quantity above.
By Lemma \ref{luno} and (\ref{sette}), if $s\in \R$ is such that $\|\nabla \mathbb{A}_{H_s} (v(s)) \| \leq \delta_0$ then
\begin{equation}
\label{otto-}
\begin{split}
|\eta(s)| \leq 2 \bigl( |\mathbb{A}_{H_s}(v(s))| + \kappa_0 \| \nabla \mathbb{A}_{H_s} (v(s)) \| \bigr) \\ \leq  2 \left( M +  \int_{[0,1] \times \T} |\eta(\sigma)| \left| \frac{\partial H_{\sigma}}{\partial \sigma} (u) \right| \, d\sigma \,dt + \kappa_0 \delta_0 \right).
\end{split} \end{equation}
If $I\subset \R$ is a maximal interval on which $\|\nabla \mathbb{A}_{H_s} (v(s)) \| > \delta_0$, then (\ref{sei}) implies that
\begin{equation}
\label{misu}
|I| \leq \frac{1}{\delta_0^2} \left( 2M + \int_{[0,1] \times \T} |\eta(s)| \left| \frac{\partial H_s}{\partial s} (u) \right| \, ds \, dt \right).
\end{equation}
By the equation (\ref{zerob}), we have the identities
\[
\eta(s) = \eta(\inf I) + \int_{\inf I}^s \int_{\T} H_{\sigma} (u) \, dt\, d\sigma = \eta(\sup I) - \int_s^{\sup I} \int_{\T} H_{\sigma} (u) \, dt\, d\sigma.
\]
These identities, together with the assumption (a), (\ref{otto-}), and (\ref{misu}), imply that for every $s\in I$ there holds
\begin{equation*}
\begin{split}
 |\eta(s)| & \leq \max\{ |\eta(\inf I)|, |\eta(\sup I)| \} + \mu |I| \\ \leq   2  \left( M +  \int_{[0,1] \times \T} |\eta(\sigma)| \left| \frac{\partial H_{\sigma}}{\partial \sigma} (u) \right| \, d\sigma \, dt + \kappa_0 \delta_0 \right) & + \frac{\mu}{\delta_0^2} \left( 2M + \int_{[0,1] \times \T} |\eta(\sigma)| \left| \frac{\partial H_\sigma}{\partial \sigma} (u) \right| \, d\sigma \,dt \right).
\end{split} \end{equation*}
We conclude that $\eta$ is bounded and that 
\begin{equation}
\label{otto} 
\sup_{s\in \R} |\eta(s)| \leq d + \left( 2 + \frac{\mu}{\delta_0^2} \right) 
\int_{[0,1] \times \T} |\eta(\sigma)| \left| \frac{\partial H_\sigma}{\partial \sigma} (u) \right| \, d\sigma \, dt,
\end{equation}
where
\[
d = d(M) := 2M + 2 \kappa_0 \delta_0 + \frac{2 M \mu}{\delta_0^2}.
\] 
It is useful to introduce the function
\[
\theta(s,t) := \lambda\bigl(X_{H_s}(u(s,t))\bigr) - H_s(u(s,t)) + \theta_0,
\]
which by the assumption (c) is non-negative. By the identity $\lambda = \iota_Y \omega$ and by (\ref{zeroa}), we have the chain of equalities
\begin{eqnarray*}
\lambda \Bigl( \frac{\partial u}{\partial t} \Bigr) - \eta(s) H_s (u) = \lambda \Bigl( \frac{\partial u}{\partial t} - \eta(s) X_{H_s} (u) \Bigr) + \eta(s) \bigl( \lambda(X_{H_s}(u)) - H_s(u) \bigr) \\ = \omega\Bigl( Y(u), \frac{\partial u}{\partial t} - \eta(s) X_{H_s} (u) \Bigr) + \eta(s) \theta(s,t) - \theta_0 \eta(s) \\
= \Bigl\langle Y(u), J_t \Bigl(  \frac{\partial u}{\partial t} - \eta(s) X_{H_s} (u) \Bigr) \Bigr\rangle_t + \eta(s) \theta(s,t) - \theta_0 \eta(s) \\ = \Bigl\langle Y(u), \frac{\partial u}{\partial s} \Bigr\rangle_t + \eta(s) \theta(s,t)- \theta_0 \eta(s).
\end{eqnarray*}
Integration over $\T$ produces the identity
\[
\mathbb{A}_{H_s} (v(s)) = -\int_{\T}  \Bigl\langle Y(u), \frac{\partial u}{\partial s} \Bigr\rangle_t \, dt + \eta(s) \int_{\T} \theta(s,t)\, dt - \theta_0 \eta(s),
\]
from which we deduce the inequality
\[
|\eta(s)| \int_{\T} \theta(s,t)\, dt \leq \bigl| \mathbb{A}_{H_s}(v(s)) \bigr| + \int_{\T} \Bigl| \Bigl\langle Y(u), \frac{\partial u}{\partial s} \Bigr\rangle_t \Bigr| \, dt + \theta_0 |\eta(s)|.
\]
The last integrand can be estimated as
\[
\Bigl| \Bigl\langle Y(u), \frac{\partial u}{\partial s} \Bigr\rangle_t \Bigr| \leq \frac{1}{2\delta} |Y(u)|^2_t + \frac{\delta}{2} \Bigl| \frac{\partial u}{\partial s} \Bigr|^2_t,
\]
for every $\delta>0$. Together with (\ref{sette}), the above two inequalities imply that for every $s$ there holds
\begin{equation}
\label{nove}
\begin{split}
|\eta(s)| \int_{\T} \theta(s,t)\, dt \\ \leq M + \int_{[0,1]\times \T} |\eta(\sigma)| \left| \frac{\partial H_{\sigma}}{\partial \sigma} (u) \right| \, d\sigma\, dt + \frac{1}{2\delta} \int_{\T}   |Y(u)|^2_t  \, dt + \frac{\delta}{2} \int_{\T}  \left| \frac{\partial u}{\partial s} \right|^2_t \, dt + \theta_0 |\eta(s)|.
\end{split}
\end{equation}
Let $c$ be a positive number. By (\ref{zerob}) and the assumption (a), the function $s\mapsto \eta(s) + \mu s$ is increasing, so the set
\[
I := \set{s\in [0,1]}{c < \eta(s) + \mu s < c+\mu}
\]
is a (possibly empty) interval. By the definition of $I$,
\begin{eqnarray}
\label{dieci}
c-\mu < \eta(s) < c+ \mu, \quad \forall s\in I, \\
\label{undici}
|\eta(s)| \geq c, \quad \forall s\in I^c,
\end{eqnarray}
where $I^c$ denotes the complement of $I$ in $[0,1]$. By (\ref{dieci}), (\ref{zerob}), and the assumption (d),
\[
2\theta_2 \mu \geq \theta_2 \bigl( \eta(\sup I) - \eta(\inf I)  \bigr)= \theta_2 \int_{I} \eta'(s)\, ds = \theta_2 \int_{I \times \T} H_s(u)\, ds \, dt \geq  \int_{I \times \T} \Bigl| \frac{\partial H_s}{\partial s} (u) \Bigr| \, ds \, dt - \theta_3.
\]
By (\ref{dieci}), $|\eta| < c+\mu$ on $I$, so the above inequality implies the estimate
\begin{equation}
\label{dodici}
\int_{I \times \T} |\eta(s)| \Bigl| \frac{\partial H_s}{\partial s} (u) \Bigr| \, ds dt \leq (c+\mu) (2\mu\theta_2 + \theta_3).
\end{equation}
By integrating (\ref{nove}) on the set $I^c$, whose measure does not exceed 1, we find
\begin{equation}
\label{princi}
\begin{split}
& \int_{I^c \times \T} |\eta(s)| \theta(s,t)  \, ds dt   \leq M + \int_{[0,1] \times \T} |\eta(s)| \left| \frac{\partial H_s}{\partial s} (u) \right| \, ds \, dt \\ & + \frac{1}{2\delta} \int_{I^c \times \T} |Y(u)|_t^2 \, dt + \frac{\delta}{2} \int_{[0,1]\times \T} \Bigl| \frac{\partial u}{\partial s} \Bigr|_t^2 \, ds\, dt + \theta_0 \sup_{s\in \R} |\eta(s)|.
\end{split}
\end{equation}
We estimate the last four terms separately, starting from the last one. By (\ref{otto}), we have
\begin{equation}
\label{princi1}
\theta_0 \sup_{s\in \R} |\eta(s)| \leq \theta_0 d(M) + \theta_0 \left( 2 + \frac{\mu}{\delta_0^2} \right) \int_{[0,1] \times \T} |\eta(s)| \left| \frac{\partial H_s}{\partial s} (u) \right| \, ds \, dt.
\end{equation}
By (\ref{sei}), we have
\begin{equation}
\label{princi2}
\frac{\delta}{2} \int_{[0,1]\times \T} \Bigl| \frac{\partial u}{\partial s} \Bigr|_t^2 \, ds \, dt \leq \delta M + \frac{\delta}{2} \int_{[0,1]\times \T} |\eta(s)| \left| \frac{\partial H_s}{\partial s} (u) \right|\, ds \, dt.
\end{equation}
By the assumption (c), we have
\begin{equation}
\label{princi3}
\frac{1}{2\delta} \int_{I^c \times \T} |Y(u)|_t^2 \, dt \leq \frac{1}{2\delta \theta_1} \int_{I^c \times \T} \theta(s,t)\, ds \,dt \leq \frac{1}{2 \delta \theta_1 c} \int_{I^c \times \T} |\eta(s)| \theta(s,t)\, ds \, dt,
\end{equation}
where we have used also (\ref{undici}).
By (\ref{dodici}) and the assumption (e),
\begin{equation}
\label{princi4}
\begin{split}
\int_{[0,1] \times \T} |\eta(s)| \left| \frac{\partial H_s}{\partial s} (u) \right|\, ds dt \leq  (c+\mu)(2\mu\theta_2 + \theta_3) + \int_{I^c \times \T} |\eta(s)| \left| \frac{\partial H_s}{\partial s} (u) \right|\, ds \, dt  \\ \leq  (c+\mu)(2\mu\theta_2 + \theta_3) + \epsilon \int_{I^c \times \T} |\eta(s)| \theta(s,t) \, ds \, dt .
\end{split}
\end{equation}
By taking (\ref{princi1}), (\ref{princi2}), (\ref{princi3}), and (\ref{princi4}) into account, (\ref{princi}) implies the estimate
\begin{equation*} \begin{split}
& \left( 1 - \frac{1}{2\delta \theta_1 c} -  \epsilon \Bigl( 1 + \frac{\delta}{2} + 2\theta_0 + \frac{\theta_0 \mu}{\delta_0^2} \Bigr) \right) \int_{I^c \times \T} |\eta(s)| \theta(s,t)\, ds \, dt \\ & \leq 
M + (c+\mu)(2\mu \theta_2 + \theta_3)\left( 1 + \frac{\delta}{2} + \theta_0 \Bigl( 2 + \frac{\mu}{\delta_0^2} \Bigr) \right) + \delta M + \theta_0 d(M).
\end{split} \end{equation*}
By the assumption (\ref{epsi}) on $\epsilon$, if we choose $\delta>0$ small enough and $c>0$ large enough, the constant in front of the integral is positive, so we have a uniform upper bound 
\[
\int_{I^c \times \T} |\eta(s)| \theta(s,t) \, ds \, dt \leq a'(M).
\]
By the assumption (e), also the integral of $|\eta| | \partial H_s/\partial s (u)|$ over $I^c \times \T$ has a uniform upper bound, and by using also (\ref{dodici}) we deduce the uniform upper bound  
\[
\int_{[0,1] \times \T} |\eta(s)| \left| \frac{\partial H_s}{\partial s} (u) \right| \, ds \, dt \leq a^{\prime\prime}(M).
\]
Then (\ref{sei}) and (\ref{sette}) imply the desired estimates, with $a(M) := M + a^{\prime\prime}(M)$.
\end{proof}
\qed

\section{The Rabinowitz-Floer complex}
\label{rfcsec}

The aim of this section is to recall the definition of the Rabinowitz-Floer complex from \cite{cf09}. The only differences are that the ambient manifold $\hat{W}$ is the completion of the Liouville domain $W$ (and not an arbitrary exact convex symplectic manifold as in \cite{cf09}) and that we use Hamiltonians which are eventually quadratic, instead of eventually constant. The effect of this choice of a different class of Hamiltonians is discussed at the end of the section.

\paragraph{The boundary homomorphism.} Let us assume that the closed Reeb orbits of $(\partial W,\alpha)$ are of {\em Morse-Bott type}, meaning that for each $T>0$ the set $\mathscr{P}_T\subset \partial W$ consisting of all $T$-periodic points of the Reeb flow is a closed submanifold, the rank of $d\alpha|_{\mathscr{P}_T}$ is locally constant, and 
\[
T_x \mathscr{P}_T = \ker (D\phi_T^R(x) - I), \quad \forall x\in \mathscr{P}_T,
\]
where $\phi_t^R$ denotes the Reeb flow. 
Such a condition is generic, for instance in the sense that the set of functions $\sigma \in C^{\infty}(\partial W,]0,+\infty[)$ such that the closed Reeb orbits of $(\graf \sigma,\lambda|_{\graf \sigma})$ are of {\em Morse-Bott type} is generic in $C^{\infty}(\partial W,]0,+\infty[)$. Actually, the stronger property that all the closed orbits of the Reeb flow are non-degenerate is generic, see Appendix B in \cite{cf09} (we recall that a $T$-periodic point $x\in \partial W$ is {\em non-degenerate} if the restriction of $D\phi_T^R(x)$ to the contact hyperplane $\ker \alpha(x)$ does not have the eigenvalue 1). Treating the Morse-Bott situation is not more difficult than treating the non-degenerate situation - in which the critical manifolds of $\mathbb{A}$ corresponding to the closed Reeb orbits are circles - so we deal with the former condition, which includes for instance geodesic flows on symmetric spaces.  

A Hamiltonian $H\in C^{\infty}(\hat{W})$ is said to be {\em compatible} with the contact manifold $(\partial W,\alpha)$ if its set of zeros is $\partial W$, and if the restriction of $X_H$ to $\partial W$ coincides with the Reeb vector field $R$. It is said to be {\em eventually quadratic} if $H(x,\rho) = h(\rho)$ for $\rho$ large, with
\[
h(\rho) = a \rho^2 + b \rho + c,
\]
for some numbers $a>0$, $b,c\in \R$. Let us fix an eventually quadratic Hamiltonian $H$ compatible with $(\partial W,\alpha)$. We shall associate a chain complex to the corresponding Rabinowitz action functional $\mathbb{A}$ by using the analogue of Frauenfelder's {\em Morse homology with cascades}, see the Appendix of \cite{fra04}.

By the Morse-Bott assumption, the set $\crit\, \mathbb{A}$ of critical points  of $\mathbb{A}$ is a smooth manifold, and $\mathbb{A}$ is constant on each of its connected components, which are finite-dimensional (but in general different components have different dimension). 
Moreover, for each $A>0$ the set
\[
\set{z\in \crit\, \mathbb{A}}{|\mathbb{A}(z)| \leq A}
\]
is compact, and the set of critical values of $\mathbb{A}$ is discrete. Furthermore, every solution $v:\R \rightarrow \Lambda \hat{W} \times \R$ of the Rabinowitz-Floer equation (\ref{rfleq}) with finite energy (\ref{energy})
converges to some $v(-\infty)\in \crit\, \mathbb{A}$ for $s\rightarrow -\infty$ and to some $v(+\infty)\in \crit\, \mathbb{A}$ for $s\rightarrow +\infty$. 

We fix a smooth Morse function $a$ and a Riemannian metric $g_{\mathbb{A}}$ on $\crit\, \mathbb{A}$, such that the negative gradient flow $\phi^{-\nabla a}_t$ of $a$ is {\em Morse-Smale}. This means that for every pair of critical points $z^-,z^+$ of $a$, the unstable manifold $W^u(z^-;-\nabla a)$ is transverse to the stable manifold $W^s(z^+;-\nabla a)$. The Morse index of $z\in \crit\, a$ is denoted by $\ind (z;a)$, so 
\[
\ind (z;a) = \dim W^u(z;-\nabla a) = \codim\, W^s(z;-\nabla a).
\]
If $z^-,z^+\in \crit\, a$ and $m\geq 1$ is an integer, a {\em flow line from $z^-$ to $z^+$ with $m$ cascades} is a $m$-tuple $(v_1,\dots,v_m)$ of {\em non-stationary} solutions of the Rabinowitz-Floer equation (\ref{rfleq}) (stationary solutions are solutions which do not depend on $s$) such that
\[
v_1(-\infty) \in W^u(z^-;-\nabla a), \quad v_m(+\infty) \in W^s(z^+;-\nabla a),
\]
and for every $j=1,\dots,m-1$
\[
v_{j+1}(-\infty) \in \phi^{-\nabla a}_{\R^+} (v_j(+\infty)).
\] 
The set of flow lines from $z^-$ to $z^+$ with $m$ cascades is denoted by $\tilde{\mathcal{M}}_m(z^-,z^+)$, and its quotient by the free action of $\R^m$ given by $s$-translations on each $v_j$ is denoted by $\mathcal{M}_m(z^-,z^+)$. 

The set of {\em flow lines from $z^-$ to $z^+$ with zero cascades}, that we denote by $\tilde{\mathcal{M}}_0(z^-,z^+)$, is the intersection $W^u(z^-;-\nabla a)\cap W^s(z^+; - \nabla a)$. When $z^-\neq z^+$, the quotient of $\tilde{\mathcal{M}}_0 (z^-,z^+)$ by the free $\R$-action is denoted by $\mathcal{M}_0(z^-,z^+)$, while in the case $z^-=z^+$ we define $\mathcal{M}_0(z^-,z^+)$ to be the empty set.
Finally,
\[
\tilde{\mathcal{M}}(z^-,z^+) := \bigcup_{m\geq 0} \tilde{\mathcal{M}}_m (z^-,z^+), \quad
\mathcal{M}(z^-,z^+) := \bigcup_{m\geq 0} \mathcal{M}_m (z^-,z^+).
\]
Since $\mathbb{A}$ is strictly decreasing on non-stationary solutions of the Rabinowitz-Floer equation (\ref{rfleq}), if $z^-$ and $z^+$ belong to the same connected component of $\crit\, \mathbb{A}$ then $\mathcal{M}_m(z^-,z^+)=\emptyset$ for every $m\geq 1$, while if $\mathcal{M}_m(z^-,z^+)\neq \emptyset$ for some $m\geq 1$, then $\mathbb{A}(z^-) > \mathbb{A}(z^+)$ and $\mathcal{M}_0(z^-,z^+)=\emptyset$. 

Assume that the almost complex structures $J_t$ are of contact type outside of a compact set. By our assumption on $H$, Propositions \ref{etalim} and \ref{ulim} (or Proposition \ref{ulim2}), together with the energy identity (\ref{energy}), imply that the elements of $\tilde{\mathcal{M}}(z^-,z^+)$ have a uniform $L^{\infty}$ bound. Since the symplectic form $\omega$ is exact, a standard non-bubbling-off argument in Floer homology guarantees that $\mathcal{M}(z^-,z^+)$ is compact up to breaking.

For a generic choice of $J_t$ and $g_{\mathbb{A}}$, the spaces $\mathcal{M}(z^-,z^+)$ are smooth finite dimensional manifolds, and their components of dimension zero are compact. In particular, the zero-dimensional part of $\mathcal{M}(z^-,z^+)$ is a finite set, and we define $n_{RF}(z^-,z^+)\in \Z_2$ to be the parity of such a set.  

The Rabinowitz-Floer $\Z_2$-vector space $RF=RF(W,\lambda,a)$ is generated by all formal sums
\begin{equation}
\label{fs}
\sum_{z\in \mathcal{Z}} z,
\end{equation}
where $\mathcal{Z}$ is a subset of $\crit\, a$ such that
\[
\sup_{z\in \mathcal{Z}} \mathbb{A}(z) < +\infty.
\]
Note that in general $\mathcal{Z}$ may be an infinite set, because $\mathbb{A}$ is not bounded from below.
The homomorphism
\[
\partial = \partial(W,\lambda,a,H,J,g_{\mathbb{A}}) : RF(W,\lambda,a) \rightarrow RF(W,\lambda,a)
\]
is defined as the unique homomorphism such that
\[
\partial z^ - = \sum_{z^+ \in \crit\, a} n_{RF} (z^-,z^+) \, z^+, \quad \forall z^- \in \crit\, a.
\]
A standard argument implies that $\partial \circ \partial = 0$, so $\{RF,\partial\}$ has the structure of a differential $\Z_2$-vector space. It also carries an $\R$ filtration given by the action level. More precisely, if $I\subset \R$ is an interval, $\partial$ induces a boundary homomorphism on the subspace
\[
RF^I = RF^I(W,\lambda,a) 
\]
consisting of all formal sums (\ref{fs}) where $\mathcal{Z}$ consists of critical points $z$ of $a$ such that $\mathbb{A}(z)\in I$. Moreover, there are natural chain maps
\[
RF^I \rightarrow RF^{I'},
\]
whenever $\inf I \leq \inf I'$ and $\sup I \leq \sup I'$.

\paragraph{Grading.}   Here we follow the $\Z$-grading convention of \cite{cfo09}, which differs from the $(\Z + 1/2)$-grading convention of \cite{cf09} by a factor $1/2$. The Rabinowitz-Floer differential vector space has a 
$\Z$-grading whenever the first Chern class $c_1(T\hat{W})$ of the complex vector bundle $(T\hat{W},J)$ vanishes, a condition which does not depend on the choice of the $\omega$-compatible almost complex structure $J$, but only on the Liouville domain $(W,\lambda)$.
In such a case, the grading on $RF$ is defined by
\begin{equation}
\label{grad}
\mu(z) = \mu(x,\eta) := \mu_{CZ}^{\tau}(x) - \frac{1}{2} \nu^{\tau}(x) + \ind(z;a),
\end{equation}
for every critical point $z=(x,\eta)$ of $a$ such that $\eta\neq 0$. Here $\mu^{\tau}_{CZ}(z)$ and $\nu^{\tau}(x)$ are the {\em transverse Conley-Zehnder 
index} and the {\em transverse nullity} of the $1$-periodic orbit $x$ of the Hamiltonian flow associated to the autonomous Hamiltonian $\eta H$. We recall that the transverse nullity $\nu^{\tau}(x)$ of $x$ is the dimension of the linear subspace 
\[
\ker \Bigl(D\phi_1^{X_{\eta H}} (x(0)) - \mathrm{id}\Bigr) \cap \ker \alpha(x(0)),
\]
so $\nu^{\tau}(x)=0$ if and only if $x$ is non-degenerate; by the Morse-Bott assumption, the connected component $K_z$ of $\crit \, \mathbb{A}$ which contains $z=(x,\eta)$ has dimension
\begin{equation}
\label{dimcrit}
\dim K_z = \nu^{\tau}(x) + 1.
\end{equation}
The transverse Conley-Zehnder $\mu^{\tau}(x)$ index is a half-integer, which 
can be defined as the Conley-Zehnder index (in the sense of Robbin and 
Salamon, \cite{rs93}) of the path in the symplectic group $\mathrm{Sp}(2n-2)$ which is obtained by conjugating the restriction $D\phi^{X_{\eta H}}_t(x(0))$ to the invariant symplectic subbundle given by the contact distribution $\ker \alpha$ by a symplectic trivialization. The fact that $c_1(T\hat{W})=0$ allows to fix a class of such trivializations for any homology class in $H_1(\hat{W},\Z)$ for which the above index is well-defined, see \cite{sei06b} or \cite{cf09} for more details. Notice also that if $y(t) = x(t/\eta)$ is the Reeb closed orbit associated to $x$, then
\[
\mu^{\tau}_{CZ}(x)= (\sgn \eta) \mu_{CZ}(y),
\] 
where $\mu_{CZ}(y)$ is the standard Conley-Zehnder index that one 
associates to closed Reeb orbits.
If $z=(x,\eta)$ is a critical point of $a$ with $\eta=0$, then $x$ is a constant loop in $\partial W$, and $\mu(z)$ is defined as
\begin{equation}
\label{grad0}
\mu(z) = \mu(x,0) := \ind (z;a) - n + 1.
\end{equation}
Notice that for such critical points, $\mu(z)$ may take all the values in the interval
\[
-n + 1 \leq \mu(z) \leq n.
\]
Moreover, if we agree that for a constant loop $x$ we have $\mu^{\tau}_{CZ}(x):=0$ and $\nu^{\tau}(x) := 2n-2$, both the identities (\ref{grad}) and (\ref{dimcrit}) hold also for $\eta=0$.

With such definitions, one can prove that
\[
\dim \mathcal{M}(z^-,z^+) = \mu(z^-) - \mu(z^+) - 1,
\]
see Proposition 4.1 in \cite{cf09}. Here and in what follows, we adopt the convention that any manifold which is declared to have negative dimension is actually empty.

If $RF_k$ denotes the subgroup of $RF$ obtained by considering formal sums as in (\ref{fs}) where all the critical points $z\in \mathcal{Z}$ have index $\mu(z)=k$, then the homomorphism $\partial$ maps $RF_k$ into $RF_{k-1}$. Therefore, $RF_*$ is a chain complex of $\Z_2$-vector spaces.     

\paragraph{Invariance.} If one changes the loop of almost complex structures (of contact type outside of a compact set) or the Riemannian metric $g_{\mathbb{A}}$ on $\crit \, \mathbb{A}$, one gets isomorphic Rabinowitz-Floer complexes.

If one changes the Morse function $a$ and the Hamiltonian $H$, on gets chain equivalent Rabinowitz-Floer complexes. This can be proved by the standard homotopy argument from Floer homology, by noticing that if $H_0$ and $H_1$ are eventually quadratic Hamiltonians compatible with $(\partial W,\alpha)$ and we define
\[
H_s := (1-s) H_0 + s H_1,
\]
then there exists a finite set of numbers $0=s_0<s_1<\dots<s_k=1$ such that for every $j\in \{1,\dots,k\}$ the $s$-dependent Hamiltonian
\[
H^j_s(w) := (1-\chi(s)) H_{s_{j-1}} (w) + \chi(s) H_{s_j} (w)
\]
satisfies the assumptions of Lemma \ref{nonaut}. Here $\chi$ is a smooth cut-off function on $\R$ such that $\chi=0$ on $\R^-$ and $\chi=1$ on $[1,+\infty[$. In particular, the {\em Rabinowitz-Floer homology}
\[
HRF = HRF (W,\lambda) := \frac{\ker \partial}{\im \partial}
\]
depends only on the Liouville domain $(W,\lambda)$. 

The same Lemma \ref{nonaut} allows also to deal with Hamiltonians whose zero-set varies with $s$. In particular, it allows to prove that if $\Sigma\subset \partial W \times ]0,+\infty[ \subset \hat{W}$ is the graph of a smooth function $\sigma: \partial W \rightarrow ]0,+\infty[$ and the contact form $\alpha_{\Sigma}:= \lambda|_{\Sigma}$ defines a Reeb vector field $R_{\Sigma}$ with the Morse-Bott property, then the corresponding Rabinowitz-Floer homology, defined by using eventually quadratic Hamiltonians $H$ which are compatible with $(\Sigma,\alpha_{\Sigma})$, does not depend on $\Sigma$. This fact allows to define the Rabinowitz-Floer homology of a Liouville domain $(W,\lambda)$ also when the Morse-Bott condition for $R=R_{\partial W}$ does not hold.

More generally, one can prove that if two Liouville domains $(W,\lambda)$ and $(W',\lambda')$ are {\em Liouville isomorphic} -- meaning that there exists a diffeomorphism $\phi:\hat{W}\rightarrow \hat{W'}$ such that $\phi^* \lambda' = \lambda + df$, for some compactly supported function $f$ on $\hat{W}$, see \cite{sei06b} -- then their Rabinowitz-Floer homologies are isomorphic.

The above arguments prove the invariance within the class of Hamiltonians which are eventually quadratic. Actually, the combined use of Lemma \ref{nonaut} and Propositions \ref{ulim} or \ref{ulim2} would allow to prove the invariance within many classes of Hamiltonian having a similar behavior at infinity. However, we do not know how to get the energy and action estimates for $s$-dependent Hamiltonians which interpolate between an eventually constant $H_0$ and an eventually quadratic $H_1$, or even an eventually linear $H_1$. This raises the question whether our Rabinowitz-Floer homology for eventually quadratic Hamiltonian coincides with the original one defined in \cite{cf09} using eventually constant Hamiltonians. Since the critical set of $\mathbb{A}_H$ does not depend on the Hamiltonian $H$ - as long as $H$ is compatible with $(\partial W,\alpha)$ - and for every critical point $z=(x,\eta)$  of $\mathbb{A}_H$ the loop $x$ takes values into $\partial W$, the behavior of $H$ outside of $W$ should not matter that much, so one suspects that the answer is affirmative. 

When using field coefficients, as we are, one can actually prove this fact arguing as in the proof of Proposition 4.13 of \cite{cfo09}, by considering the truncated complexes $RF^{[a,b]}$ and by taking limits. Indeed, in this case the whole Rabinowitz-Floer homology is naturally isomorphic to the double direct-inverse limit of the truncated homologies
\[
HRF (W,\lambda) \cong \lim_{\substack{\rightarrow \\ A_+}} \lim_{\substack{\leftarrow \\ A_-}} HRF^{[A_-,A_+]} (W,\lambda),
\]
for $A_- \downarrow -\infty$ and $A_+\uparrow +\infty$ (see \cite{cf09b}, the direct and inverse limit must be taken in this order). So it is enough to prove that eventually quadratic and eventually constant Hamiltonians determine the same truncated Rabinowitz-Floer homology $HRF^{[a,b]}$. This can be deduced from Proposition \ref{ulim2} in the following way. 

Let $H_1$ be an eventually quadratic Hamiltonian compatible with $(\partial W,\alpha)$. Since we already have the invariance within the class of eventually quadratic Hamiltonians, we can choose a particular one, for instance a Hamiltonian $H_1$ such that $H_1<0$ on $W \setminus (\partial W \times ]0,1])$ and 
\[
H_1(x,\rho) = h_1(\rho) \quad \mbox{with } h_1(\rho) = \frac{1}{2} (\rho^2 - 1), \quad \forall (x,\rho)\in \partial W \times ]0,+\infty[.
\]
Let $\chi$ be a smooth cut-off function on $\R$ such that
\[
\chi(s) = s \;\; \forall s\leq 1, \quad \chi(s) = 2 \;\; \forall s\geq 3, \quad 0 \leq \chi' \leq 1, \quad -1 \leq \chi'' \leq 0,
\]
and choose a number $\kappa$ such that
\begin{equation}
\label{kappa}
\kappa \geq \max_{s\in \R} \bigl( (2s+1)^3 \chi''(s)^2 + (2s+1) \chi'(s)^2 \bigr),
\end{equation}
for instance $\kappa = 7^3+7$.
In particular, $\kappa\geq 2$, so the function $h_1$ satisfies the assumption (c') of Proposition \ref{ulim2} with the constant $\kappa$, that is
\[
h_1'(\rho)^2 \leq \kappa ( h_1(\rho) + 1 ) , \quad \rho^2 h_1''(\rho)^2 \leq \kappa ( h_1(\rho) + 1), \quad \forall \rho>0.
\]
Let $A_-\leq A_+$ be real numbers. By Proposition \ref{ulim2}, there is a compact set $K=K(A_-,A_+)\subset \hat{W}$ with the following property: For every Hamiltonian $H$ on $\hat{W}$ which coincides with $H_0$ on the complement of $\partial W \times [2,+\infty[$ in $\hat{W}$ and such that $H(x,\rho) = h(\rho)$ on $\partial W \times ]0,+\infty[$ for some smooth function $h$ which satisfies
\begin{equation}
\label{dadim}
h'(\rho)^2 \leq \kappa ( h(\rho) + 1 ) , \quad \rho^2 h''(\rho)^2 \leq \kappa ( h(\rho) + 1), \quad \forall \rho>0,
\end{equation}
the $u$ part of every solution $v=(u,\eta)$ of the Rabinowitz-Floer equation associated to $H$ such that $\mathbb{A}_H(v(\R))\subset [A_-,A_+]$ takes values in the compact set $K$. For $\epsilon>0$, consider the eventually constant Hamiltonian 
\[
H_0(w) := \frac{1}{\epsilon} \chi \bigl( \epsilon H_1 (w)),\quad \forall w \in \hat{W},
\]
which satisfies
\[
H_0(x,\rho) = h_0(\rho)  \quad \mbox{with } h_0(\rho) = \frac{1}{\epsilon} \chi \bigl( \epsilon h_1(\rho) \bigr) = \frac{1}{\epsilon} \chi \left( \frac{\epsilon}{2} (\rho^2 -1 ) \right),
\]
for every $(x,\rho) \in \partial W \times ]0,+\infty[$. The Hamiltonian $H_0$ coincides with $H_1$ on the open set
\[
W \cup \Bigl( \partial W \times \Bigl]0, \sqrt{\frac{2+\epsilon}{\epsilon}} \Bigr[ \Bigr),
\]
and we choose $0<\epsilon\leq 1$ to be so small that the compact set $K$ is contained in the set above. 

We claim that (\ref{kappa}) and the fact that $\epsilon\leq 1$ imply that $h_0$ satisfies (\ref{dadim}). Since $h_0=h_1$ on $]0,((2+\epsilon)/\epsilon)^{1/2}]$, it is enough to check (\ref{dadim}), that is the inequalities
\begin{eqnarray*}
\rho^2 \chi' \left( \frac{\epsilon}{2} (\rho^2 - 1) \right)^2 \leq \kappa \left( \frac{1}{\epsilon} \chi \Big( \frac{\epsilon}{2} (\rho^2 - 1) \Bigr) + 1 \right), \\
\rho^2 \left( \epsilon \rho^2 \chi''\Bigr( \frac{\epsilon}{2} (\rho^2 - 1) \Bigr) + \chi' \Big( \frac{\epsilon}{2} (\rho^2 -1 ) \Bigr) \right)^2 \leq 
\kappa \left( \frac{1}{\epsilon} \chi \Big( \frac{\epsilon}{2} (\rho^2 - 1) \Bigr) + 1 \right),
\end{eqnarray*}
for $\rho\geq ((2+\epsilon)/\epsilon)^{1/2}$. By the change of variable $s = \epsilon (\rho^2 -1 )/2$, this is equivalent to checking the inequalities
\begin{eqnarray}
\label{chi1}
(2s + \epsilon) \chi'(s)^2 \leq \kappa (\chi(s) + \epsilon), \\
\label{chi2}
(2s + \epsilon) \bigl( (2s + \epsilon) \chi''(s) + \chi'(s) \bigr)^2 \leq \kappa ( \chi(s) + \epsilon),
\end{eqnarray}
for every $s\geq 1$. The inequality (\ref{chi1}) follows immediately from (\ref{kappa}), using the fact that $\chi(s)\geq 1$ for $s\geq 1$. By expanding the square and using the inequalities $\epsilon\leq 1$, $\chi' \geq 0$, $\chi'' \leq 0$, and (\ref{kappa}), we can estimate the left-hand side of (\ref{chi2}) by
\begin{equation*}
\begin{split}
 (2s + \epsilon) \bigl( (2s + \epsilon) \chi''(s) + \chi'(s) \bigr)^2 \leq (2s + \epsilon) \bigl( (2s + \epsilon)^2 \chi''(s)^2 + \chi'(s)^2 \bigr) \\
\leq (2s + 1)^3 \chi''(s)^2 + (2 s + 1) \chi'(s)^2 \leq \kappa \leq \kappa (\chi(s) + \epsilon), 
\end{split} \end{equation*}
for every $s\geq 1$, proving (\ref{chi2}).

Therefore, for all the solutions $v=(u,\eta)$ of the Rabinowitz-Floer equation associated to either $H_0$ or $H_1$, with $\mathbb{A}_{H_0}(v(\R))$ or, respectively, $\mathbb{A}_{H_1}(v(\R))$ contained in $[A_-,A_+]$, the image of $u$ is contained in $K$. Since $H_0$ and $H_1$ coincide on an open neighborhood of $K$, we conclude that the boundary homomorphisms of the Rabinowitz-Floer complexes for $H_1$ and for $H_0$ restricted to $[A_-,A_+]$ coincide:
\[
\partial^{[A_-,A_+]}(W,\lambda,a,H_1,J,g_{\mathbb{A}}) =  \partial^{[A_-,A_+]}(W,\lambda,a,H_0,J,g_{\mathbb{A}}).
\]
In particular, the corresponding homologies coincide. Together with the fact, proved in \cite{cf09}, that the truncated Rabinowitz-Floer homology is independent on the choice of the Hamiltonian within the class of eventually constant ones, this concludes the proof.

\section{The Morse chain and differential complexes for the geodesic energy functional}
\label{morcom}

In the following six sections we restrict our attention to the case where the Liouville domain $(W,\lambda)$ is the cotangent disk bundle $(D^*M,\lambda)$ associated to a 
closed Riemannian manifold $(M,g)$ of dimension $n\geq 2$ (see Example 1 in Section \ref{setting}). The aim of this section is to recall the definition of the Morse chain complex and the Morse differential complex of the energy functional associated to closed geodesics on $(M,g)$. 

\paragraph{The geodesic energy functional.} The geodesic energy functional is the non-negative functional
\[
\mathbb{E}(\gamma) = \int_{\T} g(\gamma'(t),\gamma'(t))\, dt
\]
defined on the loop space $\Lambda M$ of $M$. This functional is smooth and satisfies the Palais-Smale condition on the Hilbert manifold $W^{1,2}(\T,M)$ of absolutely continuous loops $\gamma:\T \rightarrow M$ whose derivative is square-integrable, endowed with its natural complete Riemannian structure, that is
\begin{equation}
\label{natmet}
\langle\langle \xi_1,\xi_2 \rangle\rangle := \int_{\T} \bigl( g(\nabla_t \xi_1,\nabla_t \xi_2) + g(\xi_1,\xi_2) \bigr)\, dt,\quad \forall \xi_1,\xi_2 \in T_{q} W^{1,2}(\T,M), \; \forall q\in W^{1,2} (\T,M),
\end{equation}
where $\nabla_t$ denotes the Levi-Civita covariant derivative along the curve $q$ (see \cite{kli82}).

The critical points of $\mathbb{E}$ on $W^{1,2}(\T,M)$ are the (possibly constant) closed geodesics of $(M,g)$, parametrized on $\T$ with constant speed. The assumption that the closed Reeb orbits of $(S^*M,\alpha)$ are of Morse-Bott type is equivalent to the fact that $\mathbb{E}$ is a Morse-Bott functional, meaning that $\crit\, \mathbb{E}$ is a closed submanifold, and for each $\gamma\in \crit\, \mathbb{E}$ the second differential $d^2 \mathbb{E} (\gamma)$ of $\mathbb{E}$ at $\gamma$ is Fredholm and its kernel is the tangent space of $\crit\, \mathbb{E}$ at $\gamma$. We recall that a continuous bilinear form on a Hilbert space is said to be Fredholm if the bounded operator which represents it with respect to the Hilbert product is Fredholm, and that this notion does not depend on the choice of an equivalent Hilbert product. The Morse index of a closed geodesic $\gamma$ is the non-negative integer
$\ind ( \gamma; \mathbb{E})$, 
that is the dimension of a maximal linear subspace of $T_{\gamma} W^{1,2} (\T,M)$ on which $d^2 \mathbb{E}(\gamma)$ is negative-definite.

The symbol $g_{\Lambda}$ denotes an arbitrary smooth Riemannian metric on $W^{1,2}(\T,M)$ which is globally equivalent to the standard one (\ref{natmet}), and once such a metric has been fixed, $\nabla \mathbb{E}$ denotes the corresponding gradient vector field of $\mathbb{E}$. By the Morse-Bott assumption, the manifold $\crit\, \mathbb{E}$ is normally hyperbolic with respect to the negative gradient flow of $\mathbb{E}$. 
In particular, if $\gamma$ is a non-constant closed geodesic then
\[
\dim W^u(\gamma; - \nabla \mathbb{E}) = \ind (\gamma; \mathbb{E}), \quad \codim \,W^s(\gamma; -\nabla \mathbb{E}) =  \ind (\gamma; \mathbb{E}) + \dim \ker d^2 \mathbb{E}(\gamma).
\]
In the case of a constant geodesic, that is a point $q$ in $M$, we have
\[
\dim W^u(q ; - \nabla \mathbb{E}) = 0, \quad \codim\, W^s(q; -\nabla \mathbb{E}) =  \dim M.
\]
Moreover, every unstable manifold has compact closure, so a regularization argument shows that, up to the modification of  the Riemannian metric $g_{\Lambda}$, we may assume that:
\begin{enumerate}
\item[(A)] For every closed geodesic $\gamma \in \crit\, \mathbb{E}$, $W^u(\gamma; - \nabla \mathbb{E})$ is a submanifold of $C^{\infty}(\T,M)$. Moreover, for every $E>0$ there exists a constant $C=C(E)$ such that if $\gamma\in \crit \, \mathbb{E}$ has energy $\mathbb{E}(\gamma) \leq E$ then
\[
\sup_{t\in \mathbb{T}} g(q'(t),q'(t)) \leq C^2, \quad \forall q\in W^u(\gamma; - \nabla \mathbb{E}).
\]
\end{enumerate}  

\paragraph{The Morse chain complex of $\mathbb{E}$.} The cascades approach from \cite{fra04} and the results about the Morse complex on infinite dimensional manifolds from \cite{ama06m} can be combined and provide us with the following description of the Morse complex of the geodesic energy functional.
 
By the Palais-Smale condition, by the completeness of $(W^{1,2}(\T,M), g_{\Lambda})$,
by the fact that $\mathbb{E}$ is bounded from below, and by the Morse-Bott assumption, for every $q\in W^{1,2}(\T,M)$ the limit
\[
\phi_{+\infty}^{-\nabla \mathbb{E}} (q) := \lim_{s\rightarrow +\infty} \phi_s^{-\nabla \mathbb{E}} (q)
\]
exists and it is an element of $\crit\, \mathbb{E}$. On the other hand,
\[
 \phi_{-\infty}^{-\nabla \mathbb{E}} (q) := \lim_{s\rightarrow \sigma_-(q)} \phi_s^{-\nabla \mathbb{E}} (q)
\]
is either the empty set, or it is an element of $\crit \, \mathbb{E}$ (here $]\sigma_-(q),+\infty[$ denotes the maximal interval of existence of the orbit of $q$ by the local flow of $-\nabla \mathbb{E}$, which is easily seen to be positively complete). The latter fact holds if and only if the function $s\mapsto \mathbb{E}(\phi_s^{-\nabla \mathbb{E}}(q))$ is bounded from above, and this holds if and only if $q$ belongs to the unstable manifold of some closed geodesic $\gamma$; in this case, $\sigma_-(q)=-\infty$ and $\phi_{-\infty}^{-\nabla \mathbb{E}}(q) = \gamma$.

Let $e$ be a smooth Morse function on $\crit\, \mathbb{E}$, and let $g_{\mathbb{E}}$ be a Riemannian metric on $\crit\, \mathbb{E}$ such that  the negative gradient flow of $e$ is Morse-Smale. 

If $\gamma^-, \gamma^+\in \crit\, e$, then 
\[
\tilde{\mathcal{W}}_0(\gamma^-,\gamma^+) := W^u(\gamma^-; - \nabla e) \cap W^s (\gamma^+;- \nabla e),
\]
and if $m\geq 1$ is an integer, $\tilde{\mathcal{W}}_m(\gamma^-,\gamma^+)$ is the set of $m$-tuples $(q_1,\dots,q_m)$ of loops in $W^{1,2}(\T,M) \setminus \crit \, \mathbb{E}$ such that
\[
\phi^{-\nabla \mathbb{E}}_{-\infty} (q_1) \in W^u(\gamma^-; - \nabla e), \quad 
\phi^{-\nabla \mathbb{E}}_{+\infty} (q_m) \in W^s(\gamma^+; - \nabla e),
\]
and for every $j=1,\dots,m-1$
\[
\phi^{-\nabla \mathbb{E}}_{-\infty} (q_{j+1}) \in \phi^{- \nabla e}_{\R^+} \bigl( \phi_{+\infty}^{-\nabla \mathbb{E}} (q_j) \bigr).
\]
When $\gamma^- \neq \gamma^+$, the quotient of $\tilde{\mathcal{W}}_0(\gamma^-,\gamma^+)$ by the free action of $\R$ given by the flow of $-\nabla e$ is denoted by $\mathcal{W}_0(\gamma^-, \gamma^+)$, while if $\gamma^-=\gamma^+$, $\mathcal{W}_0(\gamma^-, \gamma^+)$ is defined to be the empty set. Similarly, when $m\geq 1$ the quotient of $\tilde{\mathcal{W}}_m(\gamma^-,\gamma^+)$ by the free $\R^m$-action given by the flow of $- \nabla \mathbb{E}$ is denoted by $\mathcal{W}_m(\gamma^-,\gamma^+)$. Finally,
\[
\tilde{\mathcal{W}}(\gamma^-,\gamma^+) := \bigcup_{m\geq 0} \tilde{\mathcal{W}}_m(\gamma^-,\gamma^+), \quad
\mathcal{W}(\gamma^-,\gamma^+) := \bigcup_{m\geq 0} \mathcal{W}_m(\gamma^-,\gamma^+).
\]
For a generic choice of $g_{\Lambda}$ and $g_{\mathbb{E}}$, $\mathcal{W}(\gamma^-,\gamma^+)$ is a smooth manifold of dimension
\[
\dim \mathcal{W}(\gamma^-,\gamma^+) = \mathrm{Ind}\, (\gamma^-;\mathbb{E},e) -  \mathrm{Ind}\, (\gamma^+;\mathbb{E},e) - 1,
\]
where we define
\begin{equation}
\label{Ind}
\mathrm{Ind}\, (\gamma;\mathbb{E},e) := \ind (\gamma; \mathbb{E}) + \ind (\gamma; e).
\end{equation}
When $\mathrm{Ind}\,(\gamma^+;\mathbb{E},e)=\mathrm{Ind}\,(\gamma^-;\mathbb{E},e)-1$, the discrete set $\mathcal{W}(\gamma^-,\gamma^+)$ is compact, hence finite, and we define 
\[
n_M(\gamma^-,\gamma^+) = n_M(\gamma^-,\gamma^+; - \nabla \mathbb{E}, - \nabla e) \in \Z_2 
\]
to be its parity. We define $M_k = M_k(\mathbb{E},e)$ to be the $\Z_2$-vector space generated by critical points $\gamma$ of $e$ with 
$\mathrm{Ind}\, (\gamma;\mathbb{E},e) = k$, and
\[
\partial = \partial(\mathbb{E},e,g_{\Lambda},g_{\mathbb{E}}) : M_k (\mathbb{E},e) \rightarrow M_{k-1} (\mathbb{E},e), \quad \partial \gamma^- = \sum_{\substack{\gamma^+ \in \crit\, e \\ \mathrm{Ind}\, (\gamma^+;\mathbb{E},e) = k-1}} n_M (\gamma^-,\gamma^+) \, \gamma^+,
\]
for every $\gamma^-\in \crit\, e$ with $\mathrm{Ind}\, (\gamma^-;\mathbb{E},e)=k$. Then $\partial \circ \partial = 0$, so $\{M_*(\mathbb{E},e),\partial\}$ is a chain complex of $\Z_2$-vector spaces, called the {\em Morse chain complex of the geodesic energy functional $\mathbb{E}$}. Its homology is isomorphic to the singular homology of the ambient manifold $W^{1,2}(\T,M)$. Since the latter Hilbert manifold is homotopically equivalent to the loop space $\Lambda M$, we have
\[
HM_k ( \mathbb{E},e) \cong H_k (\Lambda M), \quad \forall k\in \Z,
\]
and the space on the left-hand side is called the {\em Morse homology of the geodesic energy functional $\mathbb{E}$}.

\paragraph{The Morse differential complex of $\mathbb{E}$.} 
Let $f$ be a smooth Morse function on $\crit\, \mathbb{E}$, and let $g_{\mathbb{E}}$ be a Riemannian metric on $\crit\, \mathbb{E}$ such that  the negative gradient flow of $f$ is Morse-Smale. Starting from Section \ref{psisec}, it will be convenient to choose $f=-e$, where $e$ is the auxiliary Morse function used to define the Morse complex of $\mathbb{E}$.
If $k$ is an integer, we define $M^k(\mathbb{E},f)$ to be the $\Z_2$-vector space consisting of formal sums
\[
\sum_{\gamma\in \Gamma} \gamma,
\]
where $\Gamma$ is a subset of $\crit\, f$ consisting of closed geodesics $\gamma$ such that 
\[
\mathrm{Ind}\,(\gamma;\mathbb{E},f) = k. 
\]
In other words, $M^k(\mathbb{E},f)$ is the direct product of one copy of $\Z_2$ for each $\gamma \in \crit\, f$ with $\mathrm{Ind}\,(\gamma;\mathbb{E},f)=k$, whereas 
$M_k(\mathbb{E},e)$ is a direct sum.  We define the homomorphism
\[
\delta =  \delta(\mathbb{E},f,g_{\Lambda},g_{\mathbb{E}}) : M^k (\mathbb{E},f) \rightarrow M^{k+1} (\mathbb{E},f), \quad \delta \gamma^+ = \sum_{\substack{\gamma^- \in \crit\, f \\ \mathrm{Ind}\, (\gamma^-;\mathbb{E},f) = k+1}} n_M (\gamma^-,\gamma^+;-\nabla \mathbb{E},-\nabla f) \, \gamma^-,
\]
for every $\gamma^+\in \crit\, f$ with $\mathrm{Ind}\, (\gamma^+;\mathbb{E},f)=k$.  Then $\delta \circ \delta = 0$, so $\{M^*(\mathbb{E},f),\delta\}$ is a graded differential $\Z_2$-vector space, called the {\em Morse differential complex of the geodesic energy functional $\mathbb{E}$}. Its cohomology is isomorphic to the singular cohomology of the ambient manifold $W^{1,2}(\T,M)$, so
\[
HM^k ( \mathbb{E},f) \cong H^k (\Lambda M), \quad \forall k\in \Z,
\]
and the space on the left-hand side is called the {\em Morse cohomology of the geodesic energy functional $\mathbb{E}$}.

\paragraph{Unstable manifolds with cascades.} 
In the following sections, we shall define homomorphisms between the Morse chain and differential complexes of $\mathbb{E}$ and the Rabinowitz-Floer complex by considering spaces of half-flow lines with cascades both on the Morse and on the Rabinowitz-Floer side, with suitable coupling conditions. In order to simplify the notation, it is useful to introduce the following notion of unstable manifold with cascades on the Morse side.
 
Let $\gamma$ be a critical point of $e$. If $m\geq 1$ is an integer, we define $\tilde{\mathcal{W}}^u_m (\gamma)$ to be the set of $m$-tuples $(q_1,\dots,q_m)$ of loops in $W^{1,2}(\T , M)$ such that
\[
q_1 \in W^u( W^u(\gamma; - \nabla e); - \nabla \mathbb{E}), 
\]
and
\[
\phi_{-\infty}^{-\nabla \mathbb{E}} (q_{j+1}) \in \phi^{-\nabla e}_{\R^+} \bigl( \phi_{+\infty}^{-\nabla \mathbb{E}}(q_j)\bigr), \quad \forall j=1,\dots,m-1.
\]
We define $\mathcal{W}_m^u(\gamma)$ to be the quotient of $\tilde{\mathcal{W}}^u_m (\gamma)$ with respect to the $\R^{m-1}$-action obtained by letting the negative gradient flow of $\mathbb{E}$ act independently on the first $m-1$ loops $q_1,\dots,q_{m-1}$. The {\em unstable manifold of $\gamma$ with cascades} is the set
\[
\mathcal{W}^u(\gamma) = \mathcal{W}^u(\gamma;-\nabla \mathbb{E},- \nabla e) = \bigcup_{m\geq 1} \mathcal{W}_m^u(\gamma).
\]
For generic $g_{\Lambda}$ and $g_{\mathbb{E}}$, $\mathcal{W}^u(\gamma)$
has the structure of a smooth manifold of dimension $\mathrm{Ind}\, (\gamma)$, and there is a well-defined smooth evaluation map
\[
\mathrm{ev} : \mathcal{W}^u(\gamma) \rightarrow W^{1,2}(\T,M), \quad  \mathrm{ev}(\underline{q}) := q_m, \quad \mbox{if } \underline{q} = [( q_1,\dots,q_m ) ],
\]
whose image is pre-compact.
Notice that $\mathrm{ev}$ is injective on $\mathcal{W}^u_1(\gamma)$ (it is actually an embedding there), but it needs not be injective on the other components, because distinct orbits of the negative gradient flow of $\mathbb{E}$ may have the same limit at $+\infty$.
 
\section{Comparison between the functionals $\mathbb{A}$ and $\mathbb{E}$}
\label{criti}
\paragraph{Relationship between the functionals' level.} 
Let us fix once for all the following quadratic Hamiltonian on $T^*M = \widehat{D^*M}$:
\[
H(q,p) = \frac{1}{2} \bigl( g^*(p,p) - 1 \bigr),
\]
which is compatible with $(S^* M,\alpha)$ and eventually quadratic (see Section \ref{rfcsec}). Indeed, with such a choice,
\begin{equation}
\label{lamham}
\lambda(X_H(q,p)) = g^*(p,p), \quad \forall (q,p) \in T^*M.
\end{equation}
The Rabinowitz functional $\mathbb{A}$ is associated to this Hamiltonian.

We denote by $\pi: T^*M \rightarrow M$ the projection, and by the same symbol its restriction to $S^*M$.
We denote by $G:TM \rightarrow T^*M$ the isomorphism induced by the metric $g$ by the formula
\[
g(G^{-1} p, v) = \langle p, v \rangle = g^* ( p, Gv ), \quad \forall p\in T_q^* M, \; v\in T_q M, \; q\in M,
\]
where $\langle \cdot,\cdot \rangle$ denotes the duality pairing. 

\begin{lem}
\label{levrel}
Let $x=(q,p)\in \Lambda T^*M$  and set $E:= \mathbb{E}(q)$. Then 
\[
\mathbb{A}(x,\sqrt{E}) \leq \sqrt{E}, \quad  \mathbb{A}(x,-\sqrt{E}) \geq -\sqrt{E}.
\]
The first inequality is an equality if and only if $\sqrt{E} p=Gq'$, and the second one is an equality if and only if $\sqrt{E} p=-Gq'$. In particular, the first inequality is an equality when $(x,\sqrt{E})$ is a critical point of $\mathbb{A}$, and the second one is an equality when $(x,-\sqrt{E})$ is a critical point of $\mathbb{A}$. 
\end{lem}

\begin{proof}
Since $G$ is an isometry from $(TM,g)$ to $(T^*M,g^*)$, we have
\[
E = \mathbb{E}(q) = \int_{\T} g(q',q')\, dt = \int_{\T} g^* (Gq',Gq')\, dt.
\]
Then
\[
\mathbb{A}(x,\sqrt{E}) = \int_{\T} \langle p, q' \rangle \, dt - \sqrt{E} \int_{\T} \frac{1}{2} \bigl( g^*(p,p) - 1 \bigr)\, dt = - \int_{\T} \left( \frac{\sqrt{E}}{2} g^*(p,p) - g^*(p,Gq') \right) \, dt + \frac{\sqrt{E}}{2}.
\]
By completing the square in the first integral,
\begin{eqnarray*}
\mathbb{A}(x,\sqrt{E}) = - \frac{\sqrt{E}}{2} \int_{\T}  g^* \left( p - \frac{1}{\sqrt{E}} Gq', p - \frac{1}{\sqrt{E}} Gq' \right)\, dt + \frac{1}{2\sqrt{E}} \int_{\T}  g^* (Gq',Gq') \, dt + \frac{\sqrt{E}}{2} \\ 
= - \frac{\sqrt{E}}{2} \int_{\T}  g^* \left( p - \frac{1}{\sqrt{E}} Gq', p - \frac{1}{\sqrt{E}} Gq' \right)\, dt + \sqrt{E} \leq \sqrt{E},
\end{eqnarray*}
and the equality holds if and only if
\[
p = \frac{1}{\sqrt{E}} Gq'.
\]
This proves the statements about $\mathbb{A}(x,\sqrt{E})$. The statements about $\mathbb{A}(x,-\sqrt{E})$ are deduced from the latter ones by noticing that $\mathbb{A}((q,p),-\sqrt{E}) = - \mathbb{A}((q,-p),\sqrt{E})$.  
\end{proof} \qed

\paragraph{The critical sets of $\mathbb{E}$ and $\mathbb{A}$.} Each critical point $\gamma$ of $\mathbb{E}$ with $\mathbb{E}(\gamma)>0$, that is each non-constant closed geodesic, determines two critical points of $\mathbb{A}$, one of which has positive Rabinowitz action, the other negative action. By Lemma \ref{levrel}, these are the critical points
\begin{equation}
\label{z+-}
Z^+(\gamma) := \bigl(x^+,\sqrt{\mathbb{E}(\gamma)}\bigr), \quad Z^-(\gamma) := \bigl(x^-,-\sqrt{\mathbb{E}(\gamma)}\bigr),
\end{equation}
where $x^+,x^-:\T \rightarrow S^*M$ are the reparametrized Reeb orbits
\[
x^+(t) := \bigl(\gamma(t), g(\gamma'(t),\gamma'(t))^{-1/2} G \gamma'(t) \bigr), \quad x^-(t) := x^+(-t).
\]
The maps $Z^+$ and $Z^-$ determine the diffeomorphism:
\begin{equation}
\label{diffeo}
\set{\gamma \in \crit\, \mathbb{E}}{\mathbb{E}(\gamma)>0} \times \{-1,+1\} \rightarrow \set{(x,\eta) \in \crit\, \mathbb{A}}{\eta \neq 0}, \quad (\gamma,\pm 1) \mapsto Z^{\pm} (\gamma).
\end{equation}
It is useful to choose the auxiliary Morse functions $a\in \C^{\infty} (\crit \, \mathbb{A})$ and $e\in C^{\infty} ( \crit \, \mathbb{E})$ in such a way that:
\begin{enumerate}
\item[(A0)] $e(\gamma) = a(Z^+(\gamma)) = a(Z^-(\gamma))$ for every $\gamma \in \crit\, \mathbb{E}$ such that $\mathbb{E}(\gamma)>0$.
\end{enumerate}
The critical points corresponding to constant loops form instead two different manifolds:
\[
\set{\gamma \in \crit\, \mathbb{E}}{\mathbb{E}(\gamma) = 0} = M, \quad \set{(x,\eta)\in \crit\, \mathbb{A}}{\eta = 0} = S^* M \times \{0\}.
\]
It is convenient to assume that the Morse function $a$ on $S^*M\times \{0\}$ is related to the Morse function $e$ on $M$ by the following conditions:
\begin{enumerate}

\item[(A1)] $e|_M$ has a unique minimum point $q_{\min}$, a unique maximum point $q_{\max}$, and it is self-indexing, that is $e(q)=\ind(q;e)$ for every $q\in \crit\, e|_M$.

\item[(A2)] $e(\pi(x)) \leq a(x,0) \leq e(\pi(x)) + 1/2$ for every $x\in S^*M$.

\item[(A3)] Every critical point of $a|_{S^*M \times \{0\}}$ lies above a critical point of $e$, and for every critical point $q$ of $e$ the fiber $\pi^{-1}(q)\times \{0\}$ contains exactly two critical points of $a$, that we denote by $z_q^-=(x_q^-,0)$ and $z_q^+=(x_q^+,0)$, such that $a(z_q^-)=e(q)$ and $a(z_q^+)=e(q)+1/2$.

\item[(A4)] For every $q$ in $\crit\, e$, we have $\ind(z_q^-;a) = \ind(q;e)$ and $\ind(z^+_q;a) = \ind(q;e) + n-1$.

\end{enumerate} 

It is easy to construct pairs of Morse functions $e$, $a$ which satisfy the above conditions. These conditions imply that the Morse complex of $e|_M$ and of $a|_{S^*M\times \{0\}}$ are related in a very simple and explicit way, as it is explained in the Appendix B. If 
\[
\xi = \sum_{q\in \crit \, e|_M} \xi_q q
\]
is a chain in $M_*(e|_M)$, we denote by $z_{\xi}^-$ and $z_{\xi}^+$ the corresponding chains in $RF_*$, that is
\[
z_{\xi}^- := \sum_{q\in \crit \, e|_M} \xi_q z_q^-, \quad
z_{\xi}^+ := \sum_{q\in \crit \, e|_M} \xi_q z_q^+.
\]
Then Proposition \ref{descri} from Appendix B implies the following:

\begin{prop}
\label{partialprop}
If the conditions (A1-A2-A3-A4) hold, then the boundary operator in the Rabinowitz-Floer complex $RF(D^*M,\lambda,a)$ satisfies
\begin{eqnarray*}
\partial z_q^+ & \in z^+_{\partial q} + RF^{]-\infty,0[}, \quad & \forall q\in \crit \,e|_M , \\
\partial z_q^- & \in z_{\partial q}^- +  RF^{]-\infty,0[}, \quad & \forall q\in (\crit\, e|_M) \setminus \{q_{\max}\}, \\
\partial z^-_{q_{\max}} & \in \chi(T^*M) z_{q_{\min}}^+ +  RF^{]-\infty,0[}, &
\end{eqnarray*}
where $\chi(T^*M)$ is the (modulo 2) Euler number of the vector bundle
$T^*M\rightarrow M$.
\end{prop}

It is useful to extend the maps $Z^+$ and $Z^-$ defined in (\ref{z+-}) to the critical points of $e$ on $M = \mathbb{E}^{-1}(0)$ by setting
\begin{equation}
\label{z+-2}
Z^+(q) := z^+_q, \quad Z^-(q) := z^-_q, \quad \forall q\in (\crit \, e) \cap \mathbb{E}^{-1}(0).
\end{equation}
With such a definition, $\crit\, a$ is the disjoint union of $Z^+(\crit\, e)$ and $Z^-(\crit\, e)$. Let us denote by $RF^-$ the following subspace of $RF$:
\begin{equation}
\label{RF-}
RF^- := \set{\sum_{z\in \mathscr{Z}} z}{\mathscr{Z}\subset Z^-(\crit \, e)} = RF^{]-\infty,0[} \oplus \mathrm{Span} \, \set{z_q^-}{q\in \crit\, e|_{\mathbb{E}^{-1}(0)} }.
\end{equation}
By Proposition \ref{partialprop}, $RF^-$ is a subcomplex of $RF$ if and only if $\chi(T^*M)=0$.

\paragraph{Relationship between the indices.} Let us discuss the relationship between the transverse Conley-Zehnder index and the Morse index. 
Let $\gamma$ be a critical point of $\mathbb{E}$ with $\mathbb{E}(\gamma)> 0$, and consider the critical points $Z^+(\gamma)=(x^+,\sqrt{\mathbb{E}(\gamma)})$ and $Z^-(\gamma)=(x^-,- \sqrt{\mathbb{E}(\gamma)})$ of $\mathbb{A}$ as above. Then
\[
\dim \ker d^2 \mathbb{E} (\gamma) = \nu^{\tau} (x^+) + 1, \quad
\ind (\gamma;\mathbb{E}) = \mu_{CZ}^{\tau} (x^+) - \frac{1}{2} \nu^{\tau} (x^+),
\]
as it can be deduced for instance from Corollary 4.2 in \cite{aps08} (the full Conley-Zehnder index $\mu_{CZ}(x^+)$ and nullity $\nu(x^+)$ are considered 
there, but the above identities easily follow, because for autonomous systems $\nu(x^+)=\nu^{\tau}(x^+)+1$, and for Hamiltonians which are strictly convex in $p$, $\mu_{CZ}(x^+) = \mu_{CZ}^{\tau}(x^+) + 1/2$). In particular, all critical points $(x,\eta)$ of $\mathbb{A}$ with positive action $\mathbb{A}(x,\eta)$ have non-negative transverse Conley-Zehnder index. 

Since $x^-(t)=x^+(-t)$, the properties of the Conley-Zehnder index imply that
\[
\nu^{\tau}(x^-) = \nu^{\tau}(x^+), \quad \mu^{\tau}_{CZ}(x^-) = - \mu^{\tau}_{CZ}(x^+).
\]
It follows that 
\[
\dim \ker d^2 \mathbb{E}(\gamma) = \nu^{\tau}(x^-) + 1, \quad \ind  (\gamma;\mathbb{E}) = - \mu_{CZ}^{\tau} (x^-) - \frac{1}{2} \nu^{\tau} (x^-).
\]
We conclude that 
\begin{equation}
\label{indexrel}
\dim \ker d^2 \mathbb{E} (\gamma) = \nu^{\tau} (x^{\pm}) + 1, \quad
\ind (\gamma;\mathbb{E}) = \pm \mu_{CZ}^{\tau} (x^{\pm}) - \frac{1}{2} \nu^{\tau} (x^{\pm}).
\end{equation}
In particular,  all critical points $(x,\eta)$ of $\mathbb{A}$ negative action $\mathbb{A}(x,\eta)$ have non-positive transverse Conley-Zehnder index. 

By the assumption (A0), the diffeomorphism (\ref{diffeo}) induces the following bijection:
\[
\set{\gamma \in \crit \, e}{\mathbb{E}(\gamma)>0} \times \{-1,+1\} \rightarrow \set{(x,\eta)\in \crit\, a}{\eta \neq 0}, \quad (\gamma,\pm 1) \mapsto Z^{\pm} (\gamma),
\]
and
\[
\ind (\gamma;e) =  \ind (Z^+(\gamma);a) = \ind (Z^-(\gamma);a), \quad \forall \gamma \in (\crit\, e) \cap \{\mathbb{E}>0\}.
\]
Finally, by the definitions (\ref{grad}) and (\ref{Ind}) of the indices $\mu$ and $\mathrm{Ind}$, the above identity and (\ref{indexrel}) imply that for every critical point $\gamma$ of $e$ with $\mathbb{E}(\gamma)>0$ there holds
\[
\mu(Z^+(\gamma))  = \mu_{CZ}^{\tau} (x^+) - \frac{1}{2} \nu^{\tau}(x^+) + \mathrm{ind}\, (Z^+(\gamma);a) = \ind (\gamma;\mathbb{E}) + \ind (\gamma;e) = 
\mathrm{Ind}\, (\gamma;\mathbb{E},e),
\]
and, using also (\ref{dimcrit}),
\begin{eqnarray*}
& \mu(Z^-(\gamma))  = \mu_{CZ}^{\tau} (x^-) - \frac{1}{2} \nu^{\tau}(x^-) + \mathrm{ind}\, (Z^-(\gamma);a) = - \ind (\gamma;\mathbb{E}) - \nu^{\tau}(x^-) + \ind\, (\gamma;e) \\  & = - \ind (\gamma;\mathbb{E}) - \dim K_{\gamma} + 1 + \ind (\gamma;e) = - \ind (\gamma;\mathbb{E}) - \ind (\gamma;-e) + 1 = - \mathrm{Ind}\, (\gamma;\mathbb{E},-e) + 1,
\end{eqnarray*}
where $K_{\gamma}$ denotes the connected component of $\crit\, \mathbb{E}$ which contains $\gamma$. Hence, we have shown that
\begin{equation}
\label{Indrel}
\mu(Z^+(\gamma)) = \mathrm{Ind} (\gamma; \mathbb{E},e), \quad
\mu(Z^-(\gamma)) = 1 - \mathrm{Ind}(\gamma;\mathbb{E},-e).
\end{equation}
On the other hand, by the assumptions (A3-A4) above and by the
definitions (\ref{grad0}) and (\ref{Ind}), the indices of the critical
points of zero action are related as follows: for every critical point
$q$ of $e$ on $\mathbb{E}^{-1}(0)$, if $Z^+(q)$ and $Z^-(q)$ are the
critical points of $a$ on 
$(\crit\, \mathbb{A}) \cap \{\mathbb{A} =0\}$ defined in (\ref{z+-2}), there holds
\begin{equation*}
\begin{split}
\mu(Z^+(q)) &= \ind (Z^+(q);a) - n + 1= \ind (q;e), \\
\mu(Z^-(q)) &= \ind(Z^-(q);a) - n + 1 = \ind(q;e) - n + 1 = 1 - \ind(q;-e),
\end{split} \end{equation*}
and hence
\begin{equation}
\label{Indrel0}
\begin{split}
\mu(Z^+(q)) &= \mathrm{Ind} (q;\mathbb{E},e) = n - \mathrm{Ind} (q;\mathbb{E},-e), \\ \mu(Z^-(q)) &= \mathrm{Ind} (q;\mathbb{E},e) - n + 1 = 1 - \mathrm{Ind}(q;\mathbb{E},-e).
\end{split} \end{equation}
By (\ref{Indrel}) and (\ref{Indrel0}), the set $\set{z\in \crit\, a}{\mu(z)=k}$ coincides with
\begin{equation}
\label{bijecrit}
 \begin{array}{ll} Z^+\bigl(\set{\gamma \in \crit\, e}{\mathrm{Ind}(\gamma;\mathbb{E},e) = k}\bigr) & \mbox{for } k\geq 2, \\
Z^+\bigl(\set{\gamma \in \crit\, e}{\mathrm{Ind}(\gamma;\mathbb{E},e) = 1}\bigr) \sqcup Z^-\bigl(\set{\gamma \in \crit\, e}{\mathrm{Ind}(\gamma;\mathbb{E},-e) = 0}\bigr) & \mbox{for } k=1, \\
Z^+\bigl(\set{\gamma \in \crit\, e}{\mathrm{Ind}(\gamma;\mathbb{E},e) = 0}\bigr) \sqcup Z^-\bigl(\set{\gamma \in \crit\, e}{\mathrm{Ind}(\gamma;\mathbb{E},-e) = 1}\bigr) & \mbox{for } k=0, \\
 Z^-\bigl(\set{\gamma \in \crit\, e}{\mathrm{Ind}(\gamma;\mathbb{E},-e) = 1-k}\bigr) & \mbox{for } k\leq -1. \end{array}
\end{equation}
However, there is no reason why $Z^+$ and $Z^-$ should induce chain maps between the Morse and the Rabinowiz-Floer complexes.

\paragraph{Free homotopy classes of loops and splitting of the complexes.} If $c\in [\T,M]\cong [\T,T^*M]$ is a free homotopy class of loops in $M$, we denote by $\Lambda^c M$ and $\Lambda^c T^*M$ the connected component of $\Lambda M$ and $\Lambda T^*M$ consisting of loops in the free homotopy class $c$. The Morse chain and differential complexes of $\mathbb{E}$ split accordingly, as
\[
M_k (\mathbb{E},e) = \bigoplus_{c\in [\T,M]} M_k^c (\mathbb{E},e), \quad
M^k (\mathbb{E},e) = \prod_{c\in [\T,M]} M^k_c (\mathbb{E},e), 
\]
where $M_k^c$ and $M^k_c$ are formal sums of critical points $\gamma$ of $e$ with $[\gamma]=c$, and  
\[
HM_k^c (\mathbb{E},e) \cong H_k(\Lambda^c M), \quad HM^k_c (\mathbb{E},e) \cong H^k(\Lambda^c M).
\]
Similarly, the Rabinowitz-Floer boundary operator restricts to a homomorphism 
\[
\partial: RF_k^c (D^*M,\lambda,a) \rightarrow RF_{k-1}^c (D^*M,\lambda,a),
\]
where the space $RF^c$ consists of formal sums of critical points $(x,\eta)$ of $a$ such that $[x]=c\in [\T,M]$. The above restrictions uniquely determine the whole Rabinowitz-Floer complex, and the inclusions
\[
\bigoplus_{c\in [\T,M]}  RF_k^c (D^*M,\lambda,a) \subset RF_k (D^*M,\lambda,a) \subset \prod_{c\in [\T,M]}  RF_k^c (D^*M,\lambda,a)
\]
are in general strict.
Notice that the critical points $(x,0)$ contribute to $RF^0$, where $0$ denotes the class of contractible loops, and that if $\eta\neq 0$ the Reeb orbit $y(t) := x(t/\eta)$ associated to a critical point $(x,\eta)$ which contributes to $RF^c$ is in the free homotopy class $(\sgn \eta) c\in [\T,T^*M]$.

\paragraph{Relationship between the second differentials of $\mathbb{E}$ and $\mathbb{A}$ at critical points.}
The following differential version of Lemma \ref{levrel} is useful in order to deal with questions of automatic transversality:

\begin{lem}
\label{levreldiff}
Let $\gamma$ be a critical point of $\mathbb{E}$ with $\mathbb{E}(\gamma)>0$ and let $Z^+(\gamma)=(x^+,\sqrt{\mathbb{E}(\gamma)})$ and $Z^-(\gamma)=(x^-,-\sqrt{\mathbb{E}(\gamma)})$ be the corresponding critical points of $\mathbb{A}$. Then for every $\xi^+\in T_{x^+} \Lambda T^*M$ and $\xi^- \in T_{x^-} \Lambda T^*M$ there holds
\begin{eqnarray}
\label{piu}
2 \sqrt{\mathbb{E}(\gamma)} d^2_{xx} \mathbb{A} (Z^+(\gamma)) [ \xi^+,\xi^+ ] & \leq & d^2 \mathbb{E}(\gamma) \bigl[ d\pi (x^+)[\xi^+], d\pi (x^+)[\xi^+] \bigr], \\
\label{meno}
2 \sqrt{\mathbb{E}(\gamma)} d^2_{xx} \mathbb{A} (Z^-(\gamma)) [ \xi^-,\xi^- ] & \geq & - d^2 \mathbb{E}(\gamma) \bigl[ d\pi (x^-)[\xi^-], d\pi (x^-)[\xi^-] \bigr],
\end{eqnarray}
where $d_{xx}^2 \mathbb{A}$ denotes the second differential of $\mathbb{A}=\mathbb{A}(x,\eta)$ with respect to the first variable.
\end{lem} 

\begin{proof}
Consider the functions $\mathcal{A}, \mathcal{E} : \Lambda T^*M \rightarrow \R$ defined as
\[
\mathcal{A}(x) := \mathbb{A}\bigl(x, \sqrt{\mathbb{E}(\pi\circ x)}\bigr), \quad \mathcal{E}(x) := \sqrt{\mathbb{E}(\pi\circ x)}.
\]
By Lemma \ref{levrel}, $\mathcal{A} (x) \leq \mathcal{E} (x)$ for every $x\in \Lambda T^*M$, and $\mathcal{A}=\mathcal{E}$ at the common critical point $x^+$. Therefore,
\[
d^2 \mathcal{A}(x^+) [\xi^+,\xi^+] \leq d^2 \mathcal{E}(x^+) [\xi^+,\xi^+], \quad \forall \xi^+ \in T_{x^+} \Lambda T^*M,
\]
which implies (\ref{piu}) by computing the second differentials of $\mathcal{A}$ and $\mathcal{E}$ at $x^+$. The proof of (\ref{meno}) is based on the second inequality of Lemma \ref{levrel} and is analogous.
\end{proof} \qed

\section{The chain map $\Phi$}
\label{phisec}

\paragraph{Half flow lines with cascades.} It is useful to introduce the following notion of half flow lines with cascades for the Rabinowitz functional $\mathbb{A}$. They are the analogue of the stable and unstable manifolds with cascades for the energy functional $\mathbb{E}$, but in a case in which there is no well-defined flow.

Let $z$ be a critical point of $a$. If $m\geq 1$ is an integer, we define $\tilde{\mathcal{M}}_m^u(z)$ to be the set of $m$-tuples $(v_1,\dots,v_m)$ of maps
\[
v_1, \dots, v_{m-1} : \R \times \T \rightarrow \Lambda T^*M \times \R, \quad v_m : \R^- \times \T \rightarrow \Lambda T^*M \times \R,
\]
which are solutions of the Rabinowitz-Floer equation (\ref{rfleq}), and satisfy
\[
v_1(-\infty) \in W^u(z;-\nabla a), \quad v_{j+1}(-\infty) \in \phi^{-\nabla a}_{\R^+} (v_j(+\infty)), \quad \forall j=1,\dots,m-1.
\]
Th quotient of $\tilde{\mathcal{M}}^u_m(z)$ by the action of $\R^{m-1}$ given independent $s$-translations on the first $m-1$ components $v_1,\dots,v_{m-1}$ is denoted by $\mathcal{M}^u_m(z)$. The {\em space of negative half flow lines with cascades for $\mathbb{A}$} is the set
\[
\mathcal{M}^u(z) := \bigcup_{m\geq 1} \mathcal{M}^u_m(z).
\]
There is a well-defined evaluation map
\[
\mathrm{ev} : \mathcal{M}^u(z) \rightarrow \Lambda T^*M \times \R, \quad \mathrm{ev}(\underline{v}) := v_m(0), \quad \mbox{if } \underline{v} = [(v_1,\dots,v_m)].
\]
Similarly, $\tilde{\mathcal{M}}_m^s(z)$ is set of $m$-tuples $(v_1,\dots,v_m)$ of maps
\[
v_1 : \R^+ \times \T \rightarrow \Lambda T^*M \times \R, \quad v_2,\dots,v_m : \R \times \T \rightarrow \Lambda T^*M \times \R,
\]
which are solutions of the Rabinowitz-Floer equation (\ref{rfleq}), and satisfy
\[
v_{j+1}(-\infty) \in \phi^{-\nabla a}_{\R^+} (v_j(+\infty)), \quad \forall j=1,\dots,m-1, \quad v_m(+\infty) \in W^s(z;-\nabla a).
\]
Th quotient of $\tilde{\mathcal{M}}^s_m(z)$ by the action of $\R^{m-1}$ given independent $s$-translations on the last $m-1$ components $v_2,\dots,v_m$ is denoted by $\mathcal{M}^s_m(z)$. The {\em space of positive half flow lines with cascades for $\mathbb{A}$} is the set
\[
\mathcal{M}^s(z) := \bigcup_{m\geq 1} \mathcal{M}^s_m(z).
\]
There is a well-defined evaluation map
\[
\mathrm{ev} : \mathcal{M}^s(z) \rightarrow \Lambda T^*M \times \R, \quad \mathrm{ev}(\underline{v}) := v_1(0), \quad \mbox{if } \underline{v} = [(v_1,\dots,v_m)].
\]
It is also useful to denote the two components of the map $\mathrm{ev}$ as
\[
\mathrm{ev}_1: \mathcal{M}^{s/u}(z) \rightarrow \Lambda T^*M, \quad 
\mathrm{ev}_2: \mathcal{M}^{s/u}(z) \rightarrow \R.
\]
Both $\mathcal{M}^u(z)$ and $\mathcal{M}^s(z)$ are infinite dimensional objects, but as we are going to show, one can get finite dimensional manifolds by imposing Lagrangian boundary conditions for the first component of $\mathrm{ev}(\underline{v})$.

\paragraph{Definition of $\Phi$.} Let $\gamma$ be a critical point of $e$ and let $z$ be a critical point of $a$. We define $\mathcal{M}_{\Phi}(\gamma,z)$ to be the set
\[
\mathcal{M}_{\Phi}(\gamma,z) := \set{ (\underline{q},\underline{v}) \in \mathcal{W}^u(\gamma;-\nabla \mathbb{E},-\nabla e) \times \mathcal{M}^s(z)}{\pi\circ \mathrm{ev}_1 (\underline{v}) = \mathrm{ev}(\underline{q}), \; \mathrm{ev}_2(\underline{v}) = \sqrt{ \mathbb{E}(\mathrm{ev}(\underline{q})}}.
\] 
In other words, $\mathcal{M}_{\Phi}(\gamma,z)$ is the set of pairs $([(q_1,\dots,q_m)],[(v_1,\dots,v_k)])$ in $\mathcal{W}^u(\gamma;-\nabla \mathbb{E},-\nabla e) \times \mathcal{M}^s(z)$ coupled by the conditions
\begin{equation}
\label{coupling}
q_m(t) = \pi\circ u_1(0,t), \quad \eta_1(0) = \sqrt{\mathbb{E}(q_m)}, 
\end{equation}
where $v_1(s,t) = (u_1(s,t),\eta_1(s))$ for every $(s,t)\in \R^+ \times \T$. For a fixed loop $q\in \Lambda M$, the condition
\[
\pi\circ u_1(0,t) = q(t)
\]
is a $t$-dependent Lagrangian boundary condition for the solution $u_1$ of the Cauchy-Riemann type equation (\ref{rfleq1}) on the half cylinder $\R^+ \times \T$. As such, it defines a Fredholm problem, and together with the fact that $\mathcal{W}^u(\gamma)$ is a finite dimensional manifold, we deduce that the space $\mathcal{M}_{\Phi}(\gamma,z)$ is the set of solutions of a Fredholm problem. An index computation shows that the virtual dimension of $\mathcal{M}_{\Phi}(\gamma,z)$ (that is, the Fredholm index of the associated operator) is 
\begin{equation}
\label{dimMPhi}
\mathrm{virdim}\,  \mathcal{M}_{\Phi}(\gamma,z) = \mathrm{Ind} (\gamma;\mathbb{E},e) - \mu(z).
\end{equation}
The above formula can be obtained by combining the results of Appendix C in \cite{cf09} with those of Section 3.1 in \cite{as06}.
The energy bounds for the solutions $(\underline{q},\underline{v})$ in $\mathcal{M}_{\Phi}(\gamma,z)$ come from the following inequalities
\begin{equation}
\label{scala}
\mathbb{A}(z) \leq \mathbb{A} (v_j(s)) \leq \mathbb{A} (v_1(0)) \leq \sqrt{\mathbb{E}(q_m)} \leq \sqrt{\mathbb{E}(q_h)} \leq \sqrt{ \mathbb{E}(\gamma)},
\end{equation}
where the inequality in the middle is a consequence of Lemma \ref{levrel}, thanks to the coupling conditions (\ref{coupling}). Then uniform $L^{\infty}$ estimates for the solutions $v_2,\dots,v_m$ of (\ref{rfleq}) follow from Propositions \ref{etalim} and either \ref{ulim} or \ref{ulim2}. Set as before $v_1 = (u_1,\eta_1)$. Since $\eta_1(0)$ is non-negative, Remark \ref{etabdry} yields a  uniform $L^{\infty}$ estimate for $\eta_1$. The uniform $L^{\infty}$ estimate for $u_1$ follows from Proposition \ref{ulim+}, provided that we check the assumption (\ref{c3+}), the other ones being clearly satisfied. By Assumption (A) of Section \ref{morcom}, the loop $q_m$ has the uniform $C^1$ bound
\[
\sup_{t\in \T} g ( q_m' (t), q_m' (t) ) \leq C^2,
\]
where $C=C(\mathbb{E}(\gamma))$. Since $u_1 (0,t) = (q_m(t),p(t))$ for some curve $p(t)$, by (\ref{lamham}) and by the fact that $\eta_1(0)$ is non-negative we get the estimate
\[
\lambda \left( \frac{\partial u_1}{\partial t} (0,t) - \eta_1(0) X_H(u_1(0,t)) \right) = \langle p(t), q_m' (t) \rangle - \eta_1(0) g^*(p(t),p(t)) \leq C \, g^*(p(t),p(t))^{1/2}.
\]
Since on the complement of the zero section $M\subset T^*M$ the contact form $\alpha$ and the Liouville form $\lambda$ are related by the identity
\begin{equation}
\label{laval}
\lambda (q,p) = g^*(p,p)^{1/2} \alpha (q,p),
\end{equation}
the above inequality implies that for every $t\in \T$ such that $u_1(0,t)\notin M$ there holds
\[
\alpha \left( \frac{\partial u_1}{\partial t} (0,t) - \eta_1(0) X_H(u_1(0,t)) \right) \leq C,
\]
so the assumption (\ref{c3+}) of Proposition \ref{ulim+} is satisfied. We conclude that also $u_1$ has a uniform $L^{\infty}$ bound. 

The $C^{\infty}_{\mathrm{loc}}$ bounds for $v_1,\dots,v_k$ follow from the standard elliptic bootstrap, and from the fact that both bubbling off of either spheres or disks at the boundary of the half-cylinder $\R^+ \times \T$ (in the case of $u_1$) cannot occur, because the symplectic form $\omega$ is exact and because its primitive $\lambda$ vanishes identically on the fibers of $T^*M$. Together with the pre-compactness of $\mathrm{ev}(\mathcal{W}^u(\gamma))$, we conclude that the spaces $\mathcal{M}_{\Phi}(\gamma,z)$ are compact up to breaking.

The only obstruction to obtain transversality, by a generic choice of $g_{\Lambda}$, $g_{\mathbb{E}}$, $J_t$, and $g_{\mathbb{A}}$, is given by the stationary solutions in $\mathcal{M}_{\Phi}(\gamma,z)$, which are of two different kinds. If $\gamma$ is a critical point of $e$ with $\mathbb{E}(\gamma)>0$ and $z=Z^+(\gamma)$ is the corresponding critical point of $a$ with positive Rabinowitz action, then the space $\mathcal{M}_{\Phi}(\gamma,z)$ contains the stationary solution $(\gamma,z)$. Actually, the estimates (\ref{scala}) imply that in this case $\mathcal{M}_{\Phi}(\gamma,z)$ does not contain solutions other than the stationary one. Since in this case the virtual dimension of $\mathcal{M}_{\Phi}(\gamma,z)$ is zero, we would like the linearized operator at $(\gamma,z)$ to be an isomorphism. Since the linearized operator is Fredholm of index zero, this is equivalent to checking that the linearized problem has no non-zero solutions. This fact is proved in Lemma \ref{auttras1} below.

The other stationary solutions arise when $\gamma=q$ is a critical point of $e$ in the zero energy set $\mathbb{E}^{-1}(0) \cong M$, and $z$ is one of the two corresponding critical points $z_q^+=Z^+(q)$ or $z_q^-=Z^-(q)$ of $a$ on the critical set $( \crit\, \mathbb{A} ) \cap \mathbb{A}^{-1} (0) \cong S^*M$. By formulas (\ref{Indrel0}) and (\ref{dimMPhi}), 
\begin{eqnarray*}
\mathrm{virdim}\, \mathcal{M}_{\Phi}(q,z^+_q) = &  \mathrm{Ind} (q;\mathbb{E},e) - \mu(z^+_q) = &   
0 \\
\mathrm{virdim}\, \mathcal{M}_{\Phi}(q,z^-_q) =  &\mathrm{Ind} (q;\mathbb{E},e) - \mu(z^-_q) = &
n - 1.
\end{eqnarray*}
We shall prove that in both cases the linearized operator is onto. In the first case, this fact amounts to showing that the linearized problem has no non-zero solutions,
while in the second case that the space of solutions of the linearized problem has dimension $n-1$. Both facts are proved in Lemma \ref{auttras2} below (the second fact requires the inner product $g_{\mathbb{A}}$ at $z_q^-$ to be generic).

Postponing until the end of the section the proof of these automatic transversality results, we can define the chain map $\Phi$.
We choose generic data $g_{\Lambda}$, $g_{\mathbb{E}}$, $J_t$, and $g_{\mathbb{A}}$, so that for every $\gamma \in \crit\, e$ and every $z\in \crit \, a$, $\mathcal{M}_{\Phi}(\gamma,z)$ is a manifold of dimension $\mathrm{Ind}(\gamma;\mathbb{E},e) - \mu(z)$. By compactness and transversality we deduce that, when $\mu(z)=\mathrm{Ind}(\gamma;\mathbb{E},e)$, $\mathcal{M}_{\Phi}(\gamma,z)$ is a finite set, and we denote by $n_{\Phi}(\gamma,z)\in \Z_2$ its parity. Then we define
\[
\Phi: M_k(\mathbb{E},e) \rightarrow RF_k(D^*M,\lambda,a), \quad \Phi \gamma := \sum_{\substack{z\in \crit\, a \\ \mu(z) = k }} n_{\Phi}(\gamma,z)\, z,
\]
for every $\gamma\in \crit \, e$ such that $\mathrm{Ind} (\gamma;\mathbb{E},e)=k$.
A standard gluing argument shows that $\Phi$ is a chain map. 

\paragraph{Properties of $\Phi$.} If $\mathcal{M}_{\Phi}(\gamma,(x,\eta))$ is non-empty, then the loop $\gamma$ is freely homotopic to the loop $\pi\circ x$. Therefore, $\Phi$ preserves the splitting of the Morse and the Rabinowitz-Floer complexes indexed by the free homotopy classes of loops in $M$,
\begin{equation}
\label{splitPhi}
\Phi : M_*^c (\mathbb{E},e) \rightarrow RF_*^c (D^*M,\lambda,a), \quad \forall c \in [\T,M].
\end{equation}
The inequality (\ref{scala}) implies that $\Phi$ preserves the $\R$-filtrations given by the levels of $\mathbb{E}$ and $\mathbb{A}$. More precisely, for every $E\geq 0$,
\[
\Phi : M_*^{[0,E]} (\mathbb{E},e) \rightarrow RF_*^{]-\infty,\sqrt{E}]} (D^*M,\lambda,a).
\]
The following proposition lists the other main properties of the chain map $\Phi$.

\begin{prop}
\label{phiprop}
\begin{enumerate}

\item If $\gamma$ is any critical point of $e$, there holds
\[
\Phi \gamma = Z^+(\gamma) + \sum_{w \prec Z^+(\gamma)} n_{\Phi}(\gamma,w)\, w,
\]
where $\prec$ denotes the following partial order on the set $\crit\, a$: $z\prec w$ if $\mathbb{A}(z)<\mathbb{A}(w)$, or $\mathbb{A}(z)=\mathbb{A}(w)$ and $a(z)<a(w)$.

\item If $q$ is a critical point of $e$ on $\mathbb{E}^{-1}(0)$, then
\[
\Phi q \in z_q^+ + RF^{]-\infty,0[}.
\]
The restriction 
\[
\Phi : M_*^{\{0\}} (\mathbb{E},e) \cong M_*(e|_M) \rightarrow RF_*^{\{0\}} (D^*M,\lambda,a) \cong M_{*+n-1} (a|_{S^*M})
\]
to the complexes corresponding to the zero energy and action level induces the transfer homomorphism $\pi_!$ in homology:
\[
\Phi_* = \pi_! : HM_*^{\{0\}} (\mathbb{E},e) \cong H_*(M)\rightarrow HRF_*^{\{0\}}(D^*M,\lambda,a) \cong H_{*+n-1} (S^*M).
\] 

\item The chain map $\Phi$ is injective and it admits a left inverse $\hat{\Phi}  : RF_* \rightarrow M_*$, which is characterized by
\[
\ker \hat{\Phi} = RF^-,
\]
the subspace defined in (\ref{RF-}).

\item For every $k\geq 2$, $\Phi : M_k  \rightarrow 
RF_k $ is an isomorphism.

\end{enumerate}
\end{prop}

\begin{proof}
Statement (a) follows from the inequality (\ref{scala}) and from Lemmas \ref{auttras1}  and \ref{auttras2} below. The proof of (b) uses also the results of the Appendix B, and in particular statement (b) of Proposition \ref{descri}. We define $\hat{\Phi}:RF_* \rightarrow M_*$ recursively by
\begin{equation*}
\begin{split}
\hat{\Phi} \zeta &:=  0 \quad \quad\mbox{if } \zeta\in RF^-, \\
\hat{\Phi} Z^+(\gamma) &:=  \gamma -    \sum_{w\prec Z^+(\gamma)} n_{\Phi} (\gamma,w) \hat{\Phi} w. 
\end{split}\end{equation*} 
It is immediate to check that $\hat{\Phi} \Phi = \mathrm{Id}_{M_*}$. By definition, $RF^-\subset \ker \hat{\Phi}$. Any non-zero $\zeta\in RF$ has the form
\[
\zeta = \sum_{z\in \mathscr{Z}} z,
\]
where $\mathscr{Z}\subset \crit\, a$ contains elements which are maximal with respect to the order $\prec$. By the definition of $\hat{\Phi}$, if any of such maximal element is in $Z^+(\crit\, e)$ then $\hat{\Phi} \zeta\neq 0$, so $\ker \hat{\Phi} \subset RF^-$. We conclude that $\ker \hat{\Phi} = RF^-$, which proves (c). By the first identity in (\ref{bijecrit}), $RF^-_k=(0)$ for every $k\geq 2$, and statement (d) follows.
\end{proof} \qed

In general, the left inverse $\hat{\Phi}$ is not a chain map. In fact, a left inverse of a chain map is a chain map if and only if its kernel is a subcomplex, and in the case of $\hat{\Phi}$ this happens if and only if $\chi(T^*M)=0$.

\paragraph{Automatic transversality.} We conclude this section by proving the above mentioned automatic transversality results for stationary solutions in $\mathcal{M}_{\Phi}(\gamma,z)$.

\begin{lem}
\label{auttras1}
Let $\gamma$ be a critical point of $e$ with $\mathbb{E}(\gamma)>0$, and let $z=Z^+(\gamma)$ be the corresponding critical point of $a$ with positive Rabinowitz action. Then the linearization of problem $\mathcal{M}_{\Phi}(\gamma,z)$ at the stationary solution $(\gamma,z)$ has no non-zero solutions.
\end{lem}

\begin{proof}
Let $x:\T \rightarrow S^*M$ be the loop defined by $z=(x,\sqrt{\mathbb{E}(\gamma)})$.
Since $z\in \mathcal{M}_1^s(z) = \tilde{\mathcal{M}}_1^s(z)$, the solutions of the linearization of problem $\mathcal{M}^s(z)$ at $z$ are maps $\zeta=(\xi,\eta)$, with
\[
\xi: \R^+ \times \T \rightarrow TT^*M, \quad \xi(s,t) \in T_{x(t)} T^*M, \quad \eta:\R^+ \rightarrow \R,
\]
which solve the linearized Rabinowitz-Floer equation
\begin{equation}
\label{tuno}
\frac{d \zeta}{ds} (s) = - \nabla^2 \mathbb{A}(z) \zeta(s),
\end{equation}
and the asymptotic condition
\begin{equation}
\label{tdue}
\lim_{s\rightarrow +\infty} \zeta(s) \in T_z W^s(z;-\nabla a).
\end{equation}
Here, $\nabla^2 \mathbb{A}(z)$ is the Hessian of $\mathbb{A}$ at $z$, with respect to the $L^2$-inner product. It is a self-adjoint operator with compact resolvent, and kernel isomorphic to the tangent space at $z$ of $\crit \, \mathbb{A}$. 
Since $\gamma\in \mathcal{W}^u_1(\gamma) = \tilde{\mathcal{W}}_1^u(\gamma)$, the linearization of the first of the coupling conditions (\ref{coupling}) is
\begin{equation}
\label{ttre}
\theta := d\pi (x) [\xi(0,\cdot)] \in T_{\gamma} W^u \bigl(W^u(\gamma; - \nabla e); - \nabla \mathbb{E} \bigr).
\end{equation}
Finally, since $\gamma$ is a critical point of $\mathbb{E}$, the linearization of the second of the coupling conditions (\ref{coupling}) is just
\begin{equation}
\label{tquattro}
\eta(0) = 0.
\end{equation}
We must show that every solution $\zeta=(\xi,\eta)$ of (\ref{tuno}-\ref{tdue}-\ref{ttre}-\ref{tquattro}) is identically zero.

By (\ref{tuno}) and (\ref{tdue}), $\zeta(0)$ belongs to the linear subspace corresponding to the non-negative part of the spectrum in the spectral decomposition of $\nabla^2 \mathbb{A}(z)$. However, by condition (\ref{tquattro}) and by inequality (\ref{piu}) of Lemma \ref{levreldiff},
\[
d^2 \mathbb{A}(z)[\zeta(0),\zeta(0)] = d_{xx}^2 \mathbb{A}(z) [\xi(0),\xi(0)] \leq \frac{1}{2 \sqrt{\mathbb{E}(\gamma)}} d^2 \mathbb{E} (\gamma) [\theta,\theta],
\]
and the latter quantity is non-positive by condition (\ref{ttre}). Therefore, $\zeta(0)$ belongs to the kernel of $\nabla^2 \mathbb{A}(z)$ and hence $\zeta(s) \equiv \zeta(0)$ is a constant vector in $T_z W^s(z;-\nabla a)$. In particular, since $z$ is a non-degenerate critical point of $a$, 
\begin{equation}
\label{tcinque}
d^2 a(z) [\zeta(0),\zeta(0)] \geq \delta \| \zeta(0) \|^2,
\end{equation}
for some $\delta>0$. By the bijection between the kernels of $d^2 \mathbb{A}(z)$ and $d^2 \mathbb{E}(\gamma)$, $\theta$ belongs to the tangent space of $\crit \, \mathbb{E}$ at $\gamma$, and by (\ref{ttre}) it actually belongs to the tangent space of $W^u(\gamma;- \nabla e)$ at $\gamma$. In particular,
\begin{equation}
\label{tsei}
d^2 e(\gamma) [\theta,\theta] \leq - \delta' \|\theta\|^2,
\end{equation}
for some $\delta'>0$. On the other hand, the restrictions of $a$ and $e$ to the connected components of $\crit \, \mathbb{A}$ and $\crit \, \mathbb{E}$ which contain $z$ and $\gamma$ are related by the formula
\[
a(y,\sqrt{\mathbb{E}(\gamma)}) = e (\pi\circ y).
\]
If we differentiate this identity twice at $x$, we find
\[
d^2 a(z)[\zeta(0),\zeta(0)] = d^2 e(\gamma)[\theta,\theta],
\]
so (\ref{tcinque}) and (\ref{tsei}) imply that $\zeta(0)=0$ and hence $\zeta(s) \equiv 0$.
\end{proof} \qed
 
\begin{lem}
\label{auttras2}
Let $q$ be a critical point of $e$ on the zero energy set $\mathbb{E}^{-1}(0) \cong M$, and let $z^+_q = (x^+_q,0), \; z^-_q = (x^-_q,0) \in S^*M \times \{0\}$ be the critical points of $a$ in $\pi^{-1}(q)\times \{0\}$. Then the linearization of problem $\mathcal{M}_{\Phi}(q,z_q^+)$ at the stationary solution $(q,z_q^+)$ has no non-zero solutions. If the inner product $g_{\mathbb{A}}$ at $z_q^-$ is generic, then the space of solutions of the linearization of problem $\mathcal{M}_{\Phi}(q,z_q^-)$ at the stationary solution $(q,z_q^-)$ has dimension $n-1$.
\end{lem}

\begin{proof}
Let $z=(x,0)$ be one of $z_q^+=(x_q^+,0)$ or $z_q^-=(x_q^-,0)$
Arguing as in the proof of Lemma \ref{auttras1}, we must show that the space of solutions  $\zeta = (\xi,\eta)$, where
\[
\xi: \R^+ \times \T \rightarrow T_x T^*M,
 \quad \eta: \R^+ \rightarrow \R, 
\]
of the linear problem
\begin{eqnarray}
\label{ttuno}
\frac{d \zeta}{ds} (s) & = & - \nabla^2 \mathbb{A}(z) \zeta(s), \\
\label{ttdue}
\lim_{s\rightarrow +\infty} \zeta(s) & \in & T_z W^s(z;-\nabla a), \\
\label{tttre}
\theta:= d\pi (x) [\xi(0,\cdot)] & \in & T_q W^u(q; - \nabla e), \\
\label{ttquattro}
\eta(0) & = & 0,
\end{eqnarray}
has either dimension zero in the case of $z_q^+$, or $n-1$ in the case of $z_q^-$. 
Conditions (\ref{ttuno}) and (\ref{ttdue}) are equivalent to the fact that $\zeta(0)$ belongs to the linear subspace corresponding to the non-negative part of the spectrum in the spectral decomposition of $\nabla^2 \mathbb{A}(z)$. By (\ref{ttquattro}) and by an explicit computation of the second differential of $\mathbb{A}$ at $(x,0)$, we find
\[
d^2 \mathbb{A}(x,0)[\zeta(0),\zeta(0)] =
\int_{\T} \omega\Bigl( \xi(0,t), \frac{\partial \xi}{\partial t} (0,t) \Bigr)\, dt = - \int_{\T} \frac{d}{dt} \lambda(\xi(t))\, dt = 0,
\]
where we have used the fact that $\xi(0,\cdot)$ is a loop in $T_x T^*M$ such that $d\pi(x)[\xi(0,t)]$ does not depend on $t$, by (\ref{tttre}). Therefore, $\zeta(0)$ belongs to the kernel of $\nabla^2 \mathbb{A}(z)$. Taking also (\ref{ttdue}) into account, we deduce that $\zeta(s) \equiv \zeta(0)$ is of the form $\zeta(s) = (\xi,0)$, where $\xi\in T_{x} S^*M$ does not depend on $t$, and $(\xi,0) \in T_{z} W^s(z;-\nabla a)$. We conclude that the space of solutions of the linearized problem is isomorphic to the linear space
\begin{equation}
\label{space}
\set{(\xi,0) \in T_{z} W^s(z;-\nabla a)}{d\pi(x)[\xi] \in T_q W^u(q;-\nabla e)},
\end{equation}
so we must show that the dimension of this space is zero in the case of $z_q^+$, and $n-1$ in the case of $z_q^-$. 
By the assumptions (A2-A3), $a(x,0)\leq e(\pi(x)) + 1/2$ for every $x\in S^*M$, and the equality holds when $x=x_q^+$. Since $(x_q^+,0)$ is a critical point of $a$, and $\pi(x_q^+)=q$ is a critical point of $e$, we deduce the inequality
\[
d^2 a(z_q^+)[(\xi,0),(\xi,0)] \leq d^2 e(q) \bigl[ d\pi(x_q^+)[\xi], d\pi(x_q^+)[\xi] \bigr] , \quad \forall \xi\in T_{x_q^+} S^*M.
\]
The above inequality implies that when $z=z_q^+$ the space (\ref{space}) is zero. 
By condition (A4),
\[
\dim d\pi(x^-_q)^{-1} T_q W^u(q;-\nabla e) = \ind(q;e) + n - 1, \quad
\codim\, T_{z_q^-} W^s(z_q^-;-\nabla a) = \ind(q;e),
\]
so if the inner product $g_{\mathbb{A}}$ at $z_q^-$ is generic, the above pair of subspaces is transverse, and hence the space (\ref{space}) for $z=z_q^-$ has dimension $n-1$.
\end{proof} \qed
  
\section{The chain map $\Psi$} 
\label{psisec}

\paragraph{Definition of $\Psi$.} Let $z$ be a critical point of $a$ and let $\gamma$ be a critical point of $e$. We define $\mathcal{M}_{\Psi}(z,\gamma)$ to be the set
\[
\mathcal{M}_{\Psi}(z,\gamma) := \set{ (\underline{v},\underline{q}) \in \mathcal{M}^u(z) \times \mathcal{W}^u(\gamma; - \nabla \mathbb{E},\nabla e)}{\pi \circ \mathrm{ev}_1(\underline{v}) = \mathrm{ev} (\underline{q})(-\cdot), \; \mathrm{ev}_2 (\underline{v}) = - \sqrt{ \mathbb{E} (\mathrm{ev} (\underline{q})) } }.
\]
Equivalently, $\mathcal{M}_{\Psi}(z,\gamma)$ is the set of pairs $([(v_1,\dots,v_k)],[(q_1,\dots,q_m)])$ in $\mathcal{M}^u(z) \times \mathcal{W}^u(\gamma; - \nabla \mathbb{E},\nabla e)$ coupled by the conditions
\begin{equation}
\label{coupling2}
q_m(t) = \pi \circ u_k(0,-t), \quad \eta_k(0) = - \sqrt{\mathbb{E}(q_m)},
\end{equation}
where $v_k(s,t) = (u_k(s,t),\eta_k(s))$ for every $(s,t) \in \R^- \times \T$. 
Notice that, as it was anticipated in Section \ref{morcom}, we are using the Morse function $-e$ on $\crit\, \mathbb{E}$.
Again, the facts that $\mathcal{W}^u(\gamma;-\nabla \mathbb{E},\nabla e)$ is finite dimensional and that the first of the coupling conditions in (\ref{coupling2}) is a (parametric and $t$-dependent) Lagrangian boundary condition, imply that elements of $\mathcal{M}_{\Psi}(z,\gamma)$ are solutions of a Fredholm problem, and an index computation (see again Appendix C in \cite{cf09} and Section 3.1 in \cite{as06}) gives
\begin{equation}
\label{virdim2}
\mathrm{virdim}\, \mathcal{M}_{\Psi}(z,\gamma) = \mu(z) + \mathrm{Ind} (\gamma;\mathbb{E},-e) - 1.
\end{equation}
By the second inequality of Lemma \ref{levrel}, the elements $(\underline{v},\underline{q})$ of $\mathcal{M}_{\Psi}(z,\gamma)$ satisfy the following action bounds
\begin{equation}
\label{scala2}
\mathbb{A}(z) \geq \mathbb{A}(v_j(s)) \geq \mathbb{A}(v_k(0)) \geq - \sqrt{ \mathbb{E}(q_m)} \geq - \sqrt{ \mathbb{E}(q_h)} \geq - \sqrt{ \mathbb{E}(\gamma)},
\end{equation}
which provide us with uniform energy bounds. As before, uniform $L^{\infty}$ estimates for the solutions $v_1,\dots,v_{k-1}$ of (\ref{rfleq}) on the whole cylinder follow from Proposition \ref{etalim} and \ref{ulim2}, and there remains to treat the case of the solution $v_k = (u_k,\eta_k)$ on the half-cylinder $\R^- \times \T$. By the second coupling condition in (\ref{coupling2}), $\eta_k(0) \leq 0$, so the uniform $L^{\infty}$ estimate for $\eta_k$ follows fom Remark \ref{etabdry}. The uniform $L^{\infty}$ estimate for $u_k$ follows from Proposition \ref{ulim+} and Remark \ref{ubdry}, provided that we prove a uniform lower bound for the quantity
\[
\alpha \left( \frac{\partial u_k}{\partial t} (0,t) - \eta_k(0) X_H(u_k(0,t)) \right).
\]
Taking into account the identity (\ref{laval}), such a lower bound follows from the estimate
\[
\lambda \left( \frac{\partial u_k}{\partial t} (0,t) - \eta_k(0) X_H(u_k(0,t)) \right) = 
\langle p(t), -q_m'(-t) \rangle - \eta_k(0) g^* (p(t),p(t)) \geq - C g^*(p(t),p(t))^{1/2},
\]
where $u_k(0,t) = (q_m(-t),p(t))$ and $\eta_k(0)\leq 0$, by (\ref{coupling2}), and $C>0$ is such that
\[
\sup_{t\in \T} g(q'(t),q'(t)) \leq C^2,
\]
for every $q$ which belongs to the unstable manifold of some critical point $\gamma_0$ of $\mathbb{E}$ with energy $\mathbb{E}(\gamma_0) \leq \mathbb{E}(\gamma)$ (see Assumption (A) of Section \ref{morcom}).   
Then, the standard bubbling off and bootstrap argument implies that the spaces $\mathcal{M}_{\Psi}(z,\gamma)$ are compact up to breaking.

As in the case of $\mathcal{M}_{\Phi}$, $\mathcal{M}_{\Psi}(z,\gamma)$ may contain stationary solutions. Indeed, if $\gamma$ is a critical point of $e$ with positive energy, and $z=Z^-(\gamma)$ is the corresponding critical point of $a$ with negative Rabinowitz action, then the space $\mathcal{M}_{\Psi}(z,\gamma)$ consists of the stationary solution $(z,\gamma)$, because of (\ref{scala2}). In this case, $\mu(z) = 1 - \mathrm{Ind}(\gamma;\mathbb{E},-e)$, by (\ref{Indrel}), so the virtual dimension of $\mathcal{M}_{\Psi}(z,\gamma)$ is zero. The following lemma, whose proof uses the estimate (\ref{meno}) of Lemma \ref{levreldiff} and is analogous to the proof of Lemma \ref{auttras1}, implies that automatic transversality holds at such stationary solutions:

\begin{lem}
\label{auttras3}
Let $\gamma$ be a critical point of $e$ with $\mathbb{E}(\gamma)>0$, and let $z=Z^-(\gamma)$ be the corresponding critical point of $a$ with negative Rabinowitz action. Then the linearization of problem $\mathcal{M}_{\Psi}(z,\gamma)$ at the stationary solution $(z,\gamma)$ has no non-zero solutions.
\end{lem}

The other stationary solutions correspond to constant geodesics. Indeed, if $q$ is a critical point of $e$ on the zero energy set $\mathbb{E}^{-1}(0) \cong M$, and $z_q^+=Z^+(q)$, $z_q^-=Z^-(q)$ are the corresponding critical points of $a$ on the critical set $(\crit\, \mathbb{A} ) \cap \mathbb{A}^{-1}(0) \cong S^*M$, then $(z_q^+,q)$ and $(z_q^-,q)$ are stationary solutions in $\mathcal{M}_{\Psi}(z_q^+,q)$ and $\mathcal{M}_{\Psi}(z_q^-,q)$, respectively. 
By formulas (\ref{Indrel0}) and (\ref{virdim2}), the virtual dimension of these spaces are
\begin{eqnarray*}
\mathrm{virdim}\, \mathcal{M}_{\Psi}(z^+_q,q) = &  \mu(z^+_q) + \mathrm{Ind} (q;\mathbb{E},-e) - 1= &   
n-1 \\
\mathrm{virdim}\, \mathcal{M}_{\Psi}(z^-_q,q) =  & \mu(z^-_q) + \mathrm{Ind} (q;\mathbb{E},-e) - 1= & 0.
\end{eqnarray*}
So automatic transversality at $(z_q^+,q)$ and $(z_q^-,q)$ follows from the lemma below, whose proof is analogous to the proof of Lemma \ref{auttras2}:

\begin{lem}
\label{auttras4}
Let $q$ be a critical point of $e$ on the zero energy set $\mathbb{E}^{-1}(0) \cong M$, and let $z^+_q = (x^+_q,0), \; z^-_q = (x^-_q,0) \in S^*M \times \{0\}$ be the critical points of $a$ in $\pi^{-1}(q) \times \{0\}$. Then the linearization of problem $\mathcal{M}_{\Psi}(z_q^-,q)$ at the stationary solution $(z_q^-,q)$ has no non-zero solutions. If the inner product $g_{\mathbb{A}}$ at $z_q^+$ is generic, then the space of solutions of the linearization of problem $\mathcal{M}_{\Psi}(z_q^+,q)$ at the stationary solution $(z_q^+,q)$ has dimension $n-1$.
\end{lem}

Once automatic transversality at the stationary solutions has been checked, standard transversality arguments in Floer homology implies that for generic $J_t$, $g_{\mathbb{A}}$, $g_{\Lambda}$, and $g_{\mathbb{E}}$, the space $\mathcal{M}_{\Psi}(z,\gamma)$ is a manifold of dimension $\mu(z) + \mathrm{Ind}(\gamma;\mathbb{E},-e) - 1$, for every $z\in \crit\, a$ and $\gamma\in \crit\, e$. Compactness and transversality imply that when $\mathrm{Ind}(\gamma) = 1 - \mu(z)$, $\mathcal{M}_{\Psi}(z,\gamma)$ is a finite set, whose parity we denote by $n_{\Psi}(z,\gamma) \in \Z_2$. Then we define the homomorphism
\[
\Psi: RF_k(D^*M,\lambda,a) \rightarrow M^{1-k}(\mathbb{E},-e), \quad \Psi z := \sum_{\substack{\gamma \in \crit\, e \\ \mathrm{Ind} (\gamma; \mathbb{E}, -e) = 1-k}} n_{\Psi}(z,\gamma) \, \gamma,
\]
for every $z\in \crit\, a$ with $\mu(z) = k$. Notice again the use of the opposite Morse function on $\crit \, \mathbb{E}$.
A standard gluing argument implies that $\Psi$ is a chain map from the chain complex $RF_*$ to the chain complex $M^{1-*}$, that is
\[
\delta \Psi \zeta = \Psi \partial \zeta, \quad \forall \zeta \in RF_*( D^*M,\lambda,a).
\]

\paragraph{Properties of $\Psi$.} The first coupling condition in (\ref{coupling2}) shows that if $\mathcal{M}_{\Psi}((x,\eta),\gamma)$ is non-empty, then the loop $\pi\circ x$ is freely homotopic to the loop $\gamma(-\cdot)$. Therefore, 
\[
\Psi: RF_*^c(D^*M,\lambda,a) \rightarrow M^{1-*}_{-c}(\mathbb{E},-e), \quad \forall c\in [\T,M],
\]
where $-c$ denotes the class obtained by inverting the parameterization of the loops in $c$.

The inequality (\ref{scala2}), implies that $\Psi$ preserves the $\R$-filtrations, meaning that for every $E\geq 0$ there holds
\[
\Psi : RF_*^{]-\infty,-\sqrt{E}]} (D^*M,\lambda,a) \rightarrow M^{1-*}_{[\sqrt{E},+\infty[} (\mathbb{E},-e).
\]
The remaining relevant properties of $\Psi$ are listed in the following:

\begin{prop}
\label{psiprop}
\begin{enumerate}

\item If $\gamma$ is any critical point of $e$, there holds
\[
\Psi Z^-(\gamma) = \gamma + \sum_{\beta \succ \gamma} n_{\Psi} (Z^-(\gamma),\beta)\, \beta,
\]
where $\beta \succ \gamma$ means that either $\mathbb{E}(\beta) > \mathbb{E}(\gamma)$, or $\mathbb{E}(\beta) = \mathbb{E}(\gamma)$ and $-e(\beta) > -e(\gamma)$.

\item If $q$ is a critical point of $e$ on $\mathbb{E}^{-1}(0)$, then
\[
\Psi Z^-(q) \in q + M^{1-*}_{]0,+\infty[}(\mathbb{E},-e).
\]
The restriction 
\[
\Psi : RF_*^{\{0\}} (D^*M,\lambda,a) \cong M_{*+n-1} (a|_{S^*M}) \rightarrow M^{1-*}_{\{0\}} (\mathbb{E},-e) \cong M^{1-*} (-e|_M) \cong M_{*+n-1} (e|_M) 
\]
to the complexes corresponding to the zero energy and action level induces the homomorphism $\pi_*$ in homology:
\[
\Psi_* = \pi_* :  HRF_*^{\{0\}} \cong H_{*+n-1}(S^*M) \rightarrow HM^{1-*}_{\{0\}} (\mathbb{E},-e) \cong H^{1-*}(M) \cong H_{*+n-1}(M)  .
\] 

\item The chain map $\Psi$ is surjective and it admits a right inverse $\hat{\Psi}  : M^{1-*} \rightarrow RF_*$, which is characterized by
\[
\mathrm{Im}\, \hat{\Psi} = RF^-,
\]
the subspace defined in (\ref{RF-}).

\item For every $k\leq -1$, $\Psi : RF_k  \rightarrow 
M^{1-k}$ is an isomorphism. 

\end{enumerate}
\end{prop}

\begin{proof}
Statements (a) and (b) follow from the inequalities (\ref{scala2}) and Lemmas \ref{auttras3}, \ref{auttras4}, together with the results of the Appendix B, in particular statement (c) of Proposition \ref{descri}.

The homomorphism $\hat{\Psi} : M^{1-*} \rightarrow RF_* $ can be defined
as the homomorphism which maps $\gamma\in \crit \, e$ into
\[
\hat{\Psi} \gamma := Z^-(\gamma) - \sum_{\beta \succ \gamma} n_{\hat{\Psi}}(\gamma,Z^-(\beta))\,  Z^-(\beta),
\]
where the coefficients $n_{\hat{\Psi}}(\gamma,Z^-(\beta))$ are defined recursively by
\[
n_{\hat{\Psi}}(\gamma,Z^-(\beta)) := n_{\Psi}(Z^-(\gamma),\beta) + \sum_{\gamma \prec \alpha \prec \beta} n_{\hat{\Psi}} (\gamma,Z^-(\alpha)) \, n_{\Psi} (Z^-(\alpha),\beta), \quad \forall \gamma,\beta\in \crit\, e.
\]
It is immediate to check that $\hat{\Psi}$ is a right inverse of $\Phi$ and that its image is $RF^-$, proving (c). Statement (d) follows from the last identity in  (\ref{bijecrit}).
\end{proof} \qed

Again, the right inverse $\hat{\Psi}$ is a chain map if and only if its image is a subcomplex, so if and only if $\chi(T^*M)=0$.

\section{The chain homotopy $P$}

The aim of this section is to prove that the composition $\Psi\circ \Phi$,
\[
M_* (\mathbb{E},e) \stackrel{\Phi}{\rightarrow} RF_* (D^*M,\lambda,a) \stackrel{\Psi}{\rightarrow} M^{1-*} (\mathbb{E},-e),
\]
is chain homotopic to zero. More precisely, we shall prove the following:

\begin{prop}
\label{Pprop}
There exists a homomorphism
\[
P : M_*(\mathbb{E},e) \rightarrow M^{-*} (\mathbb{E},-e)
\]
such that
\begin{equation}
\label{chainhom}
\Psi \circ \Phi = P \partial + \delta P,
\end{equation}
and $P q_{\min}\in M^0_{]0,+\infty[} (\mathbb{E},-e)$.
\end{prop}

Notice that since $M_k=M^k=(0)$ for $k<0$, the chain homotopy $P$ can be non-zero only for $k=0$. Therefore, (\ref{chainhom})
is equivalent to
\[
\Psi \circ \Phi|_{M_1} = P \partial, \quad \Psi \circ \Phi|_{M_0} = \delta P.
\]

\paragraph{Flow arcs with cascades.}
In order to describe the spaces of maps whose counting produces the homomorphism $P$, we need to introduce {\em flow arcs with cascades} for the Rabinowitz action functional $\mathbb{A}$. The flow arcs with zero cascades are the pairs $(S,v)$, where $S>0$ and
\[
v : [-S,S] \times \T \rightarrow T^*M \times \R
\]
is a solution of the Rabinowitz-Floer equation (\ref{rfleq}). The set of flow arcs with zero cascades is denoted by $\mathcal{A}_0$. If $m\geq 1$ is an integer, we denote by $\tilde{\mathcal{A}}_m$ the set of $(m+1)$-tuples $(v_0,v_1,\dots,v_m)$ of maps
\begin{eqnarray*}
& v_0 : [0,+\infty[ \times \T \rightarrow T^*M \times \R, \quad v_m : ]-\infty,0] \times \T \rightarrow T^*M \times \R, \\
& v_j : \R \times \T \rightarrow \rightarrow T^*M \times \R, \quad \forall j=1,\dots,m-1, 
\end{eqnarray*}
which solve the Rabinowitz-Floer equation (\ref{rfleq}) and satisfy
\[
v_j(-\infty) \in \phi_{\R^+}^{-\nabla a} \bigl(v_{j-1}(+\infty)\bigr), \quad \forall j=1,\dots,m.
\]
The set of arcs with $m$ cascades $\mathcal{A}_m$ is the quotient of $\tilde{\mathcal{A}}_m$ by the action of $\R^{m-1}$ given by $m-1$ independent $s$-translations on the middle components $v_1,\dots,v_{m-1}$. The set of arcs with cascades is the disjoint union
\[
\mathcal{A} := \bigcup_{m\geq 0} \mathcal{A}_m.
\]
There are natural evaluation maps
\[
\mathrm{ev}^- = (\mathrm{ev}^-_1,\mathrm{ev}_2^-) : \mathcal{A} \rightarrow T^*M \times \R, \quad \mathrm{ev}^+ = (\mathrm{ev}^+_1,\mathrm{ev}_2^+) : \mathcal{A} \rightarrow T^*M \times \R,
\]
defined by
\begin{equation*}
\begin{split}
\mathrm{ev}^-(\underline{v}) & := \left\{ \begin{array}{ll} v(-S) & \mbox{if } \underline{v} = (S,v) \in \mathcal{A}_0, \\ v_0(0) & \mbox{if } \underline{v} = [(v_0,\dots,v_m)] \in \mathcal{A}_m \mbox{ with } m\geq 1,\end{array} \right. \\
\mathrm{ev}^+(\underline{v}) & := \left\{ \begin{array}{ll} v(S) & \mbox{if } \underline{v} = (S,v) \in \mathcal{A}_0, \\ v_m(0) & \mbox{if } \underline{v} = [(v_0,\dots,v_m)] \in \mathcal{A}_m\mbox{ with } m\geq 1. \end{array} \right.
\end{split} \end{equation*}

\paragraph{The spaces $\mathbf{\mathcal{M}_P}$.}
Let $\gamma^-$ and $\gamma^+$ be critical points of $e$. We define
\begin{eqnarray*}
\mathcal{M}_P(\gamma^-,\gamma^+) := \Bigl\{ (\underline{q}^-, \underline{v}, \underline{q}^+) \in \mathcal{W}^u(\gamma^-; - \nabla \mathbb{E}, - \nabla e) \times \mathcal{A} \times \mathcal{W}^u(\gamma^+; - \nabla \mathbb{E}, \nabla e) \, \Big| \\
\pi\circ \mathrm{ev}^-_1(\underline{v}) = \mathrm{ev}(\underline{q}^-), \; \mathrm{ev}^-_2(\underline{v}) = \sqrt{\mathbb{E}(\mathrm{ev}(\underline{q}^-))}, \; \pi\circ \mathrm{ev}^+_1(\underline{v}) = \mathrm{ev}(\underline{q}^+)(-\cdot), \; \mathrm{ev}^+_2(\underline{v}) = - \sqrt{\mathbb{E}(\mathrm{ev}(\underline{q}^+))} \Bigr\}.
\end{eqnarray*}
Equivalently, $\mathcal{M}_P(\gamma^-,\gamma^+)$ is the set of triplets $([(q_1^-,\dots,q_h^-)], \underline{v}, [(q_1^+,\dots,q_k^+)])$ in 
\[
\mathcal{W}^u(\gamma^-; - \nabla \mathbb{E}, - \nabla e) \times \mathcal{A} \times \mathcal{W}^u(\gamma^+; - \nabla \mathbb{E}, \nabla e),
\]
which are coupled by the following conditions:
\begin{eqnarray*}
\mbox{if } \underline{v} = (S,v) = (S,(u,\eta)) \in \mathcal{A}_0 \mbox{ then } \left\{ \begin{array}{ll} \pi \circ u(-S,t) = q_h^-(t) , & \eta(-S) = \sqrt{\mathbb{E}(q_h^-)}, \\ \pi \circ u(S,t) = q_k^+(-t), & \eta(S) = - \sqrt{\mathbb{E}(q_k^+)}, \end{array} \right. \\
\mbox{if } \underline{v} = [(v_0,\dots,v_m)] \in \mathcal{A}_m, \; m\geq 1, \mbox{ then } \left\{ \begin{array}{ll} \pi \circ u_0(0,t) = q_h^-(t), & \eta_0(0) = \sqrt{\mathbb{E}(q_h^-)}, \\ \pi \circ u_m(0,t) = q_k^+(-t), & \eta_m(0) = - \sqrt{\mathbb{E}(q_k^+)}, \end{array} \right.
\end{eqnarray*}
where $v_0=(u_0,\eta_0)$ and $v_m=(u_m,\eta_m)$. An index computation shows that the virtual dimension of $\mathcal{M}_P(\gamma^-,\gamma^+)$ is
\[
\mathrm{virdim}\, \mathcal{M}_P(\gamma^-,\gamma^+) = \mathrm{Ind}\, (\gamma^-;\mathbb{E},e) + \mathrm{Ind}\, (\gamma^+;\mathbb{E},-e).
\]
We are interested in the spaces $\mathcal{M}_P$ when the above virtual dimension is either 0 or 1. Notice that if $q$ is a critical point of $e$ on $\mathbb{E}^{-1}(0) \cong M$, then
\[
\mathrm{virdim}\, \mathcal{M}_P(q,q) = \ind (q;e) + \ind (q;-e) = \dim M \geq 2.
\]
Moreover, we may assume that if $\gamma$ is a critical point of $e$ corresponding to a non-constant geodesic, then $\gamma(-\cdot)$ is not a critical point of $e$. These considerations show that when the virtual dimension of $\mathcal{M}_P(\gamma^-,\gamma^+)$ does not exceed 1, then this set does not contain stationary solutions. Therefore, for a generic choice of $g_{\Lambda}$, $g_{\mathbb{E}}$, $J_t$, and $g_{\mathbb{A}}$, transversality for the problem $\mathcal{M}_P(\gamma^-,\gamma^+)$ holds, whenever
\begin{equation}
\label{ppp0}
\mathrm{Ind}\, (\gamma^-;\mathbb{E},e) + \mathrm{Ind}\, (\gamma^+;\mathbb{E},-e) \leq 1,
\end{equation}
and, in this case, $\mathcal{M}_P(\gamma^-,\gamma^+)$ is a manifold of dimension
\[
\dim \mathcal{M}_P(\gamma^-,\gamma^+) = \mathrm{Ind}\, (\gamma^-;\mathbb{E},e) + \mathrm{Ind}\, (\gamma^+;\mathbb{E},-e).
\]
By Lemma \ref{levrel}, the elements $(\underline{q}^-,\underline{v},\underline{q}^+)$ of $\mathcal{M}_P(\gamma^-,\gamma^+)$ satisfy the energy estimates
\begin{equation}
\label{Eb1}
\begin{split}
\sqrt{\mathbb{E}(\gamma^-)} \geq \sqrt{\mathbb{E}(q_j^-)} \geq \sqrt{\mathbb{E}(q_h^-)} \geq \mathbb{A}(v(-S)) \geq \mathbb{A}(v(s)) \\ \geq \mathbb{A}(v(S)) \geq - \sqrt{\mathbb{E}(q_k^+)} \geq - \sqrt{\mathbb{E}(q_j^+)} \geq - \sqrt{\mathbb{E}(\gamma^+)},
\end{split}
\end{equation}
for every $s\in [-S,S]$, in the case $\underline{v} = (S,v)\in \mathcal{A}_0$, and
\begin{equation}
\label{Eb2}
\begin{split}
\sqrt{\mathbb{E}(\gamma^-)} \geq \sqrt{\mathbb{E}(q_j^-)} \geq \sqrt{\mathbb{E}(q_h^-)} \geq \mathbb{A}(v_0(0)) \geq \mathbb{A}(v_j(s)) \\ \geq \mathbb{A}(v_m(0)) \geq - \sqrt{\mathbb{E}(q_k^+)} \geq - \sqrt{\mathbb{E}(q_j^+)} \geq - \sqrt{\mathbb{E}(\gamma^+)},
\end{split}
\end{equation}
in the case $\underline{v}=[(v_0,\dots,v_m)]\in \mathcal{A}_m$ with $m\geq 1$. These energy estimates imply uniform bounds on all the derivatives for the elements of $\mathcal{M}_P(\gamma^-,\gamma^+)$, again using Propositions \ref{etalim}, \ref{ulim2}, \ref{ulim+}, Remarks \ref{etabdry} and \ref{ubdry}, together with bubbling off and bootstrap arguments.

The main point, in order to get the compactness required to define $P$ and to check that it is a chain homotopy between $\Psi\circ \Phi$ and the zero map, is to prove that the arcs with zero cascades in $\mathcal{M}_P(\gamma^-,\gamma^+)$ cannot shrink to zero. More precisely, we shall prove the following:

\begin{lem}
\label{X}
For every pair of critical points $\gamma^-$, $\gamma^+$ of $e$ which satisfy (\ref{ppp0}),
there exists a positive number $\sigma = \sigma(\gamma^-,\gamma^+)$ such that for every $(\underline{q}^-,\underline{v},\underline{q}^+)$ in $\mathcal{M}_P(\gamma^-,\gamma^+)$ with $\underline{v} = (S,v) \in \mathcal{A}_0$ there holds $S\geq \sigma$. 
\end{lem}

\begin{proof}
The Hamiltonian $H(q,p) = 1/2(g^*(p,p)-1)$ is bounded from below by $-1/2$, so by equation (\ref{rfleq2}),
\[
\eta(-S) - \eta(S) = - \int_{-S}^S \eta'(s) \, ds= - \int_{-S}^S \int_{\T} H(u(s,t))\, dt\, ds \leq S.
\]
Therefore, it is enough to find a positive lower bound for the number $\eta(-S) - \eta(S)$.
  
This positive lower bound is easily found when one of the geodesics $\gamma^-$ or $\gamma^+$ is not contractible. In this case, we may assume that both geodesics are not contractible (otherwise $\mathcal{M}_P(\gamma^-,\gamma^+)$ is empty), and the coupling condition for $(S,v) = (S,(u,\eta))$ implies that
\[
\eta(-S) \geq \sqrt{\delta}, \quad \eta(S) \leq - \sqrt{\delta},
\]
where 
\[
\delta := \inf_{\gamma \in \Lambda M \setminus \Lambda^0 M}\mathbb{E}(\gamma) > 0.
\]

Another trivial case occurs when both $\gamma^-$ and $\gamma^+$ are constant loops, because in this case the assumption (\ref{ppp0}) on the indices implies that $\mathcal{M}_P(\gamma^-,\gamma^+)$ is empty. In fact, assume by contradiction that $\mathcal{M}_P(\gamma^-,\gamma^+)$ contains an element $(\underline{q}^-,\underline{v},\underline{q}^+)$. Since the sets
\begin{eqnarray*}
& \mathrm{ev} \bigl( \mathcal{W}^u(\gamma^-;-\nabla \mathbb{E},-\nabla e) \bigr) = W^u(\gamma^-;-\nabla e), \\ 
& \mathrm{ev} \bigl( \mathcal{W}^u(\gamma^+;-\nabla \mathbb{E},\nabla e \bigr) = 
W^u(\gamma^+;\nabla e) = W^s(\gamma^+;-\nabla e),
\end{eqnarray*}
consist of constant loops, the energy estimate (\ref{Eb1}) implies that 
\[
0 \geq \mathbb{A}(\mathrm{ev}^-(\underline{v})) \geq \mathbb{A}(\mathrm{ev}^+(\underline{v})) \geq 0.
\]
Therefore, $\underline{v}$ is constant, and $\pi(\mathrm{ev}_1^-(\underline{v})) =  \pi(\mathrm{ev}_1^+(\underline{v}))$ is a point which belongs to both $W^u(\gamma^-;-\nabla e)$ and $W^s(\gamma^+;-\nabla e)$. Since the flow of $-\nabla e$ is Morse-Smale, we deduce that
\[
0 \leq \ind (\gamma^-;e) - \ind (\gamma^+;e) = \mathrm{Ind} \, (\gamma^-;\mathbb{E},e) - \dim M + \mathrm{Ind}\, (\gamma^+; \mathbb{E},-e).
\]
which contradicts (\ref{ppp0}).

It remains to treat the case where $\gamma^-$ and $\gamma^+$ are both contractible, but not both constant.
When 
\[
\mathrm{Ind}\, (\gamma^-;\mathbb{E},e) + \mathrm{Ind}\, (\gamma^+;\mathbb{E},-e) = 0,
\]
that is when both the indices are 0, the positive lower bound is also easily found. Indeed, in this case 
\[
\mathrm{ev} \bigl( \mathcal{W}^u(\gamma^-;- \nabla \mathbb{E}, -\nabla e) \bigr) = \{\gamma^-\}, \quad \mathrm{ev} \bigl(\mathcal{W}^u(\gamma^+;- \nabla \mathbb{E}, \nabla e) \bigr) = \{\gamma^+\},
\]
so the coupling condition implies that
\[
\eta(-S) - \eta(S) = \sqrt{\mathbb{E}(\gamma^-)} + \sqrt{\mathbb{E}(\gamma^+)} \geq \sqrt{\epsilon},
\]
where
\[
\epsilon := \inf_{\substack{\gamma\in \crit{E} \\ \mathbb{E}(\gamma)>0}} \mathbb{E}(\gamma) > 0.
\]

It remains to consider the case where $\gamma^-$ and $\gamma^+$ are contractible, not both constant, and
\[
\mathrm{Ind}\, (\gamma^-;\mathbb{E},e) + \mathrm{Ind}\, (\gamma^+;\mathbb{E},-e) = 1.
\]
We treat only the case
\[
\mathrm{Ind}\, (\gamma^-;\mathbb{E},e) = 1, \quad \mathrm{Ind}\, (\gamma^+;\mathbb{E},-e) =0,
\]
because the case $\mathrm{Ind}\, (\gamma^-;\mathbb{E},e) = 0$,  $\mathrm{Ind}\, (\gamma^+;\mathbb{E},-e) =1$ is completely analogous. Again,
\[
\mathrm{ev} \bigl( \mathcal{W}^u(\gamma^+; -\nabla \mathbb{E}, \nabla e) \bigr) = \{\gamma^+\},
\]
so if $\gamma^+$ is not constant then
\[
\eta(-S) - \eta(S) \geq - \eta(S) = \sqrt{\mathbb{E}(\gamma^+)} \geq \sqrt{\epsilon}.
\]
Therefore, we may assume that $\gamma^+ = q_{\max}$, the only critical point of $-e$ on $\mathbb{E}^{-1}(0)$ with Morse index 0, and that $\gamma^-$ is not constant.
If
\[
\ind (\gamma^-;\mathbb{E}) = 0 \quad \mbox{and} \quad \ind(\gamma^-;e) = 1,
\]
the set 
\[
\mathrm{ev}\bigl(\mathcal{W}^u(\gamma^-;-\nabla \mathbb{E},-\nabla e)\bigr) = W^u(\gamma^-;-\nabla e)
\]
is contained in the connected component of $\crit\, \mathbb{E}$ which contains $\gamma^-$, so 
\[
\eta(-S) - \eta(S) \geq \eta(-S) \geq \inf_{W^u(\gamma^-;-\nabla e)} \sqrt{\mathbb{E}} \geq \sqrt{\epsilon}.
\]
Therefore, we may assume that
\[
\ind (\gamma^-;\mathbb{E}) = 1 \quad \mbox{and} \quad \ind(\gamma^-;e) = 0.
\]
It follows that the set
\[
\mathrm{ev} \bigl(\mathcal{W}^u(\gamma^-;-\nabla \mathbb{E}, - \nabla e) \bigr) 
\]
is a one-dimensional curve with two (possibly coinciding) limiting points, which are relative minima of both $\mathbb{E}$ and $e$. If they are both non-constant geodesics, we get again that $\eta(-S) \geq \sqrt{\epsilon}$, so we may assume that at least one of them is a constant loop. Being also a minimum of $e$, such a limiting point is $q_{\min}$, and we deduce that
\begin{equation}
\label{dista}
d:= \dist \Bigl( \mathrm{ev} \bigl(\mathcal{W}^u(\gamma^-;-\nabla \mathbb{E}, - \nabla e) \bigr) \cap \mathbb{E}^{-1}(0), q_{\max} \Bigr) >0.
\end{equation}
Assume by contradiction that there is a sequence 
\[
(\underline{q}_h^-,\underline{v}_h,\underline{q}_h^+)_{h\in \N} \subset \mathcal{M}_P(\gamma^-,\gamma^+)\quad  \mbox{with}\quad\underline{v}_h = (S_h,v_h) = (S_h, (u_h,\eta_h)),
\]
such that 
\[
S_h \rightarrow 0 \quad \mbox{and}\quad \eta_h(-S_h) = \sqrt{\mathbb{E}(q_h(-S_h,\cdot))} \rightarrow 0,
\]
where 
\[
q_h (s,t) := \pi \circ u_h (s,t), \quad q_h(-S_h,\cdot) \in \mathrm{ev} \bigl( \mathcal{W}^u(\gamma^-;-\nabla \mathbb{E})\bigr), \quad q_h(S_h,\cdot) = q_{\max}.
\]
Therefore, $q_h (-S_h,\cdot)$ converges  in $W^{1,2}(\T,M)$, and in particular uniformly, to a constant loop in
\[
\mathrm{ev} \bigl(\mathcal{W}^u(\gamma^-;-\nabla \mathbb{E}, - \nabla e) \bigr) \cap \mathbb{E}^{-1}(0).
\]
By (\ref{dista}), for every $h$ large enough we have
\begin{equation}
\label{ppp1}
\int_{\T} \int_{-S_h}^{S_h} \sqrt{ g \left( \frac{\partial q_h}{\partial s} (s,t),  \frac{\partial q_h}{\partial s} (s,t) \right) }\, ds \, dt \geq \int_{\T} \dist (q_h(-S_h,t),q_{\max})\, dt \geq d/2.
\end{equation}
By the local equivalence between the Levi-Civita metric and the ($t$-dependent) $\omega$-compatible metric $|\cdot|_t$ on $T^*M$, there is a number $C$ such that
\begin{equation}
\label{ppp2}
\int_{-S_h}^{S_h} \int_{\T}  g \left( \frac{\partial q_h}{\partial s} (s,t),  \frac{\partial q_h}{\partial s} (s,t) \right) \, dt\, ds \leq C \int_{-S_h}^{S_h} \int_{\T} \left| \frac{\partial u_h}{\partial s} (s,t) \right|_t^2 \, dt\, ds.
\end{equation}
By (\ref{Eb1}), 
\[
\mathbb{A}(v_h(-S_h)) \leq \sqrt{ \mathbb{E}(q_h(-S_h,\cdot))}, \quad \mathbb{A}(v_h(S_h)) \geq - \sqrt{\mathbb{E}(q_h(S_h,\cdot))} = - \sqrt{\mathbb{E}(q_{\max})} = 0,
\]
so the energy identity (\ref{uno}) for the solutions of the Rabinowitz-Floer equation implies the estimate
\begin{equation}
\label{ppp3}
\int_{-S_h}^{S_h} \int_{\T} \left| \frac{\partial u_h}{\partial s} (s,t) \right|_t^2 \, dt\, ds \leq \mathbb{A}(v_h(-S_h)) - \mathbb{A}(v_h(S_h)) \leq \sqrt{\mathbb{E}(q_h(-S_h,\cdot))}.
\end{equation}
By (\ref{ppp1}), (\ref{ppp2}), (\ref{ppp3}), and the Cauchy-Schwarz inequality, we find
\begin{eqnarray*}
d^2/4 \leq \left( \int_{-S_h}^{S_h} \int_{\T}  \sqrt{g \Bigl( \frac{\partial q_h}{\partial s} (s,t),  \frac{\partial q_h}{\partial s} (s,t) \Bigr)} \, dt\, ds \right)^2 \\ \leq 2 S_h 
\int_{-S_h}^{S_h} \int_{\T}  g \left( \frac{\partial q_h}{\partial s} (s,t),  \frac{\partial q_h}{\partial s} (s,t) \right) \, dt\, ds 
 \leq 2CS_h \int_{-S_h}^{S_h} \int_{\T} \left| \frac{\partial u_h}{\partial s} (s,t) \right|_t^2 \, dt\, ds \\ \leq 2 C S_h \sqrt{\mathbb{E}(q_h(-S_h,\cdot))}.
 \end{eqnarray*}
 The latter quantity is infinitesimal for $h\rightarrow +\infty$, contradicting the fact that $d$ is positive. This contradiction concludes the proof of the existence of a positive lower bound for $S$.
\end{proof} \qed

\paragraph{Proof of Proposition \ref{Pprop}.}
The remaining part of the proof uses standard arguments from  Floer theory, and we just sketch it. If $\gamma^-$ and $\gamma^+$ are critical points of $e$ with
\[
\mathrm{Ind}\, (\gamma^-;\mathbb{E},e) = \mathrm{Ind}\, (\gamma^+;\mathbb{E},-e) = 0,
\]
then the discrete set $\mathcal{M}_P(\gamma^-,\gamma^+)$ is compact, hence finite. In fact, by Lemma \ref{X}, the only sequences in $\mathcal{M}_P(\gamma^-,\gamma^+)$ which might not converge are, up to subsequences, of the form $(\underline{q}_h^-,\underline{v}_h,\underline{q}_h^+)$, where $\underline{v}_h$ is either an arc with at least one cascade, or $\underline{v}_h=(S_h,v_h)\in \mathcal{A}_0$ and $S_h\rightarrow +\infty$. In all cases, breaking must occur and in the limit we find an element which belongs to a component of negative virtual dimension of either $\mathcal{M}_{\Phi}$, $\mathcal{M}_{\Psi}$, or $\mathcal{M}_{\partial}$, violating transversality. 
If $n_P(\gamma^-,\gamma^+)\in \Z_2$ denotes the parity of the set $\mathcal{M}(\gamma^-,\gamma^+)$, we define the homomorphism
\[
P : M_0 (\mathbb{E},e) \rightarrow M^0 (\mathbb{E},-e), \quad P \gamma^- := \sum_{\substack{\gamma^+ \in \crit \, e \\ \mathrm{Ind}\, (\gamma^+;\mathbb{E},-e) = 0}} n_P(\gamma^-,\gamma^+)\, \gamma_+.
\]
As we have already checked in the proof of Lemma \ref{X}, $\mathcal{M}_P(q_{\min},q_{\max})$ is empty, so
\[
P q_{\min} \in M^0_{]0,+\infty[} (\mathbb{E},-e),
\]
as claimed.

Now assume that $\gamma^-$ and $\gamma^+$ are critical points of $e$ with
\[
\mathrm{Ind}\, (\gamma^-;\mathbb{E},e) + \mathrm{Ind}\, (\gamma^+;\mathbb{E},-e) = 1.
\]
Using again the fact that $S$ is bounded away from zero, one can show that the limiting points of each open arc in $\mathcal{M}_P(\gamma^-,\gamma^+)$ are elements of either 
\[
\mathcal{W}(\gamma^-,\gamma; -\nabla \mathbb{E},-\nabla e) \times \mathcal{M}_P(\gamma,\gamma^+),
\]
for some $\gamma\in \crit\, e$ with $\mathrm{Ind}\, (\gamma;\mathbb{E},e) = \mathrm{Ind}\, (\gamma^-;\mathbb{E},e)-1$, or  
\[
\mathcal{M}_P(\gamma^-,\gamma)\times \mathcal{W}(\gamma^+,\gamma; -\nabla \mathbb{E},\nabla e),
\]
for some $\gamma \in \crit\, e$ with $\mathrm{Ind}\, (\gamma;\mathbb{E},-e) = \mathrm{Ind}\, (\gamma^+;\mathbb{E},-e) - 1$, 
or 
\[
\mathcal{M}_{\Phi}(\gamma^-,z) \times \mathcal{M}_{\Psi}(z,\gamma^+),
\]
for some $z\in \crit\, a$ with $\mu(z) = \mathrm{Ind}\, (\gamma^-;\mathbb{E},e)$. These limiting points contribute to either $P \partial$, $\delta P$, or $\Psi\circ \Phi$. 
On the other hand, every element 
\[
((\underline{q}^-,\underline{v}^-),(\underline{v}^+,\underline{q}^+))\in \mathcal{M}_{\Phi}(\gamma^-,z) \times \mathcal{M}_{\Psi}(z,\gamma^+),
\]
with 
\[
\mathrm{Ind}\, (\gamma^-;\mathbb{E},e) = \mu(z) = 1 - \mathrm{Ind} (\gamma^+;\mathbb{E},-e),
\]
is the limiting point of some arc in $\mathcal{M}_P(\gamma^-,\gamma^+)$. More precisely, if $\underline{v}^- \in \mathcal{M}^s_h(z)$ and $\underline{v}^+ \in \mathcal{M}^u_k(z)$, then $((\underline{q}^-,\underline{v}^-),(\underline{v}^+,\underline{q}^+))$ is the limiting point of a unique arc
\[
\rho \mapsto \bigl( \underline{q}^-(\rho), \underline{v}(\rho), \underline{q}^+(\rho) \bigr) \in \mathcal{M}_P(\gamma^-,\gamma^+),
\]
with $\underline{v}(\rho) \in \mathcal{A}_{h+k-2}$. This fact can be proved by the standard gluing argument. Together with the analogous gluing results for 
\[
\mathcal{W}(\gamma^-,\gamma; -\nabla \mathbb{E},-\nabla e) \times \mathcal{M}_P(\gamma,\gamma^+)\quad \mbox{and}\quad\mathcal{M}_P(\gamma^-,\gamma)\times \mathcal{W}(\gamma^+,\gamma; -\nabla \mathbb{E},\nabla e),
\]
the above facts allow to conclude that (\ref{chainhom}) holds.
\qed  

\section{The short exact sequence and the proof of Theorem \ref{main1}}
\label{exasec}

We can use the chain homotopy $P$ constructed in the last section in order to modify $\Phi$ so that, together with $\Psi$, we get a short exact sequence. More precisely, we set
\begin{eqnarray*}
& \Theta: M_*(\mathbb{E},e)  \rightarrow RF_*(D^*M,\lambda,a) , \quad & \Theta:= \Phi - \hat{\Psi} P \partial - \partial \hat{\Psi} P, \\
& \hat{\Theta} :RF_*(D^*M,\lambda,a) \rightarrow M_*(\mathbb{E},e), \quad & \hat{\Theta} := \hat{\Phi} + \hat{\Phi} \partial \hat{\Psi} P \hat{\Phi}.
\end{eqnarray*}
By definition, $\Theta$ is a chain map which is chain homotopic to $\Phi$. 

We recall that a short exact sequence of chain complexes
\begin{equation}
\label{examod}
0 \rightarrow X \stackrel{\theta}{\rightarrow} Y \stackrel{\psi}{\rightarrow} Z \rightarrow 0
\end{equation}
is said to split if there are homomorphisms $\hat{\theta}: Y\rightarrow X$ and $\hat{\psi} : Z\rightarrow Y$ (not necessarily chain maps) such that
\[
\hat{\theta} \theta = \mathrm{Id}_X, \quad \psi \hat{\psi} = \mathrm{Id}_Z, \quad \theta \hat{\theta} + \hat{\psi} \psi = \mathrm{Id}_Y.
\]
In this case, the homomorphism $\Delta:= \hat{\theta} \partial_Y \hat{\psi}$ is a chain map from $Z$ to $X^+$, the suspension of $X$ (which is defined by $(X^+)_k := X_{k-1}$ and $\partial_{X^+} := - \partial_X$), and the induced homomorphism in homology,
\[
\Delta_* : H_* Z \rightarrow H_* X^+ = H_{*-1} X,
\]
coincides with the connecting homomorphism associated to the exact sequence (\ref{examod}) (see e.g.\ \cite{dol80}, Proposition 2.12). We are finally ready to state the first main result of this paper, which implies Theorem \ref{main1} of the Introduction:

\begin{thm}
\label{main}
Let $M$ be a closed Riemannian manifold of dimension $n\geq 2$, and let $(D^*M,\lambda)$ be its unit cotangent disk bundle, endowed with the standard Liouville 1-form. Let $\mathbb{E}$ be the geodesic energy functional on closed loops on $M$, and let $\mathbb{A}$ be the corresponding Rabinowitz action functional on $\Lambda T^*M \times \R$. Assume that $\mathbb{E}$ is Morse-Bott. If the auxiliary Morse functions $e\in C^{\infty}(\crit\, \mathbb{E})$ and $a\in C^{\infty} (\crit\, \mathbb{A})$ satisfy conditions (A0-A1-A2-A3-A4), then:
\begin{enumerate}
\item The short sequence of chain complexes
\begin{equation}
\label{exa}
0 \rightarrow M_*(\mathbb{E},e) \stackrel{\Theta}{\rightarrow} RF_*(D^*M,\lambda,a) \stackrel{\Psi}{\rightarrow} M^{1-*} (\mathbb{E},-e) \rightarrow 0
\end{equation}
is exact.
\item The homomorphisms $\hat{\Theta}$ and $\hat{\Phi}$ split the above exact sequence, that is
\[
\hat{\Theta} \Theta = \mathrm{Id}_{M_*}, \quad \Psi \hat{\Psi} = \mathrm{Id}_{M^{1-*}}, \quad \Theta \hat{\Theta} + \hat{\Psi} \Psi = \mathrm{Id}_{RF_*}.
\] 
\item The induced chain map
\[
\Delta := \hat{\Theta} \partial \hat{\Psi} : M^{1-*} (\mathbb{E},-e) \rightarrow (M_*(\mathbb{E},e))^+
\]
coincides with $\hat{\Phi} \partial \hat{\Psi}$ and satisfies
\[
\Delta \gamma = \left\{ \begin{array}{ll} \chi(T^*M) q_{\min} & \mbox{for } \gamma = q_{\max}, \\ 0 & \mbox{for } \gamma \in (\crit\, e) \setminus \{q_{\max}\}, \end{array} \right.
\]
where $\chi(T^*M)$ is the Euler number of $T^*M$.
\end{enumerate}
\end{thm}

Before proving the above theorem, let us derive its consequences (see also Theorem 1.10 in \cite{cfo09}). Since $\Theta$ is chain homotopic to $\Phi$, the long exact sequence induced by (\ref{exa}) has the form
\[
\dots \rightarrow HM_k (\mathbb{E},e) \stackrel{\Phi_*}{\rightarrow} HRF_k (D^*M,\lambda,a) \stackrel{\Psi_*}{\rightarrow} HM^{1-k} (\mathbb{E},-e) \stackrel{\Delta_*}{\rightarrow} HM_{k-1} (\mathbb{E},e) \rightarrow \dots
\]
or, taking into account the isomorphism between Morse homology (resp.\ cohomology) and singular homology (resp.\ cohomology),
\begin{equation}
\label{longexact}
\dots \rightarrow H_k (\Lambda M) \stackrel{\Phi_*}{\rightarrow} HRF_k (D^*M,\lambda,a) \stackrel{\Psi_*}{\rightarrow} H^{1-k} (\Lambda M) \stackrel{\Delta_*}{\rightarrow} H_{k-1}(\Lambda M) \rightarrow \dots
\end{equation}
The connecting homomorphism $\Delta_*$ vanishes at every index, except possibly at index zero, where
\[
\Delta_*: HM^0 (\mathbb{E},-e) \cong H^0 (\Lambda M)   \rightarrow 
H_0(\Lambda M) \cong HM_0 (\mathbb{E},e)
\]
is the composition
\[
H^0 (\Lambda M) = \prod_{c\in [\T,M]} H^0 (\Lambda^c M) \stackrel{\pi_0}{\rightarrow} H^0(\Lambda^0 M) \stackrel{\chi(T^*M)}{\longrightarrow} H_0(\Lambda^0 M) \stackrel{i_0}{\rightarrow}  \bigoplus_{c\in [\T,M]} H_0(\Lambda^c M) =
H_0(\Lambda M),
\]
where $\pi_0$ and $i_0$ denote the canonical projection and inclusion associated to the component $\Lambda^0 M$ of $\Lambda M$ consisting of contractible loops, and the middle homomorphism is the multiplication by the Euler number $\chi(T^*M)$, that is the homomorphism
\[
H^0(\Lambda^0 M) \rightarrow H_0(\Lambda^0 M), \quad \xi^* \mapsto \chi(T^*M) \xi,
\]
where $\xi$ generates $H_0(\Lambda^0 M)$ and $\xi^*$ is the dual generator of $H^0 (\Lambda^0 M) \cong H_0(\Lambda^0 M)^*$.
By the exactness of the long sequence (\ref{longexact}) and the above description of the connecting homomorphism $\Delta_*$, the Rabinowitz-Floer homology groups of the unit cotangent disk bundle are easily seen to be
\[
HRF_k (D^* M,\lambda,a) = \left\{ \begin{array}{ll} H_k (\Lambda M) & \mbox{for } k\geq 2, \\ H^{1-k} (\Lambda M) & \mbox{for } k\leq -1, \end{array} \right.
\]
when $k\neq 0,1$, wheras 
\[
HRF_k^c (D^* M,\lambda,a) = H_k (\Lambda^c M) \oplus H^{1-k}(\Lambda^c M), \quad \mbox{for } k=0,1,
\] 
if either $c\in [\T,M]$ is non-contractible or $\chi(T^*M)=0$, and
\[
HRF_0^0 (D^* M,\lambda,a) = H^1(\Lambda^0 M), \quad HRF_1^0 (D^* M,\lambda,a)= H_1(\Lambda^0 M),
\]
when $c=0\in [\T,M]$ is contractible and $\chi(T^*M)\neq 0$. These considerations conclude the proof of Theorem \ref{main1} of the Introduction.

The following proof of Theorem \ref{main} is based on Propositions \ref{partialprop}, \ref{phiprop}, \ref{psiprop}, and \ref{Pprop}, whose content we use repeatedly without further reference.

\begin{proof} 
We start by observing that $Pq_{\min} \in M^0_{]0,+\infty[}$, so $\hat{\Psi} P q_{\min} \in RF_*^{]-\infty,0[}$, hence
\begin{equation}
\label{fin1}
\partial \hat{\Psi} P q_{\min} \in RF_*^{]-\infty,0[}.
\end{equation}
We also have the chain of implications
\[
\mathrm{Im}\, \hat{\Psi} = RF^- \quad \implies \quad \mathrm{Im}\, \partial \hat{\Psi} \subset z_{q_{\min}}^+ \Z_2 + RF^- \quad \implies \quad \mathrm{Im} \, \hat{\Phi} \partial \hat{\Psi} \subset q_{\min} \Z_2.
\]
Together with (\ref{fin1}), the last inclusion implies that
\[
\mathrm{Im}\, \partial \hat{\Psi} P \hat{\Phi} \partial \hat{\Psi} \subset RF^{]-\infty,0[} \subset RF^-,
\]
and we deduce that
\begin{equation}
\label{fin2}
\hat{\Phi} \partial \hat{\Psi} P \hat{\Phi} \partial \hat{\Psi} = 0.
\end{equation}
Together with the identities $\hat{\Phi} \Phi = \mathrm{Id}_{M_*}$, $\hat{\Phi} \hat{\Psi} = 0$, (\ref{fin2}) implies that
\[
\hat{\Theta} \Theta = ( \hat{\Phi} + \hat{\Phi} \partial \hat{\Psi} P \hat{\Phi} ) (\Phi - \hat{\Psi} P \partial - \partial \hat{\Psi} P) = \mathrm{Id}_{M_*} - \hat{\Phi} \partial \hat{\Psi} P + \hat{\Phi} \partial \hat{\Psi} P - ( \hat{\Phi} \partial \hat{\Psi} P \hat{\Phi} \partial \hat{\Psi} ) P = \mathrm{Id}_{M_*}.
\] 
which shows that $\Theta$ is injective with left inverse $\hat{\Theta}$. 

We claim that
\begin{equation}
\label{fin3}
RF = \mathrm{Im}\, \Theta + RF^-.
\end{equation}
For any $\xi\in RF_*$, there exists $\eta\in M_*$ such that $\xi\in \Phi \eta + RF^-$. Thus we have the inclusions
\begin{equation}
\label{fin4}
\xi \in \Phi \eta + RF^- = (\Theta \eta + \hat{\Psi} P \partial \eta + \partial \hat{\Psi} P \eta ) + RF^- = (\Theta \eta + \partial \hat{\Psi} P \eta ) + RF^- \subset \Theta \eta + z_{q_{\min}}^+ \Z_2 + RF^-.
\end{equation}
By (\ref{fin1}),
\[
\Theta q_{\min} = \Phi q_{\min} - \partial \hat{\Psi} P q_{\min} \in z_{q_{\min}}^+  + RF^-,
\]
which is equivalent to
\[
z_{q_{\min}}^+ \in \Theta q_{\min} + RF^-.
\]
Then (\ref{fin4}) implies that 
\[
\xi \in \mathrm{Im}\, \Theta + RF^-,
\]
concluding the proof of our claim (\ref{fin3}). 
By the identity $\Psi \hat{\Psi} = \mathrm{Id}_{M^{1-*}}$,
\[
\Psi \Theta = \Psi (\Phi - \hat{\Psi} P \partial - \partial \hat{\Psi} P ) = P \partial + \delta P - \Psi \hat{\Psi} P \partial - \Psi \partial \hat{\Psi} P = P \partial + \delta P - P \partial - \delta \Psi \hat{\Psi} P = 0.
\]
So $\mathrm{Im}\, \Theta\subset \ker \Psi$, but since 
\[
RF = \ker \Psi \oplus RF^-,
\]
the identity (\ref{fin3}) implies that $\mathrm{Im}\, \Theta = \ker \Psi$. Since we already know that $\Psi$ is surjective, this concludes the proof of (a).

By (\ref{fin2}), the square of $\hat{\Phi} \partial \hat{\Psi} P$ vanishes, so $\mathrm{Id}_{M_*} + \hat{\Phi} \partial \hat{\Psi} P$ is an automorphism. Therefore, the kernel of 
\[
\hat{\Theta} = (\mathrm{Id}_{M_*} + \hat{\Phi} \partial \hat{\Psi} P ) \hat{\Phi}
\]
coincides with the kernel of $\hat{\Phi}$, that is $RF^-$. Then $\Theta \hat{\Theta}$ is the projector onto $\mathrm{Im}\, \Theta = \ker \Psi$ along $RF^-$, whereas $\hat{\Psi} \Psi$ is the projector onto $RF^-$ along $\Ker \Psi = \mathrm{Im}\, \Theta$. It follows that
\[
\Theta \hat{\Theta} + \hat{\Psi} \Psi = \mathrm{Id}_{RF_*},
\]
concluding the proof of (b).

By (\ref{fin2}), we have
\[
\Delta := \hat{\Theta} \partial \hat{\Psi} = \hat{\Phi} \partial \hat{\Psi} + \hat{\Phi} \partial \hat{\Psi} P \hat{\Phi} \partial \hat{\Psi} = \hat{\Phi} \partial \hat{\Psi}.
\]
Moreover, we have the chain of implications
\begin{eqnarray*}
&\hat{\Psi} q_{\max} \in z_{q_{\max}}^- + RF^{]-\infty,0[} \quad \implies \quad \partial 
\hat{\Psi} q_{\max} \in \chi (M) z_{q_{\min}}^+ + RF^{]-\infty,0[} \\ & \implies \quad \Delta q_{\max} = \hat{\Phi}  \partial 
\hat{\Psi} q_{\max} = \chi(T^*M) q_{\min}.
\end{eqnarray*}
Finally, if $\gamma$ is in $(\crit \, e) \setminus \{q_{\max}\}$, then the coefficient of $z_{q_{\max}}^-$ in $\hat{\Psi} \gamma$ vanishes, so $\partial \hat{\Psi} \gamma$ belongs to $RF^-$, and hence $\hat{\Phi} \partial \hat{\Psi} \gamma = 0$. This concludes the proof of (c).
\end{proof} \qed

\begin{rem}
Instead of modifying the chain map $\Phi$, we can modify $\Psi$ within its chain homotopy class, by defining $\Gamma: = \Psi - \delta P \hat{\Phi} - P \hat{\Phi} \partial$. The short sequence
\[
0 \rightarrow M_* (\mathbb{E},e) \stackrel{\Phi}{\rightarrow} RF_*(D^*M,\lambda,a) \stackrel{\Gamma}{\rightarrow} M^{1-*}(\mathbb{E},-e) \rightarrow 0
\]
is exact and the other statements of Theorem \ref{main} hold, with obvious modifications.
\end{rem}

\paragraph{Shifting the exact sequence.} We conclude this section by comparing how the long exact sequence (\ref{longexact}) arises in our approach and in \cite{cfo09}. The following observations were suggested to us by A.\ Oancea.

Theorem \ref{main} shows that the Rabinowitz-Floer complex $RF_*(D^*M,\lambda,a)$ can be identified with the direct sum of the chain complexes $M_*(\mathbb{E},e)$ and $M^{1-*}(\mathbb{E},-e)$. Moreover, the homology of the first of the latter two complexes is the symplectic homology $SH_*(D^*M)$ of the cotangent disk bundle, while the cohomology of the second one (after inverting and shifting the grading) is the symplectic cohomology $SH^*(D^*M)$ of the same Liouville domain.  

In \cite{cfo09}, the Rabinowitz-Floer complex $RF_*(W,\lambda,a)$ of the Liouville domain $(W,\lambda)$ is seen instead as a quotient, in the following way. Fix two real numbers $\alpha<0<\beta$, and choose a Hamiltonian $H\in C^{\infty}(\hat{W})$ which is $V$-shaped, meaning that $H$ is equal to a large constant in $W \setminus (\partial W \times ]0,1])$, decreases very steeply (with respect to $-\alpha$ and $\beta$) in the negative collar $\partial W \times ]0,1[$, and then increases very steeply in the positive collar $\partial W \times [1,+\infty[$. If $F_*^I(H)$ denotes the Floer complex associated to $H$ and to the action interval $I\subset \R$, standard truncation produces the short exact sequence of chain complexes
\begin{equation}
\label{short2}
0 \rightarrow F_*^{]-\infty,\alpha[}(H) \rightarrow F_*^{]-\infty,\beta[}(H) \rightarrow F_*^{]\alpha,\beta[} (H) \rightarrow 0.
\end{equation}
The homology of the first complex in (\ref{short2}) is shown to be the truncated symplectic cohomology 
\[
SH^{-*}_{]-\infty,-\alpha[}(W,\lambda),
\]
while the homology of the second one is the truncated symplectic homology 
\[
SH_*^{]-\infty,\beta[}(W,\lambda),
\] 
see Proposition 2.9 in \cite{cfo09}. On the other hand, by the main result of \cite{cfo09}, the homology of the last complex can be identified with the truncated Rabinowitz-Floer homology $H_*RF^{]\alpha,\beta[}(W,\lambda)$. 
So the long exact sequence induced by (\ref{short2}) is
\[
\dots \rightarrow SH^{-k}_{]-\infty,-\alpha[} (W,\lambda) \rightarrow SH_k^{]-\infty,\beta[} (W,\lambda) \rightarrow HRF_k^{]\alpha,\beta[} (W,\lambda) \rightarrow 
SH^{-k+1}_{]-\infty,-\alpha[} (W,\lambda) \rightarrow 
\dots,
\]
which after considering an inverse limit for $\alpha\downarrow -\infty$ followed by a direct limit for $\beta\uparrow +\infty$, produces the long exact sequence
\begin{equation}
\label{long2}
\dots \rightarrow SH^{-k} (W,\lambda) \rightarrow SH_k (W,\lambda) \rightarrow HRF_k (W,\lambda) \rightarrow SH^{-k+1} (W,\lambda) \rightarrow \dots .
\end{equation}
The homomorphism
\begin{equation}
\label{PD}
SH^{-k} (W,\lambda) \rightarrow SH_k (W,\lambda)
\end{equation}
in the above sequence can be thought as a sort of Poincar\'e duality between symplectic cohomology and homology, and Rabinowitz-Floer homology is consequently seen as obstructing such a map from being an isomorphism, via (\ref{long2}). In the particular case $W=D^*M$, the homomorphism (\ref{PD}) is actually very far from being an isomorphism: It vanishes for every $k\neq 0$, and for $k=0$ it may be non zero only on the component of contractible loops, where it coincides with the multiplication by the Euler number $\chi(T^*M)$ (see Lemma 1.11 in \cite{cfo09}).

Summing up, the Rabinowitz-Floer complex appears as the middle term in the short exact sequence (\ref{exa}), but it appears as the last term in the short exact sequence (\ref{short2}). This difference can be explained by the following general fact in homological algebra (see also Sections 1.5 and 10.1 in \cite{wei94}, or the proof of Proposition III.3.5 in \cite{gm03}).

Let 
\begin{equation}
\label{short3} 
0 \rightarrow X \stackrel{\theta}{\rightarrow} Y \stackrel{\psi}{\rightarrow} Z \rightarrow 0
\end{equation}
be a short exact sequence of chain complexes which splits: There exist homomorphisms $\hat{\theta} : Y \rightarrow X$ and $\hat{\psi}: Z \rightarrow Y$ (in general not chain maps) such that
\[
\hat{\theta} \theta = \mathrm{Id}_X, \quad \psi \hat{\psi} = \mathrm{Id}_Z, \quad \theta \hat{\theta} + \hat{\psi} \psi = \mathrm{Id}_Y.
\]
As already recalled, the connecting homomorphism $\Delta_*$ in the long exact sequence induced by (\ref{short3}), that is
\begin{equation}
\label{long3}
\dots\rightarrow  H_k X \stackrel{\theta_*}{\longrightarrow} H_k Y \stackrel{\psi_*}{\longrightarrow} H_k Z \stackrel{\Delta_*}{\longrightarrow} H_{k-1} X \rightarrow \dots
\end{equation}
is induced by the chain map
\[
\Delta: Z \rightarrow X^+, \quad \Delta := \hat{\theta} \partial_Y \hat{\psi},
\]
where $X^+$ is the suspension of $X$ (defined by $(X^+)_k:= X_{k-1}$ and $\partial_{X^+} = - \partial_X$). 
Let $C(\psi)$ be the cone of the chain map $\psi$, that is the chain complex defined by
\[
C(\psi) := Z \oplus Y^+, \quad \partial_{C(\psi)} (z,y) := (\partial_Z z + \psi y, - \partial_Y y).
\]
Then
\begin{equation}
\label{short4}
0 \rightarrow Z \stackrel{\iota}{\rightarrow} C(\psi) \stackrel{\pi}{\rightarrow} Y^+ \rightarrow 0,
\end{equation}
where $\iota z := (z,0)$ and $\pi (z,y) := y$, is a short exact sequence of chain complexes. 

We claim that $C(\psi)$ is chain equivalent to $X^+$ and thereby the long exact sequence induced by (\ref{short4}) becomes isomorphic to (\ref{long3}), up to a shift.
Indeed, the chain maps
\begin{eqnarray*}
\sigma : X^+ \rightarrow C(\psi) , \quad & \sigma x := (0,\theta x), \\
\rho: C(\psi) \rightarrow X^+, \quad & \rho (z,y) := \Delta z + \hat{\theta} y,
\end{eqnarray*}
are homotopy inverses one of the other, because of the identities
\[
\rho \sigma = \mathrm{Id}_{X^+}, \quad \mathrm{Id}_{C(\psi)} - \sigma \rho = \tau \, \partial_{C(\psi)} + \partial_{C(\psi)} \tau, \quad \mbox{with } \tau(z,y):= (0,\hat{\psi}z).
\]
Moreover, the short exact sequence (\ref{short4}) is splitted by the homomorphisms $\hat{\iota} (z,y):= z$ and $\hat{\pi} y := (0,y)$, so the connecting homorphism induced by (\ref{short4}) is the homomorphism in homology induced by the chain map $\hat{\iota} \partial_{C(\psi)} \hat{\pi} = \psi$. Together with the identities $\rho \iota = \Delta$ and $\pi \sigma = \theta$, we conclude that the long exact sequence induced by (\ref{short4}) coincides with the long exact sequence (\ref{long3}), up to a shift, as the following commuting diagram shows:
\begin{equation*}
\xymatrix{
\dots \ar[r] & H_k Z \ar[r]^{\iota_*} \ar@{=}[d] & H_k C(\psi) \ar[r]^{\pi_*} 
\ar[d]_{\rho_*} & H_k Y^+ \ar[r]^{\psi_*} \ar@{=}[d] & H_{k-1} Z \ar[r] \ar@{=}[d] & \dots \\
\dots \ar[r] & H_k Z \ar[r]^{\hspace{-20pt}\Delta_*} & H_k X^+ = H_{k-1} X \ar[r]^{\hspace{20pt}\theta_*}  \ar@<-1ex>[u]_{\sigma_*} & H_{k-1} Y \ar[r]^{\psi_*} & H_{k-1} Z \ar[r] & \dots
}
\end{equation*}
 
Let us apply the above considerations to the short exact sequence (\ref{exa}), endowed with the splitting given by statement (b) in Theorem \ref{main}. We obtain that the cone of $\Psi$ is chain equivalent to suspension of the Morse chain complex of 
$\mathbb{E}$,
\[
C(\psi) \simeq \bigl(M_* (\mathbb{E},e)\bigr)^+,
\]
and that the long exact sequence induced by (\ref{exa}), that is (\ref{longexact}), coincides up to a shift with the homology sequence induced by the short exact sequence of chain complexes
\begin{equation}
\label{shortu}
\xymatrix{
0 \ar[r] & M^{-*}(\mathbb{E},-e) \ar[r] & C(\Psi)_{*+1} \ar[r] & RF_*(W,\lambda,a)  \ar[r] & 0. \\ & & M_*(\mathbb{E},e)  \ar@{-}[u]|{\displaystyle{\simeq}} & & 
}
\end{equation}
The above short exact sequence is the analogue of (\ref{short2}), in the particular case $W=D^*M$ and after taking limits for $\alpha \downarrow -\infty$ and $\beta\uparrow +\infty$, so the induced homology sequence is (\ref{long2}). The first chain map in (\ref{shortu}) induces the homomorphism
\[
SH^{-*}(D^*M,\lambda) \cong H^{-*}M(\mathbb{E},-e) \rightarrow H_{*+1} C(\Psi) \cong H_*M(\mathbb{E},e) \cong SH_*(D^*M,\lambda),
\]
which by statement (c) in Theorem \ref{main} coincides with the homomorphism (\ref{PD}) described above.

\section{Uniformly convex domains in $T^*M$ containing a Lagrangian graph}
\label{scdsec}

Let us consider the situation of Example 2 in Section \ref{setting}: $M$ is a closed manifold and $W$ is a compact subset of $T^*M$, which has smooth fiberwise uniformly convex boundary and whose interior part contains a Lagrangian graph, that is a set of the form
\[
\set{(q,\theta(q))}{q\in M},
\]
where $\theta$ is a smooth closed one-form on $M$. We recall that in this case $W$ is a Liouville domain with respect to the one-form 
\[
\lambda_{\theta} := \lambda - \pi^* \theta.
\]
We assume that the closed orbits of the Reeb vector field $R$ on $\partial W$ induced by the contact form $\alpha := \lambda_{\theta}|_{\partial W}$ are of Morse-Bott type (see Section \ref{rfcsec}).

\paragraph{The Hamiltonian $H$.}
For every $q\in M$, let 
\[
f_q : T_q^* M \rightarrow \R
\]
be the function which takes the value 1 on $\partial W \cap T_q^*M$ and is such that the function
\[
p \mapsto f_q(\theta(q)+p), \quad p\in T_q^*M,
\]
is homogeneous of degree 2. The function $f_q$ is continuously differentiable on $T_q^*M$ and it is smooth on $T_q^*M \setminus \{\theta(q)\}$ (it is everywhere smooth if and only if $\partial W \cap T_q^*M$ is an ellipsoid). By the 2-homogeneity,
\begin{equation}
\label{homo1}
df_q(p) [p-\theta(q)] = 2 f_q(p), \quad \forall p\in T^*_q M.
\end{equation}
By the properties of $W$, we can find a smooth function $H$ on $T^*M$ which has postive definite fiberwise second differential and such that
\begin{equation}
\label{homo2}
H(q,p) = \frac{1}{2} \bigl( f_q(p) - 1 \bigr)
\end{equation} 
for every $q\in M$ and every $p$ outside a neihborhood of $\theta(q)$, contained in the interior part of $W\cap T_q^*M$. In particular, the zero level set of $H$ is $\partial W$ and (\ref{homo1}), (\ref{homo2}) imply that for every $(q,p)\in \partial W$ there holds
\[
\lambda_{\theta} \bigl( X_H(q,p) \bigr) = (\lambda - \pi^* \theta) \bigl( X_H(q,p) \bigr) = d_p H(q,p) [p-\theta(q)] = \frac{1}{2} df_q(p) [p-\theta(q)] = f_q(p) = 1,
\]
so $X_H|_{\partial W}$ coincides with the Reeb vector field $R$. 

By the discussion of Example 2 in Section \ref{setting}, the completion of $(W,\lambda_{\theta})$ is identified with $T^*M$ by means of the embedding
\[
\partial W \times ]0,+\infty[ \rightarrow T^*M, \quad \bigl((q,p),\rho \bigr) \mapsto \bigr( q,\theta(q) + \rho(p-\theta(q)) \bigr).
\]
By the fact that (\ref{homo2}) holds on $T^*M \setminus \mathrm{Int}\, (W)$, and by the 2-homogeneity property of $f_q$, for every $(q,p)\in \partial W$ and every $\rho\geq 1$ we have
\[
H \bigl( (q,p) , \rho \bigr) = H \bigl( q,\theta(q) + \rho (p-\theta(q)) \bigr) = \frac{1}{2} \Bigl( f_q \bigl( \theta(q) + \rho (p-\theta(q)) \bigr) - 1 \Bigr) = \frac{1}{2} \bigl( \rho^2 f_q(p) - 1 \bigr) = \frac{1}{2} (\rho^2 - 1).
\]  
We conclude that
the Hamiltonian $H$ is compatible with $(\partial W,\alpha)$ and eventually quadratic on the completion $(\hat{W},\lambda_{\theta}) = (T^*M,\lambda_{\theta})$ of $W$ (see Section \ref{rfcsec}), so we can use it to define the Rabinowitz-Floer homology of $(W,\lambda_{\theta})$. By construction, the second derivatives of $H$ satisfy the bounds:
\begin{eqnarray}
\label{h0}
& d^2_p H(q,p) \geq h_0 \, \mathrm{Id}, \\
\label{h1}
& \left| \nabla_{pp} H(q,p) \right| \leq h_1, \quad 
\left| \nabla_{pq} H(q,p) \right| \leq h_1 ( 1 + |p|), \quad
\left| \nabla_{qq} H(q,p) \right| \leq h_1 ( 1 + |p|^2),
\end{eqnarray}
for every $(q,p)\in T^*M$, for some $h_0,h_1>0$. Here $\nabla_{pp}$, $\nabla_{qp}$ and $\nabla_{qq}$ are the components of the Hessian of $H$ associated to the horizontal-vertical splitting of $TT^*M$ induced by some Riemannian metric on $M$. The symbol $|\cdot|$ denotes the norm on $TM$ and on $T^*M$ induced by such a metric. 

\paragraph{The Lagrangian $L$.}
Let $L:TM \rightarrow \R$ be the Fenchel transform of $H$, defined by
\[
L(q,v) := \max_{p\in T_q^* M} \Bigl( \langle p, v \rangle - H(q,p) \Bigr).
\]
By (\ref{h0}) and (\ref{h1}), $L$ is smooth, and satisfies similar bounds:
\begin{eqnarray}
\label{l0}
d^2_v L(q,v) \geq \ell_0 \, \mathrm{Id}, \\
\label{l1}
\left| \nabla_{vv} L(q,v) \right| \leq \ell_1, \quad 
\left| \nabla_{vq} L(q,v) \right| \leq \ell_1 ( 1 + |v|), \quad
\left| \nabla_{qq} L(q,v) \right| \leq \ell_1 ( 1 + |v|^2),
\end{eqnarray}
for every $(q,v)\in TM$, for some $\ell_0,\ell_1>0$. The Lagrangian action of an absolutely continuous curve $\gamma: [0,T] \rightarrow M$ is denoted by
\[
\mathbb{S}_L (\gamma) := \int_0^T L(\gamma(t),\gamma'(t))\, dt.
\]
We recall that the {\em Ma\~n\'e critical value} of the Lagrangian $L$ is the number
\[
c(L) := \inf \set{\kappa \in\R}{\mathbb{S}_{L+\kappa}(\gamma)\geq 0 \; \forall \gamma : \R/T\Z \rightarrow M \mbox{ absolutely continuous, } \forall T>0}.
\]
G.\ Contreras, R.\ Iturriaga, G.\ P.\ Paternain and M.\ Paternain \cite{cipp98} have shown that the Ma\~n\'e critical value of the Lagrangian $L$ can be expressed in terms of the dual Hamiltonian $H$ as
\begin{equation}
\label{cara}
c(L) = \inf_{u\in C^{\infty}(M)} \max_{q\in M} H(q,du(q)).
\end{equation}
Since $\theta$ is a closed one-form, the Lagrangian $L-\theta$ defines the same Euler-Lagrange equation as the Lagrangian $L$, in particular it defines the same extremal curves. Moreover, the Fenchel transform of $L-\theta$ is the Hamiltonian $(q,p) \rightarrow H(q,p+\theta(q))$. By using (\ref{cara}), we deduce that the Ma\~n\'e critical value of the Lagrangian $L-\theta$ is negative:
\[
c(L-\theta) = \inf_{u\in C^{\infty}(M)} \max_{q\in M} H(q,du(q)+\theta(q)) \leq \max_{q\in M} H(q,\theta(q)) < 0,
\]
where we have used also the fact that the image of the one-form $\theta$ is contained in $\mathrm{Int}\, (W) =\{H<0\}$. 

\begin{rem}
\label{smcv}
If one starts from the Lagrangian formulation, the setting described in this section takes the following form. Let $L$ be a Tonelli Lagrangian on $TM$, that is a smooth  function $L$ on $TM$ such that
\[
\lim_{|v|\rightarrow +\infty} \frac{L(q,v)}{|v|} = +\infty,
\]
and the fiberwise second differential $d_v^2 L(q,v)$ is everywhere positive.  We pick a real number $\kappa$ which is larger than the {\em strict Ma\~n\'e critical value of }$L$, that is the Ma\~n\'e critical value of the lift of $L$ to the Abelian cover of $M$, or equivalently the number
\[
c_0(L) := \min_{\theta} c(L-\theta),
\]
where the minimum is taken over the set of all closed one-forms on $M$,
see \cite{pp97}.   
Let $H\in C^{\infty}(T^*M)$ be the Hamiltonian which is Fenchel dual to $L$, and let $W$ be the set $\{H\leq \kappa\}$. If $\theta$ is a closed one-form on $M$ such that $c(L-\theta) = c_0(L)$, then the characterization (\ref{cara}) applied to $L-\theta$ implies that there exists a smooth function $u$ on $M$ such that
the interior part of $W$ contains the Lagrangian graph given by the image of the  closed one-form $du+ \theta$. Therefore, $W$ fulfills the requirements of this section (and of Theorem \ref{main2} of the Introduction).
\end{rem}

\begin{rem}
On the other hand, if $\kappa\leq c_0(L)$ then the energy level $H^{-1}(\kappa)$ may not be of contact type, so in particular $\{H\leq \kappa\}$ may not  be a Liouville domain with respect to any primitive of the symplectic form $\omega$ and  the Rabinowitz-Floer homology, as presented in this paper, may not be well-defined. For instance, by Theorem B.1 in \cite{con06}, $H^{-1}(\kappa)$ is never of contact type when $c_u(L)< \kappa \leq c_0(L)$ and $M$ is not the 2-torus. Here $c_u(L)$ is the Ma\~n\'e critical value of the lift of $L$ to the universal cover of $M$, a number which is in general strictly lower than $c_0(L)$, see \cite{pp97}. However, if $\kappa$ is in this range, the energy surface is of {\em virtual contact type} and Rabinowitz-Floer homology can be easily extended to this situation, as shown by K.\ Cieliebak, U.\ Frauenfelder, and G.\ P.\ Paternain in \cite{cfp09}. A work in progress by W.\ Merry \cite{mer09} shows how to determine the Rabinowitz-Floer homology for this range of energies, in the more general case where the standard symplectic form of the cotangent bundle is twisted by a weakly exact magnetic two-form. Below $c_u(L)$, the energy surface may not be of virtual contact type and extending Rabinowitz-Floer homology to such levels is substantially more difficult. See \cite{cfp09} for results in this direction and for many instructive examples.   
\end{rem}

\section{The Morse chain complex and differential complex of the free period Lagrangian action functional}

The aim of this section is to show how the fact that the Ma\~n\'e critical value of $L-\theta$ is negative allows to associate a Morse chain complex and a Morse differential complex to the free period action functional $\mathbb{S}_{L-\theta}$. The necessary analysis is due to G.\ Contreras \cite{con06} (see also \cite{cipp00}).

\paragraph{Properties of the action functional $\mathbb{S}$.} It is useful to reparametrize every closed curve $\gamma:\R /T \Z \rightarrow M$, for $T>0$ arbitrary, on the circle of unit length $\T=\R/\Z$, by setting $q(\tau) := \gamma(T \tau)$ for $\tau\in \T$. With this change of variable, the action of $\gamma$ associated to the Lagrangian $L-\theta$ takes the form
\begin{eqnarray*}
\mathbb{S}(q,T)=
\mathbb{S}_{L-\theta} (q,T) := \mathbb{S}_{L-\theta}(\gamma) = \int_0^T (L-\theta) (\gamma(t),\gamma'(t))\, dt = T \int_{\T} (L-\theta)(q(\tau),q'(\tau)/T) \, d\tau \\ =
T \int_{\T} L(q(\tau),q'(\tau)/T) \, d\tau - \int_{\T} q^*(\theta).
\end{eqnarray*}
A suitable functional space for the Morse theory of $\mathbb{S}$ is the Hilbert manifold
\[
\mathscr{M} := W^{1,2} (\T,M) \times ]0,+\infty[.
\]
We denote by $\mathscr{M}^0$ the connected component of $\mathscr{M}$ corresponding to contractible loops. 
By (\ref{l0}), there are positive numbers $\ell_2,\ell_3>0$ such that
\begin{equation}
\label{l2}
(L-\theta)(q,v) \geq \ell_2 |v|^2 - \ell_3, \quad \forall (q,v)\in TM.
\end{equation}
Together with the fact that $c(L-\theta)$ is negative, the above estimate has the following consequences:

\begin{prop}
\label{sulfunz}
\begin{enumerate}
\item $\mathbb{S}$ is strictly positive on $\mathscr{M}$, and
\[
\inf_{\mathscr{M}^0} \mathbb{S} = 0, \quad \inf_{\mathscr{M}\setminus \mathscr{M}^0} \mathbb{S}>0.
\]
\item For every $S\in \R$,
\[
\inf \set{T>0}{(q,T)\in \mathscr{M} \setminus \mathscr{M}^0, \; \mathbb{S}(q,T) \leq S} >0.
\]
\item For every $(q,T)$ in $\mathscr{M}$ there holds
\[
\int_{\T} |q'(\tau)|^2\, d\tau \leq \frac{T}{\ell_2} \mathbb{S}(q,T) + \frac{\ell_3}{\ell_2} T^2.
\]
\item For every $(q,T)$ in $\mathscr{M}$ there holds
\[
\mathbb{S}(q,T) \geq - c(L-\theta) T.
\]
\end{enumerate}
\end{prop}

\begin{proof}
By (\ref{l2}), there holds
\[
\mathbb{S}(q,T) = T \int_{\T} (L-\theta) (q(\tau),q'(\tau)/T) \, d\tau \geq \frac{\ell_2}{T} \int_{\T} |q'(\tau)|^2 \, d\tau - \ell_3 T,
\]
which implies (c). Since $M$ is compact, the length of every non-contractible absolutely continuous curve is bounded from below by some positive number $\ell$. Therefore, if $(q,T)\in \mathscr{M} \setminus \mathscr{M}^0$ then
\[
\ell^2 \leq \left( \int_{\T} |q'(\tau)|\, d\tau \right)^2 \leq \int_{\T} |q'(\tau)|^2\, d\tau,
\]
from which, using (c) and the bound $\mathbb{S}(q,T)\leq S$, we get
\[
T^2 \geq \frac{\ell_2}{\ell_3} \ell^2 - \frac{T}{\ell_3} \mathbb{S}(q,T) \geq   \frac{\ell_2}{\ell_3} \ell^2 - \frac{T}{\ell_3} S,
\]
which implies (b). If $\gamma(t) = q(t/T)$, then
\[
\mathbb{S}(q,T) = \int_0^T (L-\theta) (\gamma(t),\gamma'(t)) \, dt = \int_0^T \Bigl( (L-\theta ) (\gamma(t),\gamma'(t)) + c(L-\theta)\Bigr) \, dt - c(L-\theta) T \geq - c(L-\theta) T,
\]
because of the definition of the Ma\~n\'e critical value $c(L-\theta)$. This proves (d).
Since $c(L-\theta)<0$, the inequality (d) implies that $\mathbb{S}>0$. If $q:\T \rightarrow M$ is a constant loop, then 
\[
\mathbb{S}(q,T) = T \,L(q,0)
\]
converges to zero for $T\downarrow 0$, so the infimum of $\mathbb{S}$ on $\mathscr{M}^0$ is zero. By (b) and (d), the infimum of $\mathbb{S}$ on $\mathscr{M} \setminus \mathscr{M}^0$ is positive. This concludes the proof of (a).
\end{proof}
\qed

The functional $\mathbb{S}$ is of class $C^{1,1}$ on $\mathscr{M}$, but in general it is not twice differentiable. More precisely, $\mathbb{S}$ is everywhere twice Gateaux-differentiable, but it is
twice Fr\'ech\'et-differentiable at some $(q,T)\in \mathscr{M}$ if and only if for every $\tau\in \T$ the function $v\mapsto L(q(\tau),v)$ is a polynomial of degree at most 2 on $T_{q(\tau)} M$, see \cite{as08b}. Let $E:TM \rightarrow \R$ be the energy associated to the Lagrangian $L-\theta$, or equivalently to the Lagrangian $L$, that is the function
\[
E(q,v) = d_v (L-\theta)(q,v)[v] - (L-\theta)(q,v) = d_v L (q,v)[v] - L(q,v).
\]
Equivalently, $E$ is the composition of $H$ with the Legendre transform
\[
TM \rightarrow T^*M, \quad (q,v) \mapsto d_v L(q,v).
\]
Therefore, $E$ is constant along the Euler-Lagrange flow induced by $L$, and the quadratic growth assumptions on $H$ (\ref{h0}-\ref{h1}) imply that also $E$ has quadratic growth in $v$, so by (\ref{l1}) we can find positive numbers $e_0,e_1$ such that
\begin{equation}
\label{LvE}
E(q,v) \geq e_0 L(q,v) - e_1, \quad \forall (q,v)\in TM.
\end{equation}
The partial derivative of $\mathbb{S}$ with respect to the second variable is easily seen to be
\begin{equation}
\label{derT}
\frac{\partial \mathbb{S}}{\partial T} (q,T) = - \frac{1}{T} \int_0^T E(\gamma(t),\gamma'(t))\, dt,
\end{equation}
where as before $\gamma(t)=q(t/T)$. Therefore, $(q,T)$ is a critical point of $\mathbb{S}$ on $\mathscr{M}$ if and only if $\gamma$ is a periodic solution of the Euler-Lagrange equation associated to $L$ with (not necessarily minimal) period $T$ and energy $E(\gamma,\gamma')\equiv 0$. Equivalently, $t\mapsto d_v L(\gamma(t),\gamma'(t))$ is a $T$-periodic orbit of $X_H$ on the energy surface $\partial W = H^{-1}(0)$. Since such an energy surface is compact and does not contain critical points of $H$, the period of its closed Hamiltonian orbits is bounded away from zero, so
\[
\inf \set{T>0}{(q,T) \in \crit \, \mathbb{S}} >0.
\]
Together with statement (d) of Proposition \ref{sulfunz} and the fact that $c(L-\theta)$ is negative, we deduce that the critical values of $\mathbb{S}$ are bounded away from zero:
\begin{equation}
\label{bfafz}
\inf \set{\mathbb{S}(q,T)}{(q,T) \in \crit\, \mathbb{S}}>0.
\end{equation}  

\paragraph{A smooth pseudo-gradient for $\mathbb{S}$.}
We recall that $W^{1,2}(\T,M)$ is endowed with the complete metric (\ref{natmet}) induced by the Riemannian structure of $M$, and we can endow $\mathscr{M}$ with the product metric by the standard metric of $]0,+\infty[$. Of course, such a metric - that we still denote by $\langle\langle \cdot,\cdot \rangle\rangle$ - is not complete. The corresponding norms on $T\mathscr{M}$ and $T^* \mathscr{M}$ are denoted by $\|\cdot\|$. 

By (\ref{derT}) and (\ref{LvE}),
\begin{equation}
\label{derTvS}
\langle \langle \nabla \mathbb{S}(q,T), \frac{\partial}{\partial T} \rangle \rangle =
\frac{\partial \mathbb{S}}{\partial T} (q,T) \leq - \frac{e_0}{T} \,\mathbb{S}(q,T) + e_1.
\end{equation}
In particular,
\begin{equation}
\label{derTvS2}
\langle \langle \nabla \mathbb{S}(q,T), \frac{\partial}{\partial T} \rangle \rangle \leq  - e_1 \quad \mbox{if} \quad\mathbb{S}(q,T) \geq \frac{2e_1}{e_0} T.
\end{equation}
Moreover, if $q_0\in M$ is a constant loop, then
\begin{equation}
\label{inva}
\frac{\partial \mathbb{S}}{\partial T} (q_0,T) = L(q_0,T) = - \min_{p\in T_{q_0} M} H(q_0,p) > 0.
\end{equation}
Although the functional $\mathbb{S}$ is just $C^{1,1}$ on $\mathscr{M}$, the Morse-Bott assumption allows to construct a smooth {\em pseudo-gradient} for $\mathbb{S}$, by which we mean a smooth vector field $Z$ on $\mathscr{M}$ such that:
\begin{enumerate}
\item[(Z1)] $Z$ is bounded;
\item[(Z2)] $d\mathbb{S}(w)[Z(w)] \geq \delta\bigl(\mathbb{S}(w)\bigr) \|d\mathbb{S}(w)\|$, for every $w\in \mathscr{M}$, where $\delta$ is a continuous positive function on $]0,+\infty[$;
\item[(Z3)] $\sing Z = \crit \, \mathbb{S}$ and for every $w$ in such a set $\nabla Z (w) = \nabla^2 \mathbb{S}(w)$, the Gateaux Hessian of $\mathbb{S}$ at $w$.
\end{enumerate}
The symbol $\sing Z$ denotes the set of zeroes of the vector field $Z$, and the Jacobian $\nabla Z(w) : T_w \mathscr{M} \rightarrow T_w \mathscr{M}$ is well-defined at every $w\in \sing \, Z$. The construction of such a smooth pseudo-gradient vector field is described in \cite{as08b}, in the case of the fixed period action functional - under Morse assumptions - but the case of the free period functional - under Morse-Bott assumptions - is analogous. The construction of $Z$ far from the critical set of $\mathbb{S}$ uses just a partition of unity argument, so by (\ref{derTvS2}) and (\ref{inva}) we may also achieve the properties:
\begin{enumerate}
\setcounter{enumi}{3}
\item[(Z4)] There exists $C>0$ such that $\langle\langle Z(q,T), \partial/\partial T \rangle\rangle < 0$ if $\mathbb{S}(q,T) \geq C T$;
\item[(Z5)] the submanifold $M\times ]0,+\infty[$ is invariant with respect to the local flow of $Z$.
\end{enumerate} 
In order to achieve (Z5), one might need to change the positive function $\delta$ of (Z2), taking into account the fact that if $q_0$ is a constant loop then
\[
\mathbb{S}(q_0,T) = T L(q_0,0) \rightarrow +\infty \quad \mbox{for } T\rightarrow +\infty,
\]
uniformly in $q_0\in M$.
The construction of $Z$ near the critical set uses the fact that this set is a smooth manifold, by the Morse-Bott assumption, and some properties of the map $w\mapsto \nabla^2 \mathbb{S}(w)$ with respect to the strong topology of operators. 

We recall that a {\em Palais-Smale sequence at level $S$} for the pair $(\mathbb{S},Z)$ is a sequence $(w_h)\subset \mathscr{M}$ such that $\mathbb{S}(w_h)\rightarrow S$ and $d\mathbb{S}(w_h)[Z(w_h)]\rightarrow 0$, and that the {\em Palais-Smale condition at level $S$} holds whenever every such sequence has a subsequence which converges in $\mathscr{M}$ (see \cite{ama06m}; when $Z$ is the gradient of $\mathbb{S}$ one recovers the standard definitions). 

Let $\phi^{-Z}$ be the local flow of $-Z$ on $\mathscr{M}$. Property (Z2) above guarantees that $S$ is strictly decreasing on the non-constant orbits of $\phi^{-Z}$. If $w\in \mathscr{M}$, let $]\sigma_-(w),\sigma_+(w)[$ be the maximal interval of existence of the flow line $\sigma\mapsto \phi^{-Z}(\sigma,w)$ through $w$.
A subset $\mathscr{A}\subset \mathscr{M}$ is {\em positively} (respectively {\em negatively}) {\em invariant} with respect to $\phi^{-Z}$ if $w\in \mathscr{A}$ implies $\phi^{-Z}(\sigma,w)\in \mathscr{A}$ for every $\sigma \in [0,\sigma_+(w)[$ (resp.\ for every $\sigma \in ]\sigma_-(w),0]$). The local flow $\phi^{-Z}$ is said to be {\em positively} (resp.\ {\em negatively}) {\em complete} with respect to the positively (resp.\ negatively) invariant set $\mathscr{A}$ if every $w\in \mathscr{M}$ such that $\sigma_+(w)<+\infty$ (resp.\ $\sigma_-(w)>-\infty$) eventually enters $\mathscr{A}$, that is $\phi^{-Z}(\sigma,w)\in \mathscr{A}$ for some $\sigma>0$ (resp.\ $\sigma <0$).   

\begin{prop}
\label{suZ}
\begin{enumerate}  
\item The pair $(\mathbb{S},Z)$ satisfies the Palais-Smale condition at every level $S>0$.
\item For every $S>0$, the local flow $\phi^{-Z}$ is positively complete with respect to the sublevel $\{\mathbb{S}<S\}$.
\item Let $S>0$. Then the set
\[
\mathscr{A}_S :=
\set{(q,T)\in \mathscr{M}}{ C T < S < \mathbb{S}(q,T)}
\]
contains no critical points of $\mathbb{S}$ and is negatively invariant with respect to the local flow $\phi^{-Z}$. Moreover, the restriction of the local flow $\phi^{-Z}$ to the (negatively invariant) set $\{\mathbb{S}>S\}$ is negatively complete with respect to $\mathscr{A}_S$.
\item The functional $\mathbb{S}$ has a positive minimum on every connected component of $\mathscr{M}\setminus \mathscr{M}^0$. Moreover, $\mathbb{S}$ has a positive minimum on $(\crit \, \mathbb{S}) \cap \mathscr{M}^0$.
\item If $w\in \mathscr{M}^0$ and 
\[
\mathbb{S}(w) < \min_{(\crit \, \mathbb{S}) \cap \mathscr{M}^0} \mathbb{S},
\]
then, setting $(q(\sigma),T(\sigma)):= \phi^{-Z}(\sigma,w)$, $\mathbb{S}(q(\sigma),T(\sigma)) \rightarrow 0$, $T(\sigma)\rightarrow 0$, and $q(\sigma)$ converges in $W^{1,2}$ to a constant loop, for $\sigma \uparrow \sigma_+(w)$.  
\end{enumerate}
\end{prop}

\begin{proof}
Statement (a) is proved in \cite{con06}, Proposition 3.12, in a similar setting. Here we just sketch the argument. Let $(q_h,T_h)$ be a Palais-Smale sequence for $(\mathbb{S},Z)$ at level $S>0$. Since $c(L-\theta)<0$, by Proposition \ref{sulfunz} (d) the sequence $(T_h)$ is bounded from above. By property (Z2) of the pseudo-gradient $Z$, the sequence $\|d\mathbb{S}(q_h,T_h)\|$ is infinitesimal. Moreover, by (\ref{derTvS}),
\[
\| d\mathbb{S}(q_h,T_h)\| \geq \left| \frac{\partial \mathbb{S}}{\partial T} (q_h,T_h) \right| \geq \frac{e_0}{T_h} \mathbb{S}(q_h,T_h) - e_1,
\]
and since $\mathbb{S}(q_h,T_h)\rightarrow S>0$, we conclude that $(T_h)$ is bounded away from zero. Up to a subsequence, we may assume that $T_h \rightarrow T\in ]0,+\infty[$. Then the proof of the compactness of $(q_h)$ in $W^{1,2}(\T,M)$ is analogous to the case of the fixed period action functional (see e.g.\ \cite{ben86} or \cite{as08b}): One uses the upper bound on the action and the quadraticity of the Lagrangian to show that $(q_h)$ is equi-1/2-H\"older continuous, so up to a subsequence it converges uniformly, and the fact that $\|d_q \mathbb{S}(q_h,T_h)\|$ is infinitesimal allows to prove that the latter convergence is in $W^{1,2}$. This concludes the proof of (a).

Since $Z$ is bounded with respect to the product metric on $\mathscr{M}$, if $w$ is an element of $\mathscr{M}$ such that $\sigma_+(w)<+\infty$, then, denoting the orbit of $z$ as $\phi^{-Z}(\sigma,w) =(q(\sigma),T(\sigma))$, we have  
\[
\liminf_{\sigma \uparrow \sigma_+(w)} T(\sigma) = 0.
\]
Since $T(\sigma)>0$ for $\sigma<\sigma_+(w)$, there exists a sequence $\sigma_h\uparrow \sigma_+(w)$ such that $T'(\sigma_h)\leq 0$. However, by property (Z4) of $Z$,
\begin{equation}
\label{stp}
T'(\sigma) = \langle\langle \frac{d}{d\sigma} \phi^{-Z}(\sigma,w) , \frac{\partial}{\partial T} \rangle\rangle = - \langle\langle Z(q(\sigma),T(\sigma)),  \frac{\partial}{\partial T} \rangle\rangle > 0, \mbox{ if } C T(\sigma) \leq \mathbb{S}(q(\sigma),T(\sigma)).
\end{equation}
This fact forces the decreasing function $\sigma\mapsto \mathbb{S}(q(\sigma),T(\sigma))$ to converge to zero for $\sigma\uparrow \sigma_+(w)$. Therefore for every $S>0$, $\phi^{-Z}(\sigma,w)$ eventually enters the positively invariant set $\{\mathbb{S}<S\}$. This proves (b). 

Let $S>0$. By (\ref{stp}), the set $\mathscr{A}_S$ contains no critical points of $\mathbb{S}$ and is negatively invariant with respect to $\phi^{-Z}$. Let $\sigma \mapsto \phi^{-Z}(\sigma,w) = (q(\sigma),T(\sigma))$ be the flow line through a point $w$ in $\{\mathbb{S}>S\}$ and assume that $\sigma_-(w)>-\infty$. By the boundeness of $Z$, we have
\[
\liminf_{\sigma \downarrow \sigma_-(w)} T(\sigma) = 0.
\]
In particular, there exists $\sigma\in ]\sigma_-(w),0[$ such that $T(\sigma)<S/ C$ and hence $(q(\sigma),T(\sigma))\in \mathscr{A}_S$. This shows that the restriction of the local flow $\phi^{-Z}$ to $\{\mathbb{S}>S\}$ is negatively complete with respect to $\mathscr{A}_S$ and concludes the proof of (c).

By the Palais-Smale condition proved in (a) and by Proposition \ref{sulfunz} (a), $\mathbb{S}$ has a minimum on every connected component of $\mathscr{M}\setminus \mathscr{M}^0$. By (a) and (\ref{bfafz}), $\mathbb{S}$ has a minimum on its critical set in the component $\mathscr{M}^0$. This proves (d).

Let $z$ and $(q(\sigma),T(\sigma))$ be as in (e). By (a) and (b), $\mathbb{S}(q(\sigma),T(\sigma)) \rightarrow 0$ for $\sigma \uparrow \sigma_+(w)$. So Proposition \ref{sulfunz} (d) implies that $T(\sigma)\rightarrow 0$, and by Proposition \ref{sulfunz} (c), $q(\sigma)$ converges to a constant loop in $W^{1,2}$, for $\sigma\uparrow \sigma_+(w)$. This concludes the proof of (e).   
\end{proof} \qed

\paragraph{The Morse chain complex of $\mathbb{S}$.} 
The fact that the functional $\mathbb{S}$ admits a pseudo-gradient vector field $Z$ for which the properties (a) and (b) of Proposition \ref{suZ} hold implies that for every $S_0>0$ the {\em relative Morse chain complex} of $\mathbb{S}$ with respect to the sublevel $\{\mathbb{S}<S_0\}$ is well-defined (see \cite{ama06m}). By (\ref{bfafz}) and Proposition \ref{suZ} (d), we can choose $S_0>0$ to be so small that
\begin{equation}
\label{sbfafz}
S_0 < \min \mathrm{critval}\, \mathbb{S},
\end{equation}
where $\mathrm{critval}\, \mathbb{S}$ denotes the set of critical values of $\mathbb{S}$.
Since we are in a Morse-Bott setting, we need to fix also a Morse function $s\in C^{\infty}(\crit\, \mathbb{S})$ and a Riemannian metric $g_{\mathbb{S}}$ on $\crit\, \mathbb{S}$ such that the corresponding negative gradient flow of $s$ is Morse-Smale. By the construction of Section \ref{morcom}, up to the perturbation of the pseudo-gradient vector field $Z$ and of the metric $g_{\mathbb{S}}$, we get a boundary operator
\[
\partial : M_k^{]0,+\infty[}(\mathbb{S},s) \rightarrow M_{k-1}^{]0,+\infty[} (\mathbb{S},s).
\]
The reason why we specify the interval $]0,+\infty[$ for the Morse complex of a functional all of whose critical values are positive is that we wish to reserve the notation $M_*(\mathbb{S},s)$ for an extension of this chain complex, which takes into account also the critical points at infinity of $\mathbb{S}$. See below.
 
The homology of the chain complex $(M_*^{]0,+\infty[} (\mathbb{S},s),\partial)$ is isomorphic to the relative singular homology of the pair $(\mathscr{M},\{\mathbb{S}<S_0\})$. The singular homology of this topological pair is isomorphic to the singular homology of the pair $(\Lambda M, M)$, the space of free loops on $M$ modulo the constant loops $M\subset \Lambda M$, so:
\[
HM_k^{]0,+\infty[} (\mathbb{S},s) \cong H_k(\Lambda M, M), \quad \forall k\in \Z.
\]
In fact, we have the following:

\begin{prop}
\label{homotopy}
If $0<S_0<\min \mathrm{critval}\, \mathbb{S}$, then the topological pair $(\mathscr{M},\{\mathbb{S}<S_0\})$ is homotopically equivalent to $(\Lambda M,M)$.
\end{prop}

\begin{proof}
If $q_0\in M$ is a constant loop, then
\[
\mathbb{S}(q_0,T) = T L(q_0,0) \leq T \max_{q\in M} L(q,0),
\]
so we can find a number $T_0>0$ such that
\begin{equation}
\label{supe+}
\mathbb{S}(q_0,T) < S_0, \quad \forall q_0 \in M, \; \forall T\in ]0,T_0].
\end{equation}
Given $\epsilon>0$, we consider the following open neighborhood of $M$ in $W^{1,2}(\T,M)$:
\[
N_{\epsilon}(M) := \set{q\in W^{1,2} (\T,M)}{\|q'\|_{L^2(\T)}<\epsilon}.
\]
The length of any loop in $N_{\epsilon}(M)$ is less than $\epsilon$, so if $\epsilon$ is smaller than the injectivity radius of $M$, we can find a homotopy 
\[
\psi_s: W^{1,2} (\T,M) \rightarrow W^{1,2}(\T,M), \quad s\in [0,1],
\]
such that $\psi_0=\mathrm{id}$, $\psi_1(N_{\epsilon}(M)) = M$, and $\psi_s|_M = \mathrm{id}$ for every $s\in [0,1]$. By taking Proposition \ref{sulfunz} (c) and (d) into account, it is easy to see that there is a number $S_1>0$ such that
\begin{equation}
\label{supe}
\mathbb{S}(q,T) < S_1\quad \implies \quad \left\{ \begin{array}{l} q\in N_{\epsilon}(M), \;\; T\leq T_0, \\
\mathbb{S}(\psi_s(q),T)< S_0, \;\; \forall s\in [0,1]. \end{array} \right.
\end{equation} 
By Proposition \ref{suZ} (e), we can find a non-negative continuous function $\sigma<\sigma_+$ on $\mathscr{M}$ such that
\begin{equation}
\label{arriva}
\mathbb{S}(w) < S_0 \quad \implies \quad \phi^{-Z}_{\sigma(w)} (w) \in \{\mathbb{S} < S_1\}.
\end{equation}
Denote by $\pi_1$ and $\pi_2$ the two projections on $W^{1,2}(\T,M) \times ]0,+\infty[$. We define the continuous maps
\[
\xymatrix{
\bigl( \mathscr{M} , \{\mathbb{S}<S_0 \} \bigr) \ar@<1ex>[r]^r & \bigl( W^{1,2}(\T,M) , M \bigr) \ar@<1ex>[l]^j}
\]
by 
\[
j(q):= (q,T_0), \quad r(w) := \psi_1\left( \pi_1 \Bigl( \phi_{\sigma(w)}^{-Z} (w) \Bigr)\right).
\]
The map $j$ is well-defined because of (\ref{supe+}). The map $r$ is well-defined because of (\ref{arriva}), (\ref{supe}), and by the fact that $\psi_1$ maps $N_{\epsilon}(M)$ into $M$. For every $s\in [0,1]$, the map
\[
g_s (w) := \left( \psi_s\Bigl( \pi_1 \bigl( \phi_{\sigma(w)}^{-Z} (w) \bigr)\Bigr),  \pi_2 \bigl( \phi_{\sigma(w)}^{-Z} (w) \bigr) \right)
\]
maps $\{\mathbb{S}<S_0\}$ into itself, because of (\ref{arriva}) and (\ref{supe}). Since $\psi_0=\mathrm{id}$, $g_s$ is a homotopy between the map
\begin{equation}
\label{flux}
\bigl( \mathscr{M}, \{\mathbb{S}<S_0\}\bigr) \rightarrow \bigl( \mathscr{M}, \{\mathbb{S}<S_0\}\bigr), \quad w \mapsto \phi_{\sigma(w)}^{-Z}(w), 
\end{equation}
and the map $g_1$. The map $g_1$ is homotopic to $j\circ r$ by the homotopy
\[
h_s(w) := \left( \psi_1\Bigl( \pi_1 \bigl( \phi_{\sigma(w)}^{-Z} (w) \bigr)\Bigr),  sT_0 + (1-s) \pi_2 \bigl( \phi_{\sigma(w)}^{-Z} (w) \bigr) \right),
\]
which satisfies
\[
h_s\bigl(\{\mathbb{S}<S_0\} \bigr) \subset M \times ]0,T_0] \subset \{\mathbb{S}<S_0\}, \quad \forall s\in [0,1],
\]
because of (\ref{arriva}), (\ref{supe}), the fact that $\psi_1(N_{\epsilon}(M))=M$, and (\ref{supe+}).
The map (\ref{flux}) is clearly homotopic to the identity mapping on $(\mathscr{M}, \{\mathbb{S}<S_0\})$, which is then homotopic to $j\circ r$.

By property (Z5) and by the fact that $\psi_s$ restricts to the identity on $M$, the homotopy
\[
k_s(q) := \psi_s \Bigl( \pi_1 \bigl( \phi^{-Z}_{s\sigma(q,T_0)} (q,T_0) \bigr) \Bigr)
\]
maps $M$ into itself and shows that $r\circ j$ is homotopic to the identity mapping on $(W^{1,2}(T,M),M)$. 

We have proved that the maps $r$ and $j$ are homotopy inverses one of the other. The thesis follows from the fact that the inclusion
\[
(\Lambda M, M) \hookrightarrow  (W^{1,2}(\T,M), M)
\]
is a homotopy equivalence.
\end{proof} \qed

If we want to recover the full homology of $\Lambda M$, we need to consider also the {\em critical points at infinity} for the functional $\mathbb{S}$ (in the sense of A.\ Bahri, see \cite{bah89}). These are the virtual critical points that one finds by following the orbit of points $w$ by the local flow of $-Z$ which do not converge in $\mathscr{M}$ for $\sigma\uparrow \sigma_+(w)$. By Proposition \ref{suZ} (a), (b), (e), every such flow a line converges to $(q,0)$, where $q\in M \subset W^{1,2}(\T,M)$ is a constant loop. This fact suggests to extend the auxiliary Morse function $s$ also to the manifold $M$ - the space of critical points at infinity of $\mathbb{S}$ - and to define $M_*(\mathbb{S},s)$ to be the $\Z_2$-vector space generated by all the critical points of $s$. It is then easy to extend the boundary operator $\partial$ to $M_*(\mathbb{S},s)$: if $w$ and $q$ are critical points of $s$, with $w\in \mathscr{M}$ and $q\in M$ of index difference 1, the coefficient of $q$ in $\partial w$ is defined by counting modulo 2 the flow lines with cascades for $(-Z,-\nabla s)$ which converge to some point in $W^s(q;-\nabla s)$, whereas $\partial q$ is defined using only the negative gradient flow of $s$ on $M$. The homology of the chain complex $(M_*(\mathbb{S},s),\partial)$ is then 
\[
HM_k (\mathbb{S},s) \cong H_k (\mathscr{M}) \cong H_k(\Lambda M),
\]
for every $k\in \Z$. 

\paragraph{The differential complex of $\mathbb{S}$.} 
Similarly, the fact that the functional $\mathbb{S}$ admits a pseudo-gradient vector field $Z$ for which the properties (a) and (c) of Proposition \ref{suZ} hold imply that for every $S_0>0$ the Morse differential complex of $\mathbb{S}$ on  $\{\mathbb{S}>S_0\}\setminus \mathscr{A}_{S_0}$ is well-defined. Again, we choose $S_0>0$ smaller than the minimum critical value of $\mathbb{S}$, and we choose the auxiliary Morse function on $\crit\, \mathbb{S}$ to be $-s$. The differential 
\[
\delta : M^k_{]0,+\infty[}(\mathbb{S},-s) \rightarrow M^{k+1}_{]0,+\infty[} (\mathbb{S},-s)
\]
defines a differential complex, whose cohomology is isomorphic to the singular cohomology of the pair 
\[
(\{\mathbb{S}> S_0\} \setminus \mathscr{A}_{S_0}, \{S_0 < \mathbb{S} < S_1\} \setminus \mathscr{A}_{S_0}),
\]
where $S_1>S_0$ is smaller than the minimum critical value of $\mathbb{S}$. By excision, the cohomology of this pair is isomorphic to the cohomology of
$(\mathscr{M}\setminus \mathscr{A}_{S_0},\{\mathbb{S}<S_1\} \setminus \mathscr{A}_{S_0})$.
 As implied by Proposition \ref{compcohom} below, the cohomology of the latter pair is isomorphic to the cohomology of $(\Lambda M, M)$, so
\[
HM^k_{]0,+\infty[}(  \mathbb{S},-s) \cong H^k(\Lambda M,M), \quad \forall k\in \Z.
\]
It remains to state the following:

\begin{prop}
\label{compcohom}
The pair $(\mathscr{M}\setminus \mathscr{A}_{S_0},\{\mathbb{S}<S_1\} \setminus \mathscr{A}_{S_0})$ is homotopically equivalent to $(\Lambda M,M)$.
\end{prop}

\begin{proof}
The proof is analogous to the proof of Proposition \ref{homotopy}, taking into account the fact that the set $\mathscr{A}_{S_0}$ is negatively invariant with respect to $\phi^{-Z}$.
\end{proof} \qed

As before, we recover the cohomology of $\Lambda M$ by considering also the critical points at infinity of $\mathbb{S}$. If the Morse function is extended to $M$, and $M^*(\mathbb{S},-s)$ denotes the direct product of $\Z_2\, w$ for every critical point $w$ of $-s$, the differential $\delta$ extends to $M^*(\mathbb{S},-s)$, and its cohomology is
\[
HM^k (\mathbb{S},-s) \cong H^k (\mathscr{M}) \cong H^k (\Lambda M),
\]
for every $k\in \Z$.   

\section{The chain maps $\Xi, \Upsilon$ and the proof of Theorem \ref{main2}}
\label{finale}

Let $W\subset T^*M$, $\theta$, $\lambda_{\theta}$, $H$, and $L$ be as in Section \ref{scdsec}.
The aim of this last section is to construct the chain maps
\[
\Xi : M_* (\mathbb{S},s) \rightarrow RF_*(W,\lambda_{\theta},a), \quad \Upsilon : RF_*(W,\lambda_{\theta},a) \rightarrow M^{1-*}(\mathbb{S},-s),
\]
and to show their properties, which lead to the proof of Theorem \ref{main2} of the Introduction. Since most of the arguments are similar to the case of the chain maps $\Phi$ and $\Psi$, we omit most of the details.

\paragraph{Definition of $\Xi$ and $\Upsilon$.} We assume that the auxiliary Morse functions $a$ on $\crit \, \mathbb{A}$ and $s$ on $(\crit\, \mathbb{S}) \cup M$ (we are including the critical points at infinity of $\mathbb{S}$) are related by the analogue of conditions (A1-A4) of Section \ref{criti}, where $s$ plays the role of $e$. If $w$ is a critical point of $s$ (possibly a critical point at infinity) and $z$ is a critical point of $a$, the space $\mathcal{M}_{\Xi}(w,z)$ is the set of pairs $([(w_1,\dots,w_m)], [(v_1,\dots,v_k)])$ in $\mathcal{W}^u(w;-Z,-\nabla s) \times \mathcal{M}^s(z)$ coupled by the conditions
\begin{equation}
\label{coupling3}
q_m(t) = \pi \circ u_1(0,t), \quad \eta_1(0) = T_m,
\end{equation}
where $v_1(s,t)=(u_1(s,t),\eta_1(s))$ and $w_m = (q_m,T_m)$ is in $\mathscr{M}$ if $T_m>0$, or is a critical point at infinity of $\mathbb{S}$ if $T_m=0$ (the unstable manifold with cascades $\mathcal{W}^u$ and the spaces of half-flow lines with cascades $\mathcal{M}^s$, $\mathcal{M}^u$ for the Rabinowitz-Floer equation are defined in Sections \ref{morcom} and \ref{phisec}, respectively). 

Similarly, the space $\mathcal{M}_{\Upsilon}(z,w)$ is the set of pairs $([(v_1,\dots,v_k)], [(w_1,\dots,w_m)])$ in $\mathcal{M}^u(z) \times \mathcal{W}^u(w;-Z,\nabla s)$ coupled by the conditions
\begin{equation}
\label{coupling4}
q_m(t) = \pi \circ u_k(0,-t), \quad \eta_k(0) = - T_m,
\end{equation}
where $v_k(s,t)=(u_k(s,t),\eta_k(s))$ and $w_m = (q_m,T_m)$ are as before.

Let $\mathbb{A}_{H,\lambda}$, respectively $\mathbb{A}_{H,\lambda_{\theta}}$, be the Rabinowitz action functionals associated to the Hamiltonian $H$ and to the standard Liouville form $\lambda$, resp.\ to the perturbed Liouville form $\lambda_{\theta}$. Since the form $\theta$ is closed, the difference $\mathbb{A}_{H,\lambda} - \mathbb{A}_{H,\lambda_{\theta}}$ is constant on every connected component of $C^{\infty}(\T,T^*M) \times \R$. 
  
The energy estimates for the elements of $\mathcal{M}_{\Xi}(w,z)$ and $\mathcal{M}_{\Upsilon}(z,w)$ are based on the following fact, which can be seen as a generalization of Lemma \ref{levrel}:

\begin{lem}
\label{levrel2}
Let $L\in C^{\infty}(TM)$ be a Tonelli Lagrangian and let $H\in C^{\infty}(T^*M)$ be  its Fenchel dual Hamiltonian. Let $x=(q,p)\in C^{\infty}(\T,T^*M)$. 
Then
\[
\mathbb{A}_{H,\lambda} (x,T) \leq \mathbb{S}_L(q,T), \quad \forall T>0,
\]
and the equality holds if and only if $p(\tau) = d_v L(q(\tau),q'(\tau))$ for every $\tau\in \T$. Similarly,
\[
\mathbb{A}_{H,\lambda} (x(-\cdot),-T) \geq - \mathbb{S}_L(q,T), \quad \forall T>0,
\]
and the equality holds if and only if $p(-\tau) = d_v L(q(\tau),q'(\tau))$ for every $\tau\in \T$.
\end{lem}

\begin{proof}
This is a direct consequence of Fenchel duality. In fact, 
\begin{eqnarray*}
\mathbb{S}_L(q,T) = T \int_{\T} L(q(\tau),q'(\tau)/T)\, d\tau = T \int_{\T} \max_{\xi \in T_{q(\tau)}^* M} \Bigl( \langle \xi, q'(\tau)/T \rangle - H(q(\tau),\xi) \Bigr)\, d\tau \\ \geq \int_{\T} \langle p(\tau), q'(\tau) \rangle \, d\tau - T \int_{\T} H(q(\tau),p(\tau))\, dt = \mathbb{A}_{H,\lambda} (x,T),
\end{eqnarray*}
and the equality holds if and only if $x(\tau)$ is the image of $(q(\tau),q'(\tau))$ by the Legendre transform, for every $\tau\in \T$. 
The other statements follows from the identity
\[
\mathbb{A}_{H,\lambda}(x,T) = - \mathbb{A}_{H,\lambda} (x(- \cdot), -T).
\]
\end{proof} \qed
 
If $(\underline{w},\underline{v})$ is an element of $\mathcal{M}_{\Xi}(w,z)$, the above lemma and the coupling condition (\ref{coupling3}) imply the inequalities
\begin{equation}
\label{ffgghh}
\begin{split}
\mathbb{A}_{H,\lambda_{\theta}} (z) \leq \mathbb{A}_{H,\lambda_{\theta}} (v_j(s)) \leq
\mathbb{A}_{H,\lambda_{\theta}} (v_1(0)) = \mathbb{A}_{H,\lambda} (v_1(0)) - \int_{\T} q_m^* (\theta) \\ \leq \mathbb{S}_L (q_m,T_m) -  \int_{\T} q_m^* (\theta) = \mathbb{S}_{L-\theta}  (q_m,T_m) \leq  \mathbb{S}_{L-\theta}  (w_h) \leq 
\mathbb{S}_{L-\theta}(w),
\end{split} \end{equation}
which allows to prove the compactness of the space $\mathcal{M}_{\Xi}(w,z)$. The estimates for $\mathcal{M}_{\Upsilon}(z,w)$ are similar.
Automatic transversality for the stationary solutions follows from the differential version of Lemma \ref{levrel2}, analogous to Lemma \ref{levreldiff}. The usual counting procedure defines the chain maps $\Xi$ and $\Upsilon$. 

\paragraph{Properties of $\Xi$ and $\Upsilon$.} The chain maps $\Xi$ and $\Upsilon$ preserve the splitting of the complexes associated to the different free homotopy classes of loops in $M$. Moreover, the inequality (\ref{ffgghh}) and the analogous inequality for the elements of $\mathcal{M}_{\Upsilon}(z,w)$ imply that they preserve the action filtrations, meaning that for every $A\geq 0$ there holds
\[
\Xi: M_*^{[0,A]}(\mathbb{S},s) \rightarrow RF_*^{]-\infty,A]}(W,\lambda_{\theta},a), \quad \Upsilon: RF_*^{]-\infty,-A]} (W,\lambda_{\theta},a) \rightarrow M^{1-*}_{[A,+\infty[} (\mathbb{S},-s).
\]
Furthermore, the same inequalities and automatic transversality at the stationary solutions imply that $\Xi$ has a left inverse $\hat{\Xi}$ with kernel $RF^-$, and $\Upsilon$ has a right inverse $\hat{\Upsilon}$ with image $RF^-$. As before, $RF^-$ is the subspace of $RF$ generated by critical points $z$ of $a$ which have either negative action $\mathbb{A}(z)$, or are of the form $z=z_q^-$, for $q$ a critical point of $s$ on $M$, the space of critical points at infinity of $\mathbb{S}$.   

If we consider just the complexes corresponding to the true critical points of $\mathbb{S}$, that is $M_*^{]0,+\infty[}(\mathbb{S},s)$ and $M^*_{]0,+\infty[}(\mathbb{S},-s)$, the chain maps $\Xi$ and $\Upsilon$ induce chain maps
\[
\tilde{\Xi} : M_*^{]0,+\infty[}(\mathbb{S},s) \rightarrow RF_*^{]0,+\infty[} (W,\lambda_{\theta},a), \quad \tilde{\Upsilon} : RF_*^{]-\infty,0[} (W,\lambda_{\theta},a) \rightarrow M^{1-*}_{]0,+\infty[}(\mathbb{S},-s),
\]
which are isomorphisms. In particular,
\[
HRF_*^{]0,+\infty[} (W,\lambda_{\theta},a) \cong H_*(\Lambda M,M), \quad HRF_* ^{]-\infty,0[} (W,\lambda_{\theta},a) \cong H^{1-*} (\Lambda M, M).
\]
Arguing as in the proof of Proposition \ref{Pprop}, one sees that the composition $\Upsilon \circ \Xi$ is chain homotopic to zero, that is there is a homomorphism 
\[
Q : M_* (\mathbb{S},s) \rightarrow M^{-*} (\mathbb{S},-s),
\]
such that
\[
\Upsilon \circ \Xi = Q \partial + \delta Q,
\]
and $Q q_{\min} \in M^0_{]0,+\infty[} (\mathbb{S},-s)$.
Because of index reasons, $Q$ can be non-zero only from $M_0(\mathbb{S},s)$ to $M^0(\mathbb{S},-s)$. Then the chain map
\[
\Omega : M_*(\mathbb{S},s) \rightarrow RF_*(W,\lambda_{\theta},a), \quad \Omega := \Xi - \hat{\Upsilon} Q \partial - \partial \hat{\Upsilon} Q,
\]
is chain homotopic to $\Xi$ and, arguing as in the proof of Theorem \ref{main}, one can show that the short sequence 
\[
0 \rightarrow M_*(\mathbb{S},s) \stackrel{\Omega}{\rightarrow} RF_*(W,\lambda_{\theta},a) \stackrel{\Upsilon}{\rightarrow} M^{1-*}(\mathbb{S},-s) \rightarrow 0
\]
is exact. The induced long exact sequence is
\[
\dots \rightarrow HM_k(\mathbb{S},s) \stackrel{\Omega_* = \Xi_*}{\longrightarrow} HRF_k (W,\lambda_{\theta},a) \stackrel{\Upsilon_*}{\rightarrow} HM^{1-k} (\mathbb{S},-s) \stackrel{\Delta_*}{\rightarrow} HM_{k-1}(\mathbb{S},s) \rightarrow \dots
\]
where the connecting homomorphism 
\[
\Delta_* : HM^{-k}(\mathbb{S},-s) \cong H^{-k} (\Lambda M) \rightarrow H_k (\Lambda M) \cong HM_k (\mathbb{S},s)
\]
can be non-zero only for $k=0$, where it coincides with the composition
\[
\Delta_* = i_0 \circ \chi(T^*M) \circ \pi_0,
\]
see the discussion after the statement of Theorem \ref{main}. This concludes the proof of Theorem \ref{main2} of the Introduction.

\renewcommand{\thesection}{\Alph{section}}
\setcounter{section}{0} 

\section{Appendix - An Alexandrov-type maximum principle with Neumann conditions on part of the boundary}

The aim of this appendix is to prove Theorem \ref{alek}. The proof 
is similar to the proof of the classical Aleksandrov weak maximum principle contained in Section 9.1 of \cite{gt83}. 

Theorem \ref{alek} is stated for a first order perturbation of the Laplace operator on a bounded open subset $\Omega$ of the half-cylinder $]0,+\infty[ \times \T$, but  it is useful to translate it into an equivalent statement for a domain of the plane. To do this, we transform $\Omega$ into $\tilde{\Omega} := \varphi(\Omega)$  by means of the conformal diffeomorphism
\[
\varphi: \R^+ \times \T \rightarrow \C \setminus B_1(0), \quad \varphi(s,t) := e^{2\pi(s+ i t)}.
\]
If $u$ is a regular function on $\Omega$, then $\tilde{u}:= u\circ \varphi^{-1}$ is a regular function on $\tilde{\Omega}$ and 
\[
\nabla u(\varphi^{-1}(z))  = 2\pi \overline{z} \, \nabla \tilde{u} (z) , \quad
\Delta u(\varphi^{-1}(z))  = 4\pi^2 |z|^2 \, \Delta \tilde{u} (z).
\]
Therefore, $u$ satisfies the elliptic differential inequality
\[
\Delta u + b \cdot \nabla u \geq f
\]
if and only if $\tilde{u}$ satisfies the  elliptic differential inequality
\[
\Delta \tilde{u} + \tilde{b} \cdot \nabla \tilde{u} \geq \tilde{f},
\]
with
\[
\tilde{b}(z) := \frac{z}{2\pi |z|^2} \, b(\varphi^{-1}(z)), \quad
\tilde{f}(z):= \frac{1}{4\pi^2 |z|^2} \, f(\varphi^{-1}(z)).
\]
The set $\Omega$ is bounded if and only if $\tilde{\Omega}$ is bounded. In this case, there is a constant $c\geq 1$ such that 
\[
\frac{1}{c} \|b\|_{L^2(\Omega)} \leq \|\tilde{b}\|_{L^2(\tilde{\Omega})} \leq c \|b\|_{L^2(\Omega)}, \quad \frac{1}{c} \|f^-\|_{L^2(\Omega)} \leq \|\tilde{f}^-\|_{L^2(\tilde{\Omega})} \leq c \|f^-\|_{L^2(\Omega)}.
\]
Finally, the partial derivative with respect to $s$ of $u$ at a point $(0,t)$ becomes the radial derivative of $\tilde{u}$ at the point $e^{2\pi i t}$.

By these observations, Theorem \ref{alek} can be restated in the following equivalent way.  Let $\Omega$ be a bounded open subset of $\R^2$, assumed to be disjoint from the open unit ball $B_1(0)$, and consider the following partition $\{\Sigma,\Sigma'\}$ of $\partial \Omega$:
\[
\Sigma := \overline{\partial \Omega \setminus \partial B_1(0)}, \quad \Sigma' := \partial \Omega \setminus \Sigma.
\]
Notice that if $z\in \Sigma'$, then $|z|=1$ and there exists $\epsilon>0$ such that the segment $]1,1+\epsilon[ \, z$ is contained in $\Omega$. In particular, any $u\in C^1(\overline{\Omega})$ has a well defined radial derivative $\partial u/\partial r$ at such a point $z$. Then we have:

\begin{thm}
\label{alek2}
Let $\Omega$ be as above. Then for every map $b \in L^2(\Omega,\R^2)$ there exists a number $C$ depending only on $d:= \diam \Omega$ and on $\|b\|_{L^2(\Omega)}$ such that for every $f\in L^1_{\mathrm{loc}}(\Omega)$ and every $u\in C^2(\Omega) \cap C^1(\overline{\Omega})$ which satisfies
\begin{eqnarray}
\label{eins}
\Delta u + b \cdot \nabla u & \geq & f \quad \mbox{in }\Omega, \\
\label{zwei}
\frac{\partial u}{\partial r} & \geq & 0 \quad \mbox{on } \Sigma',
\end{eqnarray}
there holds
\begin{equation}
\label{drei}
\sup_{\Omega} u \leq \sup_{\Sigma} u + C \|f^-\|_{L^2(\Omega)},
\end{equation}
where $f^-$ denotes the negative part of $f$.
\end{thm}

\begin{proof}
If $v\in C^0(\overline{\Omega})$ and $z_0\in \Omega$, we denote by $\chi_v(z_0)$ the set of upper differentials for $v$ at $z$, that is 
\[
\chi_v(z_0) := \set{p\in \R^2}{v(z) \leq v(z_0) + p\cdot (z-z_0) \; \forall z\in\overline{ 
\Omega}}.
\]
The upper contact set $\Gamma_v$ of $v$ is the set of all $z_0$ in $\Omega$ for which $\chi_v(z_0)$ is not empty. Notice that if $z_0\in \Gamma_v$ and $v$ is twice differentiable at $z_0$, then
\begin{equation}
\label{null}
\chi_v(z_0) = \{ \nabla v(z_0)\} \quad \mbox{and} \quad D^2 v(z_0) \leq 0.
\end{equation}
By a slight abuse of notation, we indicate by $\chi_v(\Gamma_v)$ the union of all the sets $\chi_v(z_0)$ for $z_0\in \Gamma_v$, and when $v$ is twice differentiable we consider $\chi_v$ as a map from  $\Gamma_v$ into $\R^2$.

\medskip

\noindent {\sc Claim 1.} {\em Assume that $u\in C^2(\Omega)\cap C^1(\overline{\Omega})$ satisfies (\ref{zwei}),
\begin{equation}
\label{zweibis}
u \leq 0 \quad \mbox{on } \Sigma,
\end{equation}
and
\begin{equation}
\label{drei+}
\max_{\overline{\Omega}} u = u(z_0) > 0.
\end{equation}
Then $\chi_u(\Gamma_u)$ contains the half ball
\[
\set{p\in \R^2}{|p| < \frac{u(z_0)}{d}, \; p\cdot z_0 < 0}.
\]}

\medskip

By (\ref{zwei}), (\ref{zweibis}), and (\ref{drei+}), the maximum point $z_0$ belongs to $\Omega$. Let $v$ be the function whose graph is the cone of vertex $(z_0,u(z_0))$ and base $\partial \Omega$. Since $\Omega \subset B_d(z_0)$, the set $\chi_v(\Gamma_v)$ contains the ball $B_{u(z_0)/d}(0)$, and it is enough to show that any $p$ in $\chi_v(\Gamma_v)$ with $p\cdot z_0 <0$ belongs to $\chi_u(\Gamma_u)$. Consider such a $p$, and let $t_*$ be the minimum of all numbers $t$ such that
\[
u(z) \leq t + p \cdot z, \quad \forall z\in \overline{\Omega}.
\]
By minimality, there exists a point $z_*\in \overline{\Omega}$ such that the above inequality is an equality when $t=t_*$ and $z=z_*$. Therefore,
\begin{equation}
\label{vier} 
u(z) \leq u(z_*) + p \cdot (z-z_*), \quad \forall z\in \overline{\Omega}.
\end{equation}
If we prove that $z_*$ does not belong to $\partial \Omega$, then $z_*$ is in $\Gamma_u$ and $p$ belongs to $\chi_u(\Gamma_u)$, as wished. 

Assume first that $z_*$ is in $\Sigma'$. By (\ref{zwei}),
\[
u((1+\epsilon)z_*) = u(z_*) + \epsilon \frac{\partial u}{\partial r} (z_*) + o(\epsilon) \geq  u(z_*) + o(\epsilon), \quad \mbox{for } \epsilon \downarrow 0,
\]
while by (\ref{vier}),
\[
u((1+\epsilon)z_*) \leq u(z_*) + \epsilon p \cdot z_*.
\]
By comparing the above two inequalities as $\epsilon \downarrow 0$, we deduce that $p \cdot z_* \geq 0$. Then the fact that $z_0$ is a maximum point and (\ref{vier}) imply
\[
u(z_*) \leq u(z_0) \leq u(z_*) + p \cdot (z_0 - z_*) \leq u(z_*) + p\cdot z_0,
\]
so $p\cdot z_0 \geq 0$, contradicting our assumption on $p$. 

Assume now that $z_*$ is in $\Sigma$. Then (\ref{vier}) and (\ref{zweibis}) imply
\begin{equation}
\label{funf}
u(z_0) \leq u(z_*) + p \cdot (z_0 - z_*) \leq p \cdot (z_0 - z_*).
\end{equation}
On the other hand, the fact that $p$ is in $\chi_v(\Gamma_v)=\chi_v(z_0)$ implies 
\begin{equation}
\label{sechs}
0 = v(z_*) \leq v(z_0) + p \cdot (z_* - z_0) = u(z_0) + p\cdot (z_* - z_0).
\end{equation}
From (\ref{funf}) and (\ref{sechs}) we deduce that
\[
p\cdot (z_0 - z_*) = u(z_0).
\]
Then by (\ref{vier}),
\[
u(z_0 + \epsilon (z_*-z_0) ) \leq u(z_*) + p\cdot ( z_0 + \epsilon (z_*-z_0) - z_*) = (1-\epsilon) p \cdot (z_0 - z_*) = u(z_0) - \epsilon u(z_0),
\]
which contradicts the fact that $u$ is differentiable at the interior maximum point $z_0$. This concludes the proof of Claim 1.

\medskip

\noindent {\sc Claim 2.} {\em Let $g$ be a non-negative continuous function on $\R^+$. Then for every $u\in C^2(\Omega) \cap C^1(\overline{\Omega})$ which satisfies (\ref{zwei}), there holds
\[
\int_0^M g(\rho)\rho\, d\rho \leq \frac{1}{4\pi} \int_{\Gamma_u} g(|\nabla u|) |\Delta u|^2 d\lambda^2,
\]
where 
\[
M:= \frac{1}{d} \left( \sup_{\Omega} u - \sup_{\Sigma} u\right),
\]
and $d\lambda^2$ denotes the Lebesgue measure on $\R^2$. }

\medskip

Up to the sum of a constant, we may assume that $\sup_{\Sigma} u = 0$, and in particular (\ref{zweibis}) holds. If $M=0$, there is nothing to prove, so we may assume that also (\ref{drei+}) holds. Then Claim 1 implies that the half ball
\[
B_M^+  := B_M(0) \cap \set{p\in \R^2}{p\cdot z_0 < 0}
\]
is contained in $\chi_u(\Gamma_u)$. Then
\[
\int_0^M g(\rho) \rho\, d\rho = \frac{1}{\pi} \int_{B_M^+} g(|p|)\, d\lambda^2 \leq \frac{1}{\pi} \int_{\chi_u(\Gamma_u)} g(|p|)\, d\lambda^2 \leq \frac{1}{\pi} \int_{\Gamma_u} g (|\nabla u|) |\det D^2 u| \, d\lambda^2,
\]
where in the last inequality we have used (\ref{null}) and the change of variable formula (by using the fact that for every $\epsilon>0$ the map $z\mapsto \nabla u (z) - \epsilon z$ is injective on $\Gamma_u$, one could show that the last inequality is actually an equality). Then the conclusion follows from the inequality
\[
\det D^2 u \leq \left( \frac{\tr D^2 u}{2} \right)^2 = \frac{1}{4} 
|\Delta u|^2.
\]

\noindent {\sc Conclusion.} Now assume that $u\in C^2(\Omega) \cap C^1(\overline{\Omega})$ satisfies (\ref{eins}) and (\ref{zwei}). Let
\[
g(\rho) := \frac{1}{\rho^2 + \theta^2},
\]
for some $\theta>0$ to be determined. By (\ref{eins}), we have the chain of inequalities
\begin{eqnarray*}
- \Delta u \leq b\cdot \nabla u - f \leq |b| \, |\nabla u| + f^- = \bigl(|b|^2 |\nabla u|^2 + (f^-)^2 + 2 |b| \, |\nabla u| f^- \bigr)^{1/2} \\  \leq \left( |b|^2 |\nabla u|^2 + (f^-)^2 + \theta^2 |b|^2 + \frac{1}{\theta^2} |\nabla u|^2 (f^-)^2 \right)^{1/2} 
= \left( |b|^2 + \frac{1}{\theta^2} (f^-)^2 \right)^{1/2} \bigl( |\nabla u|^2 + \theta^2 \bigr)^{1/2} \\ = \left( |b|^2 + \frac{1}{\theta^2} (f^-)^2 \right)^{1/2} g(|\nabla u|)^{-1/2}.
\end{eqnarray*}
By (\ref{null}), $\Delta u \leq 0$ on $\Gamma_u$, so the above inequality implies that
\[
g(|\nabla u|) |\Delta u|^2 \leq |b|^2 + \frac{1}{\theta^2} (f^-)^2 \quad \mbox{on } \Gamma_u.
\]
Claim 2 and the above inequality yield
\begin{eqnarray*}
\frac{1}{2} \log \left( 1 + \frac{M^2}{\theta^2} \right) = \int_0^M \frac{\rho}{\rho^2 + \theta^2}\, d\rho = \int_0^M g(\rho)\rho\, d\rho \\
\leq \frac{1}{4\pi} \int_{\Gamma_u} g(|\nabla u|) |\Delta u|^2 \, d\lambda^2 \leq \frac{1}{4\pi}  \int_{\Gamma_u} \left( |b|^2 + \frac{1}{\theta^2} (f^-)^2 \right)\, d\lambda^2 \leq \frac{1}{4\pi} \left( \|b\|_{L^2(\Omega)}^2 + \frac{1}{\theta^2} \|f^-\|_{L^2(\Omega)}^2 \right),
\end{eqnarray*}
from which we get
\begin{equation}
\label{sieben}
\frac{M^2}{\theta^2} \leq 1 + \frac{M^2}{\theta^2} \leq  \exp \left( \frac{1}{
2 \pi} \Bigl( \|b\|_{L^2(\Omega)}^2 + \frac{1}{\theta^2} \|f^-\|_{L^2(\Omega)}^2 \Bigr) \right).
\end{equation}
When $f^-$ is not zero almost everywhere on $\Omega$, we choose $\theta = \|f^-\|_{L^2(\Omega)}$ and we obtain
\begin{equation}
\label{acht}
M^2 \leq \|f^-\|_{L^2(\Omega)}^2 \exp \left( \frac{1}{2\pi} \Bigl( \|b\|_{L^2(\Omega)}^2 + 1\Bigr) \right).
\end{equation}
When $f^-=0$ a.e.\ on $\Omega$, we can let $\theta\downarrow 0$ in (\ref{sieben}) and we get that $M=0$, so (\ref{acht}) holds also in this case. By the definition of $M$ and by (\ref{acht}), the conclusion (\ref{drei}) holds with
\[
C := d \cdot \exp \left( \frac{1}{4 \pi} \Bigl( \|b\|_{L^2(\Omega)}^2 + 1 \Bigr) \right).
\]
\end{proof} \qed

\begin{rem}
Theorem \ref{alek2} readily extends to bounded domains $\Omega$ in $\R^n \setminus B_1(0)$, by replacing the $L^2$ norms by $L^n$ norms. The regularity of $u$ in the assumptions could be reduced to the requirement that $u$ is in $C^0(\overline{\Omega}) \cap W^{2,n}_{\mathrm{loc}}(\Omega)$ and has a (non-negative) radial derivative at every $z\in \Sigma'$. The Laplacian could be replaced by any uniformly elliptic operator with bounded coefficients.
\end{rem}

\section{Appendix - Morse theoretical description of a Gysin sequence}

Let $M$ be a closed smooth connected manifold of dimension $n\geq 2$, and let $\pi: S^*M \rightarrow M$ be the cotangent unit sphere bundle over $M$. The aim of this appendix is to describe the associated Gysin sequence in terms of the Morse chain complexes of $S^*M$ and $M$. We use the notation of Section \ref{morcom}. The Morse complex $M_*$ and the singular homology $H_*$ are based on $\Z_2$ coefficients.

More precisely, we wish to give Morse theoretical descriptions of the homomorphisms
\[
\pi_* : H_*(S^*M) \rightarrow H_*(M), \quad \pi_! : H_*(M) \rightarrow H_{*+n-1}(S^*M),
\]
where $\pi_!$ is the transfer (or umkehr) homomorphism, that is the homomorphism obtained by composing $\pi^* : H^*(M) \rightarrow H^*(S^*M)$ with Poincar\'e duality, and of the Gysin exact sequence
\begin{equation}
\label{gysin}
\rightarrow H_k(M) \stackrel{\pi_!}{\longrightarrow} H_{k+n-1}(S^*M) \stackrel{\pi_*}{\longrightarrow} H_{k+n-1}(M) \stackrel{e\cap}{\longrightarrow} H_{k-1}(M) \rightarrow
\end{equation}
where $e\in H^n(M)$ denotes the Euler class of $T^*M$. In the case of this $(n-1)$-sphere bundle, the homorphism $e\cap$ can be non-trivial only from $H_n(M)$ to $H_0(M)$, and between such spaces it is the homomorphism
\[
e\cap: H_n(M) \rightarrow H_0(M), \quad [M] \mapsto \chi(T^*M) [\mathrm{pt}],
\]
where $[M]\in H_n(M)$ is the fundamental class, $[\mathrm{pt}]\in H_0(M)$ is the class given by a point in $M$, and $\chi(T^*M)\in \Z_2$ is the Euler number of $T^*M$, which coincides with the Euler characteristic of $M$ (up to the sign $(-1)^n$ that we can ignore because we are using $\Z_2$ coefficients). The exact sequence (\ref{gysin}) allows to compute the homology of $S^*M$ as follows:
\begin{equation}
\label{explicit}
H_k(S^*M) \cong \left\{ \begin{array}{ll} H_k(M) & \mbox{if } k\leq n-2, \\ H_{n-1}(M) \oplus \bigl(1-\chi(T^*M)\bigr) \Z_2 & \mbox{if } k=n-1, \\ H_1(M) \oplus \bigl(1- \chi(T^*M)\bigr) \Z_2 & \mbox{if } k=n, \\ H_{k-n+1}(M) & \mbox{if }k\geq n+1. \end{array} \right.
\end{equation}

Let $f\in C^{\infty}(M)$ and $h\in C^{\infty}(S^*M)$ be Morse functions. 
If we assume that
\begin{equation}
\label{nectra1}
\forall x\in \crit\, h \mbox{ such that } \pi(x) \in \crit\, f \mbox{, there holds } \ind(x;h) \geq \ind (\pi(x);f),
\end{equation}
then for generic Riemannian metrics $g_S$ on $S^*M$ and $g_M$ on $M$ the following fact holds: For every pair of critical points $x\in \crit\, h$, $q\in \crit\, f$, the unstable manifold $W^u(x;-\nabla h)$ of $x$ is transverse to the pull-back $\pi^{-1}(W^s(q;-\nabla f))$ of the
the stable manifold of $q$. Notice that the condition (\ref{nectra1}) is also necessary for the latter fact to hold: In fact, if $x\in \crit\, h$ is such that $\pi(x)\in \crit\, f$, then the transversality of $W^u(x;-\nabla h)$ and $\pi^{-1}(W^s(\pi(x);-\nabla f))$ implies that
\[
T_x W^u(x;-\nabla h) + d\pi(x)^{-1} \Bigl[ T_{\pi(x)} W^s(\pi(x);-\nabla f) \Bigr] = T_x S^*M,
\]
and by comparing the dimensions we deduce that $\ind(x;h) \geq \ind (\pi(x);f)$. 

Let $g_S$ and $g_M$ be Riemannian metrics on $S^*M$ and $M$ such that that the above transversality requirement holds, and that the negative gradient flows of $h$ and $f$ are Morse-Smale. Then, if $x\in \crit\, h$ and $q\in \crit\, f$ have the same Morse index, the set
\[
W^u(x;-\nabla h) \cap \pi^{-1} \bigl(W^s(q;-\nabla f)\bigr)
\]
is finite, and we denote by $n_{\psi}(x,q)\in \Z_2$ its parity. The homomorphism
\[
\psi : M_*(h) \rightarrow M_*(f), \quad x \mapsto \sum_{\substack{q\in \crit\, f \\ \ind(q;f) = \ind(x;h)}} n_{\psi}(x,q)\, q,
\]
is a chain map, and it induces the homomorphism $\pi_*$ in homology (see e.g.\ Appendix A in \cite{as08}):
\[
\psi_* = \pi_* : HM_*(h) \cong H_*(S^*M) \rightarrow HM_*(f) \cong H_*(M).
\]
Similarly, the condition 
\begin{equation}
\label{nectra2}
\forall x\in \crit\, h \mbox{ such that } \pi(x) \in \crit\, f \mbox{, there holds } \ind(x;h) \leq \ind (\pi(x);f) + n-1,
\end{equation}
implies that if the metrics $g_S$ and $g_M$ are generic, then for  
every pair of critical points $x\in \crit\, h$, $q\in \crit\, f$, the pull-back $\pi^{-1}(W^u(q;-\nabla f))$ of the unstable manifold of $q$ is transverse to the stable manifold $W^s(x;-\nabla h)$ of $x$. Then, if the Morse indices of $q\in \crit\, f$ and $x\in \crit \, h$ are related by the formula $\ind(x;h)=\ind(q;f) + n -1$, the set
\[
\pi^{-1}\bigl(W^u(q;-\nabla f)\bigr) \cap W^s(x;-\nabla h)
\]
is finite, and we denote by $n_{\phi}(q,x)\in \Z_2$ its parity. The homomorphism
\[
\phi : M_*(f) \rightarrow M_{*+n-1}(h), \quad q \mapsto \sum_{\substack{x\in \crit\, h \\ \ind(x;h) = \ind(q;f)+n-1}} n_{\phi}(q,x)\, x,
\]
is a chain map, and it induces the transfer homomorphism $\pi_!$ in homology (see e.g.\ \cite{cs08}):
\[
\phi_* = \pi_! : HM_*(f) \cong H_*(M) \rightarrow HM_{*+n-1}(h) \cong H_{*+n-1}(S^*M).
\]
Now assume that the Morse functions $f\in C^{\infty}(M)$ and $h\in C^{\infty}(S^*M)$ satisfy the following additional requirements:
\begin{enumerate}

\item[(1)] $f$ has a unique minimum point $q_{\min}$, a unique maximum point $q_{\max}$, and is self-indexing, that is $f(q)=\ind(q;f)$ for every $q\in \crit\, f$.

\item[(2)] $f \circ \pi \leq h \leq f \circ \pi + 1/2$ on $S^*M$.

\item[(3)] Every critical point of $h$ lies above a critical point of $f$, and for every critical point $q$ of $f$ the fiber $\pi^{-1}(q)$ contains exactly two critical points of $h$, that we denote by $x_q^-$ and $x_q^+$, such that $h(x_q^-)=f(q)$ and $h(x_q^+)=f(q)+1/2$.

\item[(4)] For every $q$ in $\crit\, f$, we have $\ind(x_q^-;h) = \ind(q;f)$ and $\ind(x^+_q;h) = \ind(q;f) + n-1$.

\end{enumerate}

If $f\in C^{\infty}(M)$ is any Morse function which satisfies (1) (see e.g.\ \cite{mil65b} for the construction of such functions), it is easy to construct a Morse function $h\in C^{\infty}(S^*M)$ so that the conditions (2-3-4) hold. In fact, let $h_0$ be a smooth Morse function on $S^{n-1}$ with a unique maximum, a unique minimum, and such that
\[
\min_{S^{n-1}} h_0 = 0, \quad   \max_{S^{n-1}} h_0 = \frac{1}{2}.
\]
Let us fix pairwise disjoint open neighborhoods $U_q\subset M$ of each critical point $q$ of $f$, and local trivializations
\begin{equation}
\label{loctriv}
\pi^{-1}(U_q) \cong \R^n \times S^{n-1},
\end{equation}
mapping $\pi^{-1}(0)$ onto $\{0\}\times S^{n-1}$. If $\varphi\in C^{\infty}(\R^n)$ is a smooth compactly supported function such that $\varphi\equiv 1$ in a neighborhood of $0$, $0\leq \varphi \leq 1$, and
\begin{equation}
\label{picc}
\|d\varphi(\xi)\| \leq \|df(\xi)\|, \quad \forall \xi\in \R^n,
\end{equation}
we define $h$ in $\pi^{-1}(U_q)$ by using the local trivialization (\ref{loctriv}) as
\[
h(\xi,\eta) := f(\xi) + \varphi(\xi) h_0(\eta), \quad \forall (\xi,\eta)\in \R^n \times S^{n-1}.
\]
Such a function extends smoothly to the whole $S^*M$ by defining $h(x) := f(\pi(x))$ for $x$ outside the union of all $\pi^{-1}(U_q)$'s, for $q\in \crit \, f$. The condition (\ref{picc}) guarantees that $h$ has no other critical points than those in $\pi^{-1} (\crit\, f)$. By construction, $h$ satisfies (2-3-4).

Notice that by the assumption (4), both the conditions (\ref{nectra1}) and (\ref{nectra2}) are satisfied. If $\xi\in M_*(f)$ is a chain, 
\[
\xi = \sum_{q\in \crit\, f} \xi_q \, q,
\]
then we denote by $x_{\xi}^-$ and $x_{\xi}^+$ the corresponding chains in $M_*(h)$, defined by
\[
x_{\xi}^- := \sum_{q\in \crit\, f} \xi_q x_q^-, \quad x_{\xi}^+ := \sum_{q\in \crit\, f} \xi_q x_q^+.
\]
The assumption (1-2-3-4) imply that the Morse complexes of $h$ and $f$ are related in a very simple way:

\begin{prop}
\label{descri}
Assume that the Morse functions $f\in C^{\infty}(M)$ and $h\in C^{\infty}(S^*M)$ satisfy (1-2-3-4). Then, the following facts hold:
\begin{enumerate}
\item The boundary homomorphisms on $M_*(h)$ and $M_*(f)$ are related by the formulas
\[
\left\{ \begin{array}{ll} \partial x_q^+ = x_{\partial q}^+ & \forall q\in \crit\, f, \\
\partial x_q^- = x_{\partial q}^- & \forall q\in \crit\, f \setminus \{q_{\max}\}, \\
\partial x_{q_{\max}}^- = \chi(T^*M) \, x_{q_{\min}}^+. & \end{array} \right.
\]
\item The homomorphism $\phi: M_*(f) \rightarrow M_{*+n-1}(h)$ maps each $q\in \crit\, f$ into $x_q^+$.
\item The homomorphism $\psi: M_*(h) \rightarrow M_*(f)$ maps $x_q^+$ into $0$ and $x_q^-$ into $q$, for every $q\in \crit\, f$.
\end{enumerate}
\end{prop}

By the conclusion (a) above, the subspace of $M_*(h)$ spanned by the $x_q^+$'s, for $q\in \crit\, f$, is a subcomplex. Instead, the subspace of $M_*(h)$ spanned by the $x_q^-$'s is a subcomplex if and only if the Euler number $\chi(T^*M)$ is zero.
By the conclusions (b) and (c) of Proposition \ref{descri}, the short sequence of chain maps
\begin{equation}
\label{short}
0 \rightarrow M_*(f) \stackrel{\phi}{\longrightarrow} M_{*+n-1}(h) \stackrel{\psi}{\longrightarrow} M_{*+n-1} (f) \rightarrow 0
\end{equation}
is exact. Therefore, it induces the long exact sequence 
\begin{equation}
\label{gysin2}
\dots \rightarrow  HM_*(f) \stackrel{\phi_*}{\longrightarrow} HM_{*+n-1}(h) \stackrel{\psi_*}{\longrightarrow} HM_{*+n-1} (f) \stackrel{\partial*}{\longrightarrow} HM_{*-1}(f) \rightarrow ...
\end{equation}
in homology. By the conclusion (a) of Proposition \ref{descri}, the connecting homomorphism $\partial_*$, which can be non-zero only between $HM_n(f)$ and $HM_0(f)$, is the homomorphism which maps $[q_{\max}]$ into $\chi(T^*M) [q_{\min}]$. We conclude that the exact sequence (\ref{gysin2}) corresponds, via the isomorphisms between Morse homology and singular homology, to the Gysin sequence (\ref{gysin}). 

\medskip

\begin{proof}[of Proposition \ref{descri}]
Let $q,q'$ be critical points of $f$, and let $x$ be a point in the intersection 
\begin{equation}
\label{ab0}
\pi^{-1}\bigl(W^u(q;-\nabla f)\bigr) \cap W^s(x_{q'}^-;-\nabla h).
\end{equation}
Then we have
\begin{equation}
\label{ab1}
\begin{split}
h(x) & \leq f(\pi(x)) + \frac{1}{2} \leq f(q) + \frac{1}{2}, \\
h(x) & \geq h(x_{q'}^-) = f(q'). \end{split}
\end{equation}
If moreover
\begin{equation}
\label{ab2}
\ind(x_{q'}^-;h) = \ind(q';f) = \ind(q;f) + n-1.
\end{equation}
then either $q=q_{\min}$ and $\ind(q';f)=n-1$, or $\ind(q;f)=1$ and $q'=q_{\max}$. In the first case, the inequalities (\ref{ab1}) imply
\[
h(x) \leq \frac{1}{2}, \quad h(x) \geq n-1.
\]
In the second case, they imply 
\[
h(x) \leq \frac{3}{2}, \quad h(x) \geq n.
\]
Since $n\geq 2$, in both cases we obtain a contradiction. This shows that when the index relation (\ref{ab2}) holds, the set (\ref{ab0}) is empty. In particular, $n_{\phi}(q,x_{q'}^-) = 0$. 

Now we assume that $x$ belongs to the set
\[
\pi^{-1}\bigl(W^u(q;-\nabla f)\bigr) \cap W^s(x_{q'}^+;-\nabla h),
\]
with
\[
\ind(x_{q'}^+;h) - n + 1 = \ind(q';f) = \ind(q;f) = k.
\]
Then the fact that $x$ belongs to $\pi^{-1}(W^u(q;-\nabla f))$ implies
\begin{equation}
\label{ab5}
h(x) \leq f(\pi(x)) + \frac{1}{2} \leq f(q) + \frac{1}{2} = k + \frac{1}{2},
\end{equation}
while the fact that $x$ belongs to $W^s(x_{q'}^+;-\nabla h)$ implies
\begin{equation}
\label{ab6}
h(x) \geq h(x_{q'}^+) = f(q') + \frac{1}{2} = k + \frac{1}{2}.
\end{equation}
Therefore, equalities must hold everywhere in (\ref{ab5}) and (\ref{ab6}), so $q=\pi(x)$ and $x=x_{q'}^+$, hence $q'=q$ and $x=x_q^+$. We deduce that $n_{\phi}(q,x_{q'}^+)$ is zero when $q'\neq q$, and it is one when $q'=q$. 
We conclude that the statement (b) holds. Since $\phi$ is a chain map, also the first identity in the statement (a) holds.

Now let $x$ be a point in the intersection
\begin{equation}
\label{ab4}
W^u(x_q^+;-\nabla h) \cap \pi^{-1} \bigl(W^s(q';-\nabla f)\bigr).
\end{equation}
Then
\begin{equation}
\begin{split}
\label{ab7}
h(x) & \leq h(x_q^+) = f(q) + \frac{1}{2}, \\
h(x) & \geq f(\pi(x)) \geq f(q').
\end{split}
\end{equation}
If moreover,
\[
\ind(x_q^+;h) = \ind(q;f) + n-1 = \ind(q';f),
\]
then either $q=q_{\min}$ and $\ind (q';f)=n-1$, or $\ind(q;f) = 1$ and $q'=q_{\max}$. In the first case, the inequalities (\ref{ab7}) become
\[
h(x) \leq \frac{1}{2}, \quad h(x) \geq n-1,
\]
and in the second case
\[
h(x) \leq \frac{3}{2}, \quad h(x) \geq n.
\]
In both cases, we conclude that the set (\ref{ab4}) is empty, and hence $n_{\psi}(x_q^+,q')=0$. 

Finally, we assume that $x$ belongs to the set
\[
W^u(x_q^-;-\nabla h) \cap \pi^{-1} \bigl(W^s(q';-\nabla f)\bigr),
\]
with 
\[
\ind(x_q^-;h) = \ind(q;f) = \ind(q';f) = k.
\]
Then
\[
h(x)\leq h(x_q^-) = f(q) = k, \quad h(x) \geq f(\pi(x)) \geq f(q') = k,
\]
from which we deduce that $x=x_q^-$, $q'=\pi(x)$, so $q'=q$,
and hence $n_{\psi}(x_q^-,q')=0$ when $q'\neq q$, while $n_{\psi}(x_q^-,q)=1$. This proves the statement (c).  

Let $P^-: M_*(h) \rightarrow M_*(h)$ be the projector onto the subspace spanned by the $x_q^-$'s, along the subspace spanned by the $x_q^+$'s. Since $\psi$ is a chain map, the statement (c) implies that
\begin{equation}
\label{ab9}
P^- \partial x_q^ - = x_{\partial q}^-, \quad \forall q\in \crit\, f.
\end{equation}
Since
\begin{equation}
\label{ab8}
\ind (x_q^-;h) = \ind(q;f) \leq n, \quad \ind (x_q^+;h) = \ind(q;f) + n-1 \geq n-1,
\end{equation}
the boundary of every $x_q^-$ with $q\neq q_{\max}$ is contained in the image of $P^-$, so (\ref{ab9}) implies the second identity of the statement (a). Since $q_{\max}$ is a cycle in $M_*(f)$ (because $HM_n(f)\cong H_n(M)\neq 0$), (\ref{ab9}) and (\ref{ab8}) imply that
\[
\partial x_{q_{\max}}^- = \delta \, x_{q_{\min}}^+,
\]
for some $\delta\in \Z_2$ to be determined. Since we are working with $\Z_2$ coefficients, a fast way to check that $\delta=\chi(T^*M)$ is the following. By (b) and (c) the short sequence of chain maps (\ref{short}) is exact, so it produces the long exact sequence (\ref{gysin2}) in homology, where the connecting homomorphism $\partial_*$ is given by
\[
\partial_* : HM_n(f) \rightarrow HM_0(f) , \quad [q_{\max}] \mapsto \delta\, [q_{\min}].
\]
Then the exactness of the sequence (\ref{gysin2}) implies that
\[
HM_{n-1}(h) = \delta \, \Z_2,
\]
but since $HM_{n-1}(h) = H_{n-1}(S^*M)$, (\ref{explicit}) implies that $\delta=\chi(T^*M)$. This concludes the proof of the statement (a) and of the proposition.
\end{proof} \qed

\begin{rem}
The results of this appendix can be easily extended to arbitrary sphere bundles over closed manifolds, and to integer coefficients, by assuming orientability. In this case, the fact that the connecting homomorphism in the long exact sequence (\ref{gysin2}) is the cap product by the Euler class of the sphere bundle can be proved by looking at the cellular filtrations produced by the negative gradient flows of $f$ and $h$.
\end{rem}

\providecommand{\bysame}{\leavevmode\hbox to3em{\hrulefill}\thinspace}
\providecommand{\MR}{\relax\ifhmode\unskip\space\fi MR }
\providecommand{\MRhref}[2]{%
  \href{http://www.ams.org/mathscinet-getitem?mr=#1}{#2}
}
\providecommand{\href}[2]{#2}


\begin{thebibliography}{CIPP00}


\bibitem[AM06]{ama06m}
A.~Abbondandolo and P.~Majer, \emph{Lectures on the {M}orse complex for
  infinite dimensional manifolds}, Morse theoretic methods in nonlinear
  analysis and in symplectic topology (Montreal) (P.~Biran, O.~Cornea, and
  F.~Lalonde, eds.), Springer, 2006, pp.~1--74.

\bibitem[APS08]{aps08}
A.~Abbondandolo, A.~Portaluri, and M.~Schwarz, \emph{The homology of path
  spaces and {F}loer homology with conormal boundary conditions}, J. fixed
  point theory appl. \textbf{4} (2008), 263--293.

\bibitem[AS06]{as06}
A.~Abbondandolo and M.~Schwarz, \emph{{On the Floer homology of cotangent
  bundles}}, Comm. Pure Appl. Math. \textbf{59} (2006), 254--316.

\bibitem[AS08]{as08}
A.~Abbondandolo and M.~Schwarz, \emph{Floer homology of cotangent bundles and the loop product},
  Preprint~41, Max-Planck-Institut f\"ur Mathematik in den Naturwissenschaften,
  Leipzig, 2008, {\tt arXiv:0810.1995 [math.SG]}.

\bibitem[AS09]{as08b}
A.~Abbondandolo and M.~Schwarz, \emph{A smooth pseudo-gradient for the {L}agrangian action
  functional}, Adv. Nonlinear Stud. \textbf{9} (2009), 597--623.

\bibitem[AF08a]{af08a}
P.~Albers and U.~Frauenfelder, \emph{Leaf-wise intersections and {R}abinowitz {F}loer homology}, {\tt
  arXiv:0810.3845 [math.SG]}, 2008.

\bibitem[AF08b]{af08b}
P.~Albers and U.~Frauenfelder, \emph{Infinitely many leaf-wise intersection
  points on cotangent bundles}, {\tt arXiv:0812.4426 [math.SG]}, 2008.

\bibitem[Bah89]{bah89}
A.~Bahri, \emph{Critical points at infinity in some variational problems},
  Pitman Research Notes in Mathematics Series, vol. 182, Longman Scientific \&
  Technical, Harlow, 1989.

\bibitem[Ben86]{ben86}
V.~Benci, \emph{Periodic solutions of {L}agrangian systems on a compact
  manifold}, J. Diff. Eq. \textbf{63} (1986), 135--161.

\bibitem[CFH95]{cfh95}
K.~Cieliebak, A.~Floer, and H.~Hofer, \emph{Symplectic homology. {II}. {A}
  general construction}, Math. Z. \textbf{218} (1995), no.~1, 103--122.

\bibitem[CF09a]{cf09}
K.~Cieliebak and U.~Frauenfelder, \emph{A {F}loer homology for symplectic
  embeddings}, Pacific Journal of Mathematics \textbf{239} (2009), 251--316.

\bibitem[CF09b]{cf09b}
K.~Cieliebak and U.~Frauenfelder, \emph{Morse homology on noncompact manifolds}, {\tt arXiv:0911.1805
  [math.SG]}, 2009.

\bibitem[CFO09]{cfo09}
K.~Cieliebak, U.~Frauenfelder, and A.~Oancea, \emph{Rabinowitz {F}loer homology
  and symplectic homology}, Ann. Sci. de l'ENS (to appear), {\tt
  arXiv:0903.0768 [math.SG]}, 2009.

\bibitem[CFP09]{cfp09}
K.~Cieliebak, U.~Frauenfelder, and G.~P. Paternain, \emph{{Symplectic topology
  of Ma\~n\'e's critical values}}, {\tt arXiv:0903.0700 [math.SG]}, 2009.

\bibitem[Con06]{con06}
G.~Contreras, \emph{The {P}alais-{S}male condition on contact type energy
  levels for convex {L}agrangian systems}, Calc. Var. Partial Differential
  Equations \textbf{27} (2006), 321--395.

\bibitem[CIPP98]{cipp98}
G.~Contreras, R.~Iturriaga, G.~P. Paternain, and M.~Paternain, \emph{Lagrangian
  graphs, minimizing measures and {M}a\~n\'e's critical values}, Geom. Funct.
  Anal. \textbf{8} (1998), no.~5, 788--809.

\bibitem[CIPP00]{cipp00}
G.~Contreras, R.~Iturriaga, G.~P. Paternain, and M.~Paternain, \emph{The {P}alais-{S}male condition and {M}a\~n\'e's critical
  values}, Ann. Henri Poincar\'e \textbf{1} (2000), 655--684.

\bibitem[CS09]{cs08}
R.~L. Cohen and M.~Schwarz, \emph{A {M}orse theoretic description of string
  topology}, New perspectives and challenges in symplectic field theory
  (L.~Polterovich, M.~Abreu, and F.~Lalonde, eds.), CRM Proc. Lecture Notes,
  vol.~49, Stanford 2007, Amer. Math. Soc., 2009, (to appear) {\tt
  arXiv:0809.0868 [math.GT]}.

\bibitem[Dol80]{dol80}
A.~Dold, \emph{Lectures on algebraic topology}, Springer, Berlin, 1980.

\bibitem[FH94]{fh94}
A.~Floer and H.~Hofer, \emph{{Symplectic homology. I. Open sets in
  $\mathbb{C}\sp n$}}, Math. Z. \textbf{215} (1994), 37--88.

\bibitem[Fra04]{fra04}
U.~Frauenfelder, \emph{The {A}rnold-{G}ivental conjecture and moment {F}loer
  homology}, Int. Math. Res. Not. \textbf{42} (2004), 2179--2269.

\bibitem[GM03]{gm03}
S.~I. Gelfand and Y.~I. Manin, \emph{Methods of homological algebra}, second
  ed., Springer Monographs in Mathematics, Springer-Verlag, Berlin, 2003.

\bibitem[GT83]{gt83}
D.~Gilbarg and N.~S. Trudinger, \emph{Elliptic partial differential equations
  of second order}, Springer, New York, 1983.

\bibitem[Kli82]{kli82}
W.~Klingenberg, \emph{Riemannian geometry}, Walter de Gruyter \& Co., Berlin,
  1982.

\bibitem[Mer09]{mer09}
W.~Merry, \emph{On the {R}abinowitz {F}loer homology of twisted cotangent bundles}, in preparation, 2009.

\bibitem[Mil65]{mil65b}
J.~Milnor, \emph{Lectures on the $h$-cobordism theorem}, Princeton University
  Press, Princeton, NJ, 1965.

\bibitem[Oan04]{oan04}
A.~Oancea, \emph{A survey of {F}loer homology for manifolds with contact type
  boundary or symplectic homology}, Symplectic geometry and {F}loer homology.
  {A} survey of the {F}loer homology for manifolds with contact type boundary
  or symplectic homology, Ensaios Mat., vol.~7, Soc. Brasil. Mat., Rio de
  Janeiro, 2004, pp.~51--91.

\bibitem[Pal63]{pal63}
R.~S. Palais, \emph{Morse theory on {H}ilbert manifolds}, Topology \textbf{2}
  (1963), 299--340.

\bibitem[PP97]{pp97}
G.~P. Paternain and M.~Paternain, \emph{Critical values of autonomous
  {L}agrangian systems}, Comment. Math. Helv. \textbf{72} (1997), no.~3,
  481--499.

\bibitem[Rab78]{rab78}
P.~H. Rabinowitz, \emph{Periodic solutions of {H}amiltonian systems}, Comm.
  Pure Appl. Math. \textbf{31} (1978), 157--184.

\bibitem[RS93]{rs93}
J.~Robbin and D.~Salamon, \emph{Maslov index theory for paths}, Topology
  \textbf{32} (1993), 827--844.

\bibitem[Sch06]{sch06}
F.~Schlenk, \emph{Applications of {H}ofer's geometry to {H}amiltonian
  dynamics}, Comment. Math. Helv. \textbf{81} (2006), 105--121.

\bibitem[Sei08]{sei06b}
P.~Seidel, \emph{A biased view of symplectic cohomology}, Current Developments
  in Mathematics, 2006, International Press, 2008, pp.~211--253.

\bibitem[Vit99]{vit99}
C.~Viterbo, \emph{{Functors and computations in Floer homology with
  applications, I}}, Geom. Funct. Anal. \textbf{9} (1999), 985--1033.

\bibitem[Vit03]{vit03}
C.~Viterbo, \emph{Functors and computations in {F}loer homology with applications,
  {II}}, preprint (first version 1996, revised in 2003).

\bibitem[Wei94]{wei94}
C.~A. Weibel, \emph{An introduction to homological algebra}, Cambridge Studies
  in Advanced Mathematics, vol.~38, Cambridge University Press, Cambridge,
  1994.

\end{thebibliography}
\end{document}